\documentclass[11pt]{article}
\usepackage{amssymb,amsmath}
\usepackage{mathrsfs}
\usepackage{latexsym}
\usepackage{color}
\usepackage{amssymb}
\usepackage{amsmath}
\usepackage{amsthm}
\usepackage{hyperref}
\usepackage{mathtools}
\usepackage{amsmath,amsfonts,amsthm,amssymb,amscd,color,xcolor,mathrsfs,enumitem}

\colorlet{darkblue}{blue!50!black}

\hypersetup{
    colorlinks,%
    citecolor=darkblue,%
    filecolor=red,%
    linkcolor=darkblue,%
    urlcolor=magenta,%
    pdfnewwindow=true,%
    pdfstartview={FitH}
}

\newcommand{\p}{\partial}
\newcommand{\e}{\varepsilon}

\newcommand{\R}{{\mathbb R}}
\newcommand{\Z}{{\mathbb Z}}
\newcommand{\IP}{{\mathbb P}}

\newcommand{\E}{{\mathbb E}}

\newcommand{\N}{{\mathbb N}}

\newcommand{\om}{\omega}
\newcommand{\eps}{\varepsilon}
\newcommand{\vp}{\varphi}
\newcommand{\vt}{\vartheta}

\newcommand{\MMM}{\mathfrak{M}}

\newcommand{\XXXX}{{\mathfrak X}}

\newcommand{\bKK}{{\boldsymbol{\KK}}}

\newcommand{\U}{{\boldsymbol{ {U}}}}

\newcommand{\eeta}{{\boldsymbol{{\eta}}}}

\newcommand{\nnu}{{\boldsymbol{{\nu}}}}

\newcommand{\oomega}{{\boldsymbol\omega}}

\newcommand{\xxi}{{\boldsymbol\xi}}

\newcommand{\AB}{\overline{B_G(1)}}

\newcommand{\BB}{{\cal B}}

\newcommand{\DD}{{\cal D}}

\newcommand{\FF}{{\cal F}}

\newcommand{\HH}{{\cal H}}
\newcommand{\KK}{{\cal K}}
\newcommand{\LL}{{\cal L}}

\newcommand{\PPP}{{\mathscr P}}
\newcommand{\PP}{{\mathcal P}}

\newcommand{\RR}{{\cal R}}

\newcommand{\dd}{{\textup d}}

\newcommand{\PPPP}{{\mathfrak P}}

\newcommand{\SSS}{{\mathscr S}}

\newcommand{\EEE}{{\boldsymbol{E}}}

\newcommand{\be}{\begin{equation}}
\newcommand{\ee}{\end{equation}}

\newcommand{\supp}{\mathop{\rm supp}\nolimits}

\newcommand{\Lip}{\mathop{\rm Lip}\nolimits}

 \newcommand{\strela}{\rightharpoonup}

\theoremstyle{plain}

\newtheorem{theorem}{Theorem}[section]

\newtheorem{lemma}[theorem]{Lemma}
\newtheorem{proposition}[theorem]{Proposition}
\newtheorem{corollary}[theorem]{Corollary}
\theoremstyle{definition}
\newtheorem{definition}[theorem]{Definition}

\theoremstyle{remark}
\newtheorem{remark}[theorem]{Remark}

\numberwithin{equation}{section}

\allowdisplaybreaks

\catcode`!=11
\let\!int\int \def\int{\displaystyle\!int}
\let\!lim\lim \def\lim{\displaystyle\!lim}
\let\!sum\sum \def\sum{\displaystyle\!sum}
\let\!sup\sup \def\sup{\displaystyle\!sup}
\let\!inf\inf \def\inf{\displaystyle\!inf}
\let\!cap\cap \def\cap{\displaystyle\!cap}
\let\!max\max \def\max{\displaystyle\!max}
\let\!min\min \def\min{\displaystyle\!min}
\catcode`!=12

\catcode`@=11
\@addtoreset{equation}{section}
\catcode`@=12

\begin{document}
\title{Long-time behaviour of  dynamical systems driven by bounded  mixing noises}
\author{Peng Gao\footnote{School of Mathematics and Statistics, Center for Mathematics and Interdisciplinary Sciences,
Northeast Normal University, Changchun, China. Email: gaopengjilindaxue@126.com},
~~Sergei Kuksin\footnote{Universit\'e Paris Cit\'e and Sorbonne Universit\'e, CNRS,
IMJ-PRG, Paris, France;  Steklov Mathematical Institute of RAS, Peoples' Friendship University of Russia \& National Research
University Higher School of Economics,  Moscow, Russia.
Email: sergei.kuksin@imj-prg.fr}}

\date{\today}
\maketitle

\begin{abstract}
We study the mixing properties of  discrete-time and continuous-time dissipative dynamical systems driven by  bounded mixing random forces. The continuous-time systems are reduced to discrete-time random dynamical systems generated by time-one maps, so that the main analysis is carried out in the discrete setting. We introduce a class of
mixing random forcings whose regular conditional distributions with respect to the past
 satisfy natural regularity, recurrence, and non-degeneracy assumptions, extending the framework previously developed for more restrictive classes of processes in a paper by Kuksin-Shirikyan in GAFA (2025).
 Under a linearised controllability assumptions on the system, we prove exponential mixing in the total variation metric for
  finite-dimensional phase spaces. We then establish an infinite-dimensional counterpart yielding exponential mixing in the
  dual-Lipschitz metric under suitable amendments of restrictions on the system and
 the random forcing.
  Our approach is based on lifting the dynamics to an appropriate Markov process on an infinite-dimensional history space and applying a Doeblin coupling argument through the method of Kantorovich functional. As applications,
  we derive exponential mixing for a broad class of ordinary  differential equations driven by mixing  random processes with bounded continuous trajectories.
   As an application of our result to PDEs we discuss the randomly perturbed primitive
  equations of atmospheric dynamics.
  \end{abstract}


\tableofcontents
 \setcounter{section}{-1}

\section{Introduction}

This work continues the study of discrete-time and continuos-time dynamical systems driven by bounded and mixing random forces,
initiated in \cite{KS-2025}.  Specially, let  $H$ and $E$ be finite- or infinite-dimensional Hilbert spaces, serving respectively as the phase space and the space of controls. We consider the systems
\be\label{a1}
u_k =S(u_{k-1}, \eta_k), \quad k\ge1; \qquad u_0=v, \; \; u_k \in H,
\ee
and
\be\label{a2}
\dot u(t) = F(u(t), \eta(t)), \quad t\ge0; \qquad u(0) =v, \;\; u(t) \in H.
\ee
Here $\{ \eta_k^\om \in E, k\in\Z\}$ and $\{ \eta^\om(t)  \in E, t\in\R\}$ are bounded random processes, the map $S: H\times E\to H$ is
$C^2$-smooth, and the non-autonomous vector-field $F$ on $H$ is such that its flow-maps are well defined and  $C^2$-smooth.

Equation~\eqref{a2} can be reduced to a system of the form \eqref{a1}. Indeed, for $k\in\Z$, let
$
I_k=[k-1,k),
$
define $\eta_k:=\eta \mid_{I_k}$, and regard $\eta_k$ as an element of a Hilbert space of curves $I_k\to E$, naturally identified with a Hilbert space $\hat E$ of curves $I:=I_1\to E$ (for example, $\hat E=L^2(I,E)$). Let
$
S_1\times\hat E\to H
$
be the map sending $(v,\eta_1)$ to $u(1)$, where $u(t)$, $t\in I$, is a solution of \eqref{a2}.
Assuming that $S_1$ is well defined and continuous, and writing $u_k=u(k)$ for $k=0,1,\dots$, we obtain
\be\label{a3}
u_k =S_1(u_{k-1}, \eta_k), \quad k\ge1; \qquad u_0=v.
\ee
Thus, the discrete-time system \eqref{a3} describes the evolution of solutions to \eqref{a2} at integer times. Accordingly, the main part of the paper (Sections~1--5) is devoted to discrete-time systems of the form \eqref{a1}. However, the applications presented in Section~6 concern the more interesting continuous-time systems \eqref{a2}, studied through their reduction to systems
 \eqref{a3}.

 As in \cite{KS-2025, KS-2026}, our objective is to prove that a broad class of systems \eqref{a1}, driven by suitable random processes $\eta$, is mixing with respect to either the total variation or the dual-Lipschitz distance on the space of probability measures on $H$ (see \eqref{t_var} and \eqref{LipD} below). More precisely, for any two initial conditions $v$ and $v'$, the corresponding solutions ${u_k}$ and ${u_k'}$ satisfy
$
\operatorname{dist}\bigl(\DD(u_k),\DD(u_k')\bigr)\to0
$
in one of these two metrics. The speed of this convergence is referred to as the {\it rate of mixing}.

Mixing in random dynamical systems (including systems \eqref{a1} and \eqref{a2}) is a classical topic, motivated by numerous applications in science and engineering. Mixing for systems with Markovian
dynamics which are strongly Feller or finite-dimensional has been intensively studied
starting the works of Kolmogorov and Doeblin in the 1930s. Now this is a well developed topic;
 see the books \cite{Kha, MT, Bor}.
The case of infinite-dimensional Markov systems whose dynamics is not strongly Feller has attracted considerable attention since the beginning of this century and  also is relatively well understood; see \cite{KS-book}. These results apply to systems \eqref{a1} in which the $\eta_k$ are i.i.d. random variables, as well as to equations \eqref{a2} that are stochastic ODEs or PDEs.

In contrast, mixing for systems \eqref{a1} and \eqref{a2} with non-Markovian dynamics has been poorly understood,
with only a few isolated results available for special classes of systems (see, e.g., \cite{EPR99, JP97, Hai05}).
The situation changed with the work \cite{KS-2025}, which introduced an approach to systems \eqref{a1} and \eqref{a2}
driven by bounded, stationary, and mixing random processes; such equations generally define non-Markovian dynamics.
We  note similarities between the results and  approach of \cite{KS-2025} and those in the theory of SRB measures
 for Anosov diffeomorphisms; see \cite[Remark~1.1.8]{KS-2025}.

The goals of this work are twofold: to present a more accessible exposition of the approach to mixing in non-Markovian systems developed in \cite{KS-2025}, and to further develop this approach so as to extend its applicability to a broader class of systems \eqref{a1} and \eqref{a2}.

Motivated by the first goal, we begin in Sections~3--4 by considering systems \eqref{a1} with finite-dimensional phase spaces~$H$, and then explain in Section~5 how the proofs can be adapted to the infinite-dimensional setting. Again, in order to simplify the presentation, we restrict ourselves to the case where the control space $E$ is Hilbert  (whereas in \cite{KS-2025} it is allowed to be a Banach space).  Our results for finite-dimensional systems are obtained under weaker assumptions than those of the main theorem in \cite{KS-2025}, which applies to both finite- and infinite-dimensional systems. They also include the result established in \cite{KS-2026} that mixing in finite-dimensional systems holds in the total variation distance.

In pursuit of the second goal, we do not assume that the process ${\eta_k}$ is stationary. Instead, we show that the approach developed in \cite{KS-2025} also applies to equations \eqref{a2} driven by bounded non-stationary processes $\eta(t)$ with continuous sample paths. (The examples considered in \cite{KS-2025} involve equations \eqref{a2} driven by processes whose sample paths belong to $L^2_{\mathrm{loc}}$.) As we show in Sections~\ref{s_6.3} and~\ref{s_6.4}, extending the theory to equations driven by random forces with continuous sample paths significantly broadens the class of equations \eqref{a2} to which our results apply.

The presentation in this paper largely follows that of \cite{KS-2025}. However, the differences are substantial, and several proofs adapted from \cite{KS-2025} have been considerably simplified. We also keep the notation deliberately close to that of \cite{KS-2025, KS-2026}.

We always assume that the random variables ${\eta_k}$ in equation \eqref{a1}
are supported on a compact set $\KK \subset E$, that is,
$
\eta_k^\omega \in \KK
$
for all $k$ and $\omega$. We also assume that equation~\eqref{a1} is dissipative,
in the sense that there exists a compact set $X\subset H$ such that
$
S(X\times \KK) \subset X.
$
Thus, \eqref{a1} defines a dynamical system on $X$. Without loss of generality, we assume that both $\KK$ and $X$ contain the origins of the corresponding spaces. In addition, we assume that the unforced system, corresponding to $\eta_k\equiv 0$, is strictly dissipative, in the sense that all its solutions converge to the origin at a rate depending only on $|v|_H$.
The above assumptions on system~\eqref{a1} and the process $\eta$ will be in force throughout the Introduction. In addition, we impose one further assumption on the system and three additional assumptions on the process. For the sake of readability, we state these assumptions in a slightly informal form in the Introduction.
\smallskip

\noindent {\bf (S)} There exists a finite-dimensional subspace $F\subset E$ such that, for every
$(u,\eta)\in X\times\KK$, the linear operator
$
D_\eta S(u,\eta): F \to H
$
is surjective.
\smallskip

To state the assumptions on the process $\eta$, we first introduce some notation. For any integer $l$, let $\eeta_l$ denote the past of $\eta$ up to time $l$:
$
\eeta_l=(\dots,\eta_{l-1},\eta_l),
$
which we regard as an element of the space
\be\label{sequences}
\bKK := \KK^{\Z_-} = \{ (\dots, \xi_{-1}, \xi_0) \mid \xi_k \in \KK\} \subset E^{\Z_-}.
\ee
The space~$E^{\Z_-}$ is provided  with the Tikhonov topology, and we metrise the topology's  restriction to  $\boldsymbol{\KK}$ by the distance
\begin{equation}\label{distance-EE}
\dd(\xxi,\xxi')=\sum_{k=-\infty}^0 \alpha^{|k|}\|\xi_{k}-\xi_{k}'\|_E, \quad
\xxi=(\xi_k\in \KK,k\in\Z_-),\quad \xxi'=(\xi_k'\in \KK,k\in\Z_-),
\end{equation}
where $\alpha\in(0,1)$ is a fixed number. Spaces $E^{\Z_-}$ and $\bKK$  both are Polish (i.e., complete and separable).
Our assumptions concerning  the process $\eta$  are made in terms of its regular conditional
distributions with respect to the past,
$$
Q_l(\xxi;\cdot):=\IP\{\eta_{l+1}\in\cdot\,|\,\eeta_l=\xxi\}\in \PP(\KK)
$$
(e.g., see \cite[Section 10.2]{dudley2002}). Our first assumption on the process $\eta$ is analogous to the strong Feller property for Markov processes:
\smallskip

\noindent $\pmb{(\eta1)}$ For any integer $l$, the mapping
$$
Q_l:(\bKK,d)\to(\PP(\KK),|\cdot|_{\mathrm{var}}),\qquad
\xxi\mapsto Q_l(\xxi;\cdot),
$$
is Lipschitz continuous.
\smallskip

Assumption $\pmb{(\eta1)}$ implies that the process $\{\eta_k\}$ forgets its past exponentially fast. Indeed, by
the definition \eqref{distance-EE} of the metric on $\bKK$, it follows that the regular
conditional distribution of $\eta_{l+1}$
given $\eeta_l$, regarded as a function of a past variable $\eta_{l-k}$, is Lipschitz continuous with a Lipschitz constant
of order $\alpha^k$. Moreover, as shown in Corollary~C, assumption $\pmb{(\eta1)}$, together with assumption
$\pmb{(\eta2)}$ stated below, implies that the process ${\eta_k}$ is exponentially feeble mixing, i.e. mixing with a smaller set of
observables (see \cite[Section~4.2]{KS-2025}). Although somewhat different from the classical definition of mixing (e.g. as in \cite{IL}), it still implies
SLLN and CLT for the corresponding processes, see \cite{DDL} and \cite[Section~4.1]{KS-book}.

We also assume that the process satisfies the following recurrence-to-the-origin condition:
\smallskip

\noindent { $\pmb{(\eta2)}$} For any integer $l$, any $\delta>0$, any $n\in\N$, and any fixed past $\eeta_l=\xxi$, there exists a time, depending only on $\delta$ and $n$, such that then, with positive probability the process $\eta_k$ enters the $\delta$-neighbourhood of the origin in $E$ and remains there for the next $n$ time steps.
\smallskip

A precise formulation of this assumption is given below in condition (RZ). To state the final restriction, let
$F^\perp=\{y_F^\perp\}$ denote the orthogonal complement of the subspace $F=\{y_F\}$ appearing in
Assumption~({\bf S)}.
\smallskip

\noindent  ($\pmb{\eta3}$)  For each $\xxi \in \bKK$ and $l\in\Z$, the conditional distributions
$
Q_{l F} (\xxi, y_F^\perp; \cdot) \in \PP(F)
$
 of measure $Q_l(\xxi; \cdot)$ with respect to the projection $ E\to F^\perp$ admit densities with respect to the
  Lebesgue measure $l_F$ on $F$:
 $$
Q_{l F}(\xxi,y_F^\perp;\dd y_F)=\rho_{l F}(\xxi,y_F^\perp,y_F)\ell_F(\dd y_F),
$$
where  $\rho_{l F}$'s are uniformly Lipschitz functions of their arguments.\footnote{That is, their Lipschitz
constants are bounded by an $l$-independent number.}
\medskip

In Section \ref{A1} we show that assumptions $(\pmb{\eta1}) - (\pmb{\eta3})$ make a mild restriction, and plenty
 of processes $\{\eta_k\}$ meet them.
Our main result for the finite-dimensional system \eqref{a1} is the following theorem. Here and throughout the paper, we denote by $\{ u_k(v), k\ge0\}$, a solution of \eqref{a1}, corresponding to an  initial condition $u_0=v$.
\smallskip

\noindent {\bf Theorem A.} Suppose that $\dim H<\infty$ and assumptions {\bf(S)} and $\pmb{(\eta1)}$--$\pmb{(\eta3)}$ hold. Then there exist positive constants $\gamma$ and $C$ such that, for any $v,v'\in X$,
\begin{equation}\label{main1}
\| \DD u_k(v) - \DD u_k(v')\|_{var} \le C e^{-\gamma k}, \qquad \forall\, k\ge0.
\end{equation}
Moreover, there exists a process ${\hat u_k}, {k\in\Z}$ satisfying \eqref{a1} in the sense of distributions:
$
\DD(\hat u_k)=S_*\bigl(\DD(\hat u_{k-1},\hat\eta_k)\bigr)$ for all $k\in\Z,
$
where $\hat\eta$ is some  process having the same distribution as $\eta$. This process is such that
for any $v\in X$, estimate \eqref{main1} remains valid if $u_k(v')$ is replaced by $\hat u_k$.
These properties determine the  distribution of  $\hat u$ uniquely.
\smallskip

If the process ${\eta_k}$ is stationary, then the process ${\hat u_k}$ is stationary as well. Consequently,
$
\DD(\hat u_k)\equiv\mu,
$
where $\mu$ is a stationary measure of equation \eqref{a1}. This result is established in \cite{KS-2025} under slightly
stronger assumptions on the system. Although these differences are minor, they are important for certain applications of the theorem.
Often, systems of the form \eqref{a1} that are dissipative in the above sense have the following additional property: for any $v\in H$ and any sequence ${\eta_k} \in \KK$, the corresponding solution of \eqref{a1} enters the set $X$ after at most $k_0(|v|_H)$ steps and remains there thereafter. In this case, a slightly more general version of Theorem~A, stated and proved in the main text and allowing for random initial data $v$, implies that estimate \eqref{main1} holds for all $v,v'\in H$, with the constant $C$ replaced by
$
C\bigl(\max(|v|_H,|v'|_H)\bigr).
$
\medskip

For infinite-dimensional systems \eqref{a1}, we retain assumptions $\pmb{(\eta1)}$ and $\pmb{(\eta2)}$, while replacing the linearised controllability assumption \textbf{(S)} and the decomposability and regularity assumption $\pmb{(\eta3)}$ by suitable modifications. We begin with the following version of $\pmb{(\eta3)}$:
\smallskip

\noindent {$\pmb{(\eta3')}$} There exists an increasing sequence of finite-dimensional subspaces
$
F_n\varsubsetneq F_{n+1}\subset E
$
such that, for every $n$, assumption $\pmb{(\eta3)}$ holds with $F=F_n$.
\smallskip

To formulate an infinite-dimensional analogue of assumption (S), we recall the following definition,
standard in the theory of  parabolic PDEs.
\begin{definition}\label{d_determining}
A closed subspace $G\subset H$ is called {\it determining} for system \eqref{RDSN} if there exists a constant $\varkappa\in(0,1)$ such that
\begin{equation}\label{determining}
\bigl|(I-{\mathsf P}_G)D_uS(u,\eta)\bigr|_{\LL(H)}
\le \varkappa
\quad\text{for all } u\in X,\ \eta\in \KK,
\end{equation}
where ${\mathsf P}_G\to H$ denotes the orthogonal projection onto~$G$.
\end{definition}

This property  usually holds with a finite-dimensional space $G$
if \eqref{a2} is a quasilinear parabolic equation, defined on a bounded space-domain.
See in the main text Remark~\ref{r_determ}.
This leads to the following infinite-dimensional version of assumption (S).
\smallskip

\noindent {\bf (S$'$)} There exists a finite-dimensional determining subspace $G\subset H$ such that, for every $(u,\eta)$ belonging to a neighbourhood of $X\times\KK$ in $H\times E$, the closure of the
image of the space $\bigcup_n F_n$ under the linear map $D_\eta S(u,\eta)$ contains~$G$. Here $F_n$'s are
the spaces in assumption  {$\pmb{(\eta3')}$}.
\smallskip

In Section~5 we prove the following result.
\smallskip

\noindent {\bf Theorem B.} Under assumptions {\bf (S$'$)}, $\pmb{(\eta1)}$, $\pmb{(\eta2)}$, and $\pmb{(\eta3')}$, the conclusions of Theorem~A remain valid for systems \eqref{a1} with $\dim H=\infty$, provided that the total variation distance in \eqref{main1} is replaced by the dual-Lipschitz distance $|\cdot|_{L,H}^*$.
\smallskip

Note that, unlike in \cite{KS-2025}, we do not assume that the space $\bigcup_n F_n$ is dense in $E$.
As in \cite{KS-2025}, the proofs of both theorems are based on lifting system \eqref{a1} to a Markov system on the space
$\XXXX = X\times\bKK$. The latter is not strongly Feller, and $\XXXX$ is a complicated infinite-dimensional space. Nevertheless, an application of the Doeblin coupling, formulated in terms of the Kantorovich functional as in \cite{kuksin2006, KS-book}, enables us to prove that the lifted system is exponentially mixing, which in turn implies the assertions of the theorems.

In Section \ref{s_3.2} we make the following curious observation. Let us apply our result on the
exponential mixing for the lifted Markov system
on $\XXXX$ in the case when the original system  \eqref{a1} is the trial one
 with $H=\R^0=\{0\}$ and $S(0,\eta)=0$. Then
we see that then  assumptions $\pmb{(\eta3)}$  and $\pmb{(S)}$
with $F=\{0\}$ becomes vacuous,
whereas assumptions $\pmb{(\eta1)}$ and $\pmb{(\eta2)}$ reduce to conditions on the
 transition probabilities $Q_l$. This yields the following corollary to Theorem~A, related to \cite[Case~2]{dobrushin-1970}:
\smallskip

\noindent {\bf Corollary C.} If a family of transition probabilities $\{Q_l\,\}_{l\in\Z}$ from $\bKK$ to $\KK$ satisfies assumptions $\pmb{(\eta1)}$ and $\pmb{(\eta2)}$, then there exists a random process ${\eta_l}$, unique in distribution, whose regular
conditional distributions (see \eqref{cond}) are given by $Q_l$'s. Moreover, this process is exponentially mixing in the sense that for any integer $\tau$  and any vector $\xxi\in\bKK$,
$$
\|
\DD (\eta_{\tau+k} \mid \eeta_\tau =\xxi) - \DD( \eta_{\tau+k}) \|^*_L \le C e^{-\gamma k} \quad \forall k\ge0,
$$
where $\gamma, C>0$.

See Section \ref{s_3.2} for further properties of this process $\eta$.
In the final  Section \ref{s_6}, we discuss applications of Theorems A and B to systems of the form \eqref{a2}.
These applications concern systems driven by processes $\eta(t)$ whose trajectories belong either to space
 $L^2_{\mathrm{loc}}(\R,E)$ or to $H^1_{\mathrm{loc}}(\R,E)$ (and are therefore continuous).
 The applications to systems driven by processes with $L^2_{\mathrm{loc}}(\R,E)$ trajectories are similar
  to those presented in \cite{KS-2025}, where in addition  the processes $\eta$
   are assumed to be stationary or statistically periodic.
   Accordingly, here we restrict our attention to systems driven by processes with $H^1_{\mathrm{loc}}(\R,E)$ trajectories.

 Firstly, in Section~\ref{s_6.2} we discuss random processes $\eta(t)$ for which the associated discrete-time process
 $\{\eta_k =\eta\!\mid_{I_k}, k\in\Z\}$ (see \eqref{a3})  ``almost satisfy" assumptions $\pmb{(\eta1)}-\pmb{(\eta3)}$.
 For that end  the following   construction is crucial. The space $H^1_0 =H^1_0(I)$, $I=[0,1)$, is
 a subspace of $H^1 =H^1(I)$ of codimension two.
 Let us take two affine functions on $I$, $\phi_1(t) = t$ and $\phi_2(t)=1-t$. Then
 $
 \phi_1(0) =\phi_2(1)=0$ and $\phi_1(1) =\phi_2(0)=1,
 $
 so $H^1 = \R \phi_1\oplus \R\phi_2\oplus H^1_0$ and each curve $\xi(t) \in \HH^1 =H^1(I,Y)$ may be uniquely decomposed as
 $$
 \xi(t) = \mu_1 \phi_1(t) +  \mu_2 \phi_2(t) + \xi^0(t), \quad\text{where}\quad  \mu_1= \xi(1), \; \mu_2 = \xi(0), \;\; \xi^0 \in \HH^1_0 = H^1_0(I, Y).
 $$
 Accordingly, if $\eta(t) \in Y$ is a process with $H^1_{loc}$-trajectories and $\nu_k= \eta(k)$, $k\in\Z$, then
 \be\label{H^1}
 \HH^1\ni \eta_k = \nu_{k-1} \phi_1 + \nu_k\phi_2 + \eta_k^0, \qquad \eta_k^0 \in \HH^1_0, \quad k\in\Z.
 \ee
 Except in degenerate situations, the process $\{\eta_k, k\in\Z\}$ does not satisfy assumption $\pmb{(\eta1)}$ with $E=\HH^1$. Indeed, the measures $Q_l(\xxi;\cdot)$ and $Q_l(\xxi';\cdot)$ are supported on the closed subsets of $\HH^1$ consisting of curves whose values at $t=0$ are $\xxi(0)$ and $\xxi'(0)$, respectively. These two subsets are disjoint whenever $\xxi(0)\neq\xxi'(0)$, and
 then
 $
 \| Q_l(\xxi;\cdot) -  Q_l(\xxi'; \cdot)\|_{var} =1.
 $
But in  assumptions $\pmb{(\eta3')}$ and ${\bf (S')}$ we do not require that the
 space $\cup F_n$ is dense in $E$, and  assumption $\pmb{(\eta3)}$ deals only with  one finite-dimensional  space $F$.
 Owing to this, as we show in Section~\ref{s_6.2}, both of our main theorems apply to equations driven by processes $\eta(t)$ with $H^1_{\mathrm{loc}}$ trajectories, provided that assumptions $\pmb{(\eta1)}$--$\pmb{(\eta3)}$ and $\pmb{(\eta3')}$ hold for the modified process
$
\hat\eta_k =(\nu_k, \eta_k^0) \in \R \times \HH_0^1.
$
This condition is satisfied by a large class of stochastic processes $\eta(t)$.

In Section \ref{s_6.3.2} we consider ODEs driven by continuous  random forces:
\begin{equation}\label{a4}
\dot x(t) = f(x(t), \eta(t)), \qquad x(t) \in H,\; \eta(t) \in E; \qquad \dim H, \;\dim E<\infty.
\end{equation}
To apply Theorem~A, we first pass from \eqref{a4} to the corresponding discrete-time equation \eqref{a3}. There
we  represent $\eta_k$ in the form \eqref{H^1}, define $w_k=(u_k, \nu_k)$, $\hat\eta_k = (\nu_k, \eta_k^0)$
and replace \eqref{a3} by the following equivalent
system:
\be\label{a5}
w_k = \hat S(w_{k-1}, \hat\eta_k), \qquad k\ge1, \quad w_0=( v, \eta_0^0),
\ee
where
$
\hat S \big((u,\mu), (\nu, \eta^0)\big) = \big( S(u, \mu\phi_1+\nu\phi_2+\eta^0), \nu\big).
$
We then show that Theorem~A applies to system \eqref{a5}, provided that:

--equation \eqref{a4}  is dissipative and satisfies a rather mild assumption {\bf (K)}
(which holds, for example, if the equation admits a suitable Lyapunov function and all trajectories corresponding to $\eta(t)\equiv0$ converge to the origin),

--  the linear mappings $
D_\eta f(x,\zeta) :E \to H
$
are surjective, and

-- the  assumptions $\pmb{(\eta1)}-\pmb{(\eta3)}$ hold for the process $\hat\eta$.

\noindent
Application of the theorem implies that eq.~\eqref{a4}, considered at integer moments of time,  is
exponentially mixing in the total variation distance, while being considered at  reals moments $t\ge0$ it is exponentially mixing in the
dual-Lipschitz distance.\footnote{We did not put much efforts in trying to prove that
the equation is mixing in the total variation distance for $t\ge0$.
Maybe for that assumption $\pmb{(\eta3)}$ should be replaced by a bit stronger condition.}

Thus, after the preliminaries described above, Theorem~A applies to a broad class of randomly forced dissipative ODEs \eqref{a4} driven by bounded mixing random processes with continuous sample paths. Moreover, Theorem~A, together with the reduction \eqref{a5}, provides a general framework for studying other classes of equations driven by continuous mixing random forces. In particular, in a work currently in preparation, we apply this approach to equations \eqref{a2} driven by Gaussian random forces
 with fast decaying correlations.
\smallskip

To illustrate applications of our results to PDEs with randomness, we have chosen the primitive equations of ocean and atmosphere
dynamics. These are equations for a vector-function $u=(u_1,u_2)$ of space-variables $(x,y,z)$:
\begin{equation}\label{a6}
\partial_t u-\Delta u+(u\cdot \nabla_2)u-(\int_{-h}^z {\rm{div}}_2 u(\cdot,\cdot,\xi)d\xi)\partial_z u+\nabla_2 p=\eta.
\end{equation}
The equations are supplemented with periodic boundary conditions and even symmetry in the $z$-variable; see Section~\ref{s_6.4}. In \eqref{a6}, $\eta(t)$ is a random process taking values in the space $E:=H^m$, $m\ge2.$ Decomposing $\{\eta_k = \eta\!\mid_{I_k}\}$
as in \eqref{H^1}, we assume that the process  $\hat\eta_k=(\nu_k, \eta_k^0)$ satisfies assumptions
$\pmb{(\eta1)}, \pmb{(\eta2)}, \pmb{(\eta3')}$.  Theorem~B then applies to the corresponding system \eqref{a5}. Consequently, equation~\eqref{a6} is exponentially mixing in the dual-Lipschitz norm $\| \cdot\|^*_{L, H^m}$.
Once again, a key point in verifying the assumptions of Theorem~B for equation~\eqref{a6} is that the theorem applies to systems driven by continuous random forces.

\medskip

\noindent
{ \bf Acknowledgements.}
PG thanks Professor Sergei Kuksin for his invitation to Universit\'{e}
Paris Cit\'{e}, and for his scientific supervision on ergodicity and mixing for random dynamical systems.
PG thanks the financial support of the China Scholarship Council (No. 202406620219) and National Natural Science Foundation of China (Grant No. 12371188).
SK acknowledges the support from the Basic Research Program of HSE University.

\subsection*{Notation and Agreements}
We fix  some complete probability space $(\Omega,\FF,\IP)$ isomorphic to interval $[0,1]$ with
the Lebesgue $\sigma$-algebra and Lebesgue measure. For short we call it ``probability space
$\Omega$", and
 systematically use it and its independent copies.
All Hilbert spaces are assumed to be separable.
We write $\Z$ ($\Z_+$, $\Z_-$) for the set of (non-negative, non-positive) integers,
denote by~$B_E(a,r)$ an open $r$-ball in a Hilbert space~$E$, centred at~$a$, and write~$B_E(r)$ if $a=0$.  When~$E$ is finite dimensional, we write $\ell=\ell_E$
for the Lebesgue measure on~$E$. By~$\DD(\eta)$ we denote the law of a random variable~$\eta$, and $\strela$ signifies the weak convergence of measures.
If~$M$ is a Polish (i.e. complete metric and separable)
 space, then  we write $\BB(M)$ for its Borel $\sigma$-algebra  and~$\PP(M)$ for the set of probability Borel
measures on~$M$. Borel-measurable maps from~$M$ to another measurable space are often called just~{\it measurable\/}, and a random field  $\{\zeta^\omega_x\in M_1, x\in M\}$, where $M_1$ is another Polish space, is said to be measurable if it define a measurable map $(\omega,x)\to\zeta_x^\omega$ from~$\Omega\times M$ to~$M_1$.
For a measurable map $f:M\rightarrow N$ and a measure $\mu\in \PP(M)$, we denote by
 $f_*(\mu)\in\PP(N)$ its image under the map $f$. Given a subset $Q\subset M$, we denote
 by $Q^c$ its complement, by $\bar Q$ its closure, and by $\bf 1_Q$ its indicator function.
  For Hilbert
  spaces $X$ and $Y$ (of finite or infinite dimension) $\LL(X,Y)$ stands for the Banach space of of bounded
  linear operators from $X$ to $Y$, given the operator norm,  and for an operator $A\in \LL(X,Y)$,
  $A^*$ stands for the adjoint operator.

\section{The setting}\label{s_1}
Let $H$  and $E$ be  Hilbert spaces of finite or infinite dimension, and $S$ be a map
 $$S:H\times E\to H,$$
which is  twice continuously differentiable and is  bounded on  bounded subsets of $H\times E$,
 together with its derivatives up to the second order.
For $\tau\in \Z$, consider the random dynamical system
\begin{equation}\label{RDSN}
	u_k=S(u_{k-1},\eta_k), \quad k\ge \tau+1,
\end{equation}
where $u_k\in H$, $\{\eta_k\}_{k\in\Z}$ is a random process that takes values in~$E$, and is defined on a complete probability space~$\Omega$ (see Notation and Agreements).
For $k\in \Z,$ we denote by $\pi_k\subset \FF$ the $\sigma$-algebra, generated by the random variables $\{\eta_l\}_{l\leq k}.$
We will write $\pi_k(\{\eta_l\})$ to indicate that this $\sigma$-algebra is defined in terms of the process $\{\eta_l\}.$
We assume that for each $k$, $\supp\DD(\eta_k)\subset\KK$, where $\KK$ is a compact subset of~$E$. For convenience of presentation, we assume that $\eta_k^\omega\in \KK$
for all $k$ and $\omega$, so $\{\eta_k\}$ may be regard as a process in $\KK$.
We also suppose that there is a compact subset $X\subset H$ such that $S(X\times\KK)\subset X$.
The sets $X$ and $\KK$  are assumed to
contain the origins of  spaces $H$ and $E$, respectively.
System \eqref{RDSN} is supplemented with an initial condition
\begin{equation}\label{ICN}
	u_\tau=u\in X,
\end{equation}
where $u\in X$ is a given point or, more generally, a $\pi_\tau-$measurable random variable.
Then~\eqref{RDSN} define a system $\{u_k(u)\}_{k\ge \tau}$ in~$X$, and a solution $u_s$ of~\eqref{RDSN} and \eqref{ICN} is $\pi_s-$measurable $\forall s\geq \tau$. We are interested in the long-time behaviour of the distributions of trajectories for this system. 

\begin{definition}\label{dws}
A process $\{\hat{u}_k\}_{k\geq \tau},$ defined on some probability space,
is called a {\it weak solution\/}\ of \eqref{RDSN}, \eqref{ICN}, if it satisfies \eqref{RDSN}, \eqref{ICN} with
 the process $\{\eta_k\}_{k\in \Z},$
replaced by some other process $\{\hat{\eta}_k\}_{k\in \Z}$, distributed as $\{\eta_k\}_{k\in \Z}$ and
defined on the same probability space as $\{\hat u_k\}$, while $u$ is replaced by
$\hat{u}$, measurable with respect to the $\sigma$-algebra $\pi_{\tau}(\{\hat{\eta}_{k}\})$, and
 such that $\DD(u,\eeta_{\tau})=\DD(\hat{u},\hat{\eeta}_{\tau}).$
\end{definition}

It is easy to see that all weak solutions $\{\hat{u}_k\}$
of \eqref{RDSN}, \eqref{ICN} have the same distribution. Our goal is to investigate the long-time behaviour of this distribution as $k\to\infty$.

Let $M$ be a Polish space with a metric $d_{M}$, $\mathcal{P}(M)$ be the space of probability measures on $(M,\mathcal{B}(M))$ and $C_{b}(M)$ be the space of bounded continuous functions $f: M\rightarrow \mathbb{R}$, endowed with the sup-norm $\|f\|_{\infty}.$
For any measures~$\mu,\nu\in \mathcal{P}(M)$, we define the total variation distance between them  as
\be\label{t_var}
\|\mu-\nu\|_{var}=\sup\limits_{\Gamma\subset \mathcal{B}(M)}|\mu(\Gamma)-\nu(\Gamma)|=\tfrac12\sup\limits_{f\in C_{b}(M),\|f\|_{\infty}\leq 1}|(f,\mu)-(f,\nu)|\leq 1.
\ee
We recall that if $M=\mathbb{R}^{n}$ with  Lebesgue measure $l_{n}(dx)$ and $\mu=\rho_{\mu}(x)l_{n}(dx),\nu=\rho_{\nu}(x)l_{n}(dx)$ for some $L^1-$functions $\rho_{\mu}$ and $\rho_{\nu}$, then
\begin{equation}\label{20}
\|\mu-\nu\|_{var}=\tfrac12\int_{\mathbb{R}^{n}}|\rho_{\mu}(x)-\rho_{\nu}(x)|l_{n}(dx).
\end{equation}

Denote by $L_{b}(M)$ the space of bounded Lipschitz functions on $M,$
that is, of functions $f\in C_{b}(M)$ for which
$Lip(f):=\sup_{u\neq v}\frac{|f(u)-f(v)|}{d_M(u,v)}<\infty.$
This space is endowed with the Banach norm
$\|f\|_{L}:=\|f\|_{\infty}+Lip(f).$
 The dual-Lipschitz distance between  measures~$\mu,\nu\in \mathcal{P}(M)$ is defined as
\begin{equation}\label{LipD}
\|\mu-\nu\|_{L,M}^*=\|\mu-\nu\|_L^*=\sup_{\|f\|_{L}\leq 1}(\langle f,\mu\rangle-\langle f,\nu\rangle)=\sup_{\|f\|_{L}\leq 1}|\langle f,\mu\rangle-\langle f,\nu\rangle|\leq 2.
\end{equation}
Obviously $\| \cdot\|_L^* \le \frac12 \|\cdot\|_{var}$.
The Kantorovich distance between measures $\mu,\nu\in \mathcal{P}(M)$  is
\begin{equation}\label{Kant}
\|\mu-\nu\|_{K}=\sup_{ \text{Lip}\, (f) \le1} (\langle f,\mu\rangle-\langle f,\nu\rangle)
=\sup_{ \text{Lip}\, (f) \le1}|\langle f,\mu\rangle-\langle f,\nu\rangle|\leq \infty.
\end{equation}
Then $\| \cdot\|_L^* \le \| \cdot\|_K$, and it
 is well known (see \cite{villani2009}) and is easy to check that if the space $M$ is bounded, then the two last distances are
equivalent:
\begin{equation}\label{Lip_Kant}
C_M^{-1} \| \mu-\nu\|_K \le \|\mu-\nu\|_L^* \le  \| \mu-\nu\|_K.
\end{equation}
The celebrated Kantorovich--Rubinstein theorem (see \cite{villani2009}) suggests  an equivalent  definition for the Kantorovich
distance:
\begin{equation}\label{K_R}
\| \mu-\nu\|_K  = \inf_{\xi, \eta} \EEE d_M(\xi, \eta),
\end{equation}
where the infimum is taken over all random variables $\xi, \eta \in M$ such that $\DD\xi=\mu $ and $\DD\eta=\nu$. Moreover, the
infimum in the r.h.s. is attained.

\begin{definition}\label{dm}
System~\eqref{RDSN} is called {\it exponentially mixing in the total variation distance\/}\, if there are positive numbers~$C$ and~$\gamma$ such that, for any $\tau\in \Z$ and any $\pi_\tau-$measurable initial states $u,u^\prime\in X$, the corresponding weak solutions $\{u_k\}_{k\geq \tau},\{u_k^\prime\}_{k\geq \tau}$ of~\eqref{RDSN}, \eqref{ICN} satisfy the inequality
	\begin{equation}\label{ev}
		\|\DD(u_k)-\DD(u_k^\prime)\|_{\mathrm{var}}\le C\,e^{-\gamma (k-\tau)}, \quad k\ge \tau+1.
	\end{equation}
This system is called {\it exponentially mixing in the dual-Lipschitz distance\/} if relation \eqref{ev} holds with the distance $\|\cdot\|_{\mathrm{var}}$
replaced by $\|\cdot\|_L^*$.
\end{definition}

As in Introduction, for an integer $l$  we denote $\eeta_l=\{\eta_k,k\leq l\}$ and regard $\eeta_l$ as an element of
 $\bKK=\KK^{\Z_-} \subset E^{\Z_-} $,
where $\bKK$ is metrised by the distance \eqref{distance-EE}, and
 denote by $Q_l(\xxi;\cdot)$ the regular conditional distribution
\begin{equation}\label{cond}
Q_l(\xxi;\cdot):=\IP\{\eta_{l+1}\in\cdot\,|\,\eeta_l=\xxi\}\in \PP(\KK).
\end{equation}
Then
\begin{equation}\label{11}
\mathcal{D}(\eta_{l+1})(\cdot)=\int_{E}Q_l(\xxi;\cdot)\mathcal{D}(\eeta_{l})(d\xxi),
\end{equation}
and the measure $Q_l(\xxi;\cdot)$ is defined uniquely, up to a set of $\xxi$'s of zero $\mathcal{D}(\eeta_{l})$-measure, see \cite{dudley2002}.

From now on we assume that the process $\{\eta_k\}$ satisfies the
following continuity condition, similar to the strong Feller property for Markov processes.
\begin{itemize}
	\item [\hypertarget{SF}{\bf(SF)}] \sl For any $l\in \Z$, the regular conditional distribution $Q_l(\xxi;\cdot)$
	as in \eqref{cond}
can be chosen to be a Lipschitz-continuous mapping from~$\bKK$ to the space~$(\PP(\KK),\|\cdot\|_{var})$. Moreover, there is $C>0$, such that
\begin{equation*}\label{LP-measures}
	\|Q_l(\xxi;\cdot)-Q_l(\xxi';\cdot)\|_{var}\le C\dd(\xxi,\xxi')\quad
	\mbox{for any $\xxi,\xxi'\in\bKK$~and $l\in \Z$}.
\end{equation*}
\end{itemize}
This condition allows us to uniquely define $Q_l(\xxi;\cdot)$ for $\xxi\in \bKK$~and $l\in \Z$, and to regard $Q_l(\xxi;\cdot)$ as a transition probability from $\bKK$ to $\KK$, Lipschitz continuous in $\xxi$ (see \cite[Section 5.1]{{parthasarathy2005}}).

Now we state a recurrence to the origin
 condition for the process $\{\eta_k\}$. For that, for $\xxi\in\bKK$ and $l,k, s \in \Z$ such that
$k>l$ and $s\ge0$, we consider the regular conditional distributions
$
\DD\{ (\eta_k, \dots, \eta_{k+s}) \mid \eeta_l = \xxi\}.
$
The latter may be expressed  via the transition probabilities $Q_r$ since for any $m\ge1$,
\[
\begin{split}
\mathcal{D}\big( {\eta}_{l+1},
\dots, {\eta}_{l+m} \mid \eeta_l=  \xxi)\big)
=Q_l(\xxi;\dd y_{1})Q_{l+1}(\xxi_{1};\dd y_{2})\cdots Q_{l+m}(\xxi_{m};\dd y_{m}),
\end{split}
\]
where $\xxi_{m}:=(\xxi,y_{1},\dots,y_{m})$.

\begin{description}\sl
	\item [\hypertarget{RZ}{(RZ) Recurrence to zero}.] \sl For any $n\in\N$   and $\delta>0$, there is an integer $s\ge0$ such that
	\begin{equation}\label{RZ}
\inf_{l\in\Z}	\inf_{\xxi\in\bKK}\mathbb{P} \big\{| {\eta}_k|<\delta,~k=l+s+1,\ldots,l+s+n \mid \eeta_l = \xxi
\big\}>0.
	\end{equation}
\end{description}

Jointly  assumptions {\hypertarget{RZ}{\bf(RZ)}}, {\hypertarget{SF}{\bf(SF)}} imply that the process $\{\eta_k\}$ is feebly
 exponentially mixing. See  Corollary~\ref{cor1}.

\section{The Markovian lifting}\label{s_2}
Let us introduce the product space $\XXXX=X\times\boldsymbol{\KK}$ endowed with the metric
\begin{equation}\label{metric-X}
\dd_\XXXX(U,U')=L\,\|v-v'\|_H+\dd(\xxi,\xxi'),
\end{equation}
where $U=(v,\xxi),U'=(v',\xxi')$ and $L\ge1$ is a parameter specified later. We denote the natural projections of $\XXXX$ to its
 two components by
\begin{equation}\label{nat-proj}
	\Pi_X:\XXXX\to X, \quad \Pi_\bKK:\XXXX\to\bKK.
\end{equation}
We consider a map $\SSS:\XXXX\times \KK\to\XXXX$ defined by the relation
\begin{equation}\label{S-stationary}
\SSS(U,\eta)=\bigl(S(v,\eta),(\xxi,\eta)\bigr), 	
\end{equation}
where $U=(v,\xxi)\in\XXXX$ and $\eta\in \KK$.

Let $\{u_k, k\ge1\}$, be a solutions of \eqref{RDSN}, \eqref{ICN}. Denoting $U_k= (u_k, \eeta_k)$, $k\ge\tau$, we
immediately see that the process $\{U_k\}$ satisfies
\be\label{Uk}
U_k = \SSS(U_{k-1}, \eta_k), \quad k\ge \tau+1, \;\; U_\tau =(v, \eeta_\tau).
\ee
Note that the initial data $U_\tau$  are of special form -- their second component always is $\eeta_\tau$.
Our next goal is to extend this dynamics to a Markov dynamics in $\XXXX$ which allows for all initial data at $k=\tau$.
Since we care only about  the laws of solutions which do not change if we replace the driving  process $\{\eta_k\}$ by
another one with the same distribution, we will firstly  construct a
 special probability space $(\widehat\Omega,\widehat\FF,\widehat\IP)$ and on it a
 processes ${}_\tau\eta_k$, distributed   as  $\{\eta_k\}$.   Then we will use that  space
 and the tools, employed  to construct the new process, to build a required Markov extension of dynamics \eqref{Uk}.

To define $\widehat\Omega$, we take countably many independent copies of the complete probability space
 $\Omega$ as in Notation and Agreements, and set
$$\widehat\Omega:=\Omega^{\Z}=\cdots\times\Omega\times\Omega\times\cdots\times\Omega\times
\cdots,~\widehat\Omega=\{\oomega:=(\omega_k,k\in\Z):\omega_k\in\Omega\}.$$
We provide $\widehat\Omega$ with the product $\sigma$-algebra and product probability $\widehat\IP$, and then
define $\widehat\FF$ as a completion of that $\sigma$-algebra in $\widehat\IP$. For integers $k\ge \tau$ and
$\widehat\omega\in\widehat\Omega$, let ${}_\tau\omega_k=(\omega_\tau,\cdots,\omega_k)$
(so ${}_\tau\omega_\tau=\omega_\tau$), and let $\oomega_k=\{\omega_l, l\le k\}$.
Let $\{\widehat\FF_k\}_{k\in \Z}$ be the natural filtration of $(\widehat\Omega,\widehat\FF,\IP)$ (completed with respect to
$\widehat\IP$). That is, $\widehat\FF_k$ consists of those elements of $\widehat\FF$ that depend
only on $\oomega_k$. In what follows we refer to the space $(\widehat\Omega,\widehat\FF,\IP,\widehat\FF_k)$
as a \textit{suitable (filtered) probability space}, constructed from $\Omega$.
Now for any $\tau\in \Z$, on the space $\widehat\Omega$ we will construct a process $\{{}_\tau\eta_k\}_{k\in \Z}$ as follows.

(i) For $k\le\tau$, we define the random variables~${}_\tau\eta_k$ as ${}_\tau\eta_k(\oomega)=\eta_k(\omega_{\tau})$,
where $\{\eta_k\}$ is the original process in~\eqref{RDSN}.

(ii) For $k\ge\tau+1$, let $\zeta_k=\zeta_k^{\eeta}(\omega)$ be a measurable random field on $\bKK$, valued in $\KK$,
i.e. a measurable mapping
$$\bKK\times \Omega\ni (\eeta,\omega)\mapsto\zeta_k^{\eeta}(\omega)\in\KK,$$
such that
\be\label{zeta_law}
\mathcal{D}(\zeta_k^{\eeta})=Q_{k-1}(\eeta;\cdot)\quad \forall \eeta.
\ee
We denote by $\zeta_k^{\eeta}(\oomega), \oomega\in \widehat\Omega$, the random field which depends on $\oomega\in \widehat\Omega$ only via $\omega_k$,
$$\zeta_k^{\eeta}(\oomega):=\zeta_k^{\eeta}(\omega_k)$$
(here, certainly, $\omega_k$ is the $k-$th component of $\oomega$).
Then, $\{\zeta_k\}_{k\geq \tau+1}$ are independent random fields on $\widehat\Omega$. We will use their naturally defined lifts
to $\XXXX$:
\begin{equation}\label{lift}
\zeta_k^U(\omega)=\zeta_k^{\Pi_{\bKK}U}(\omega),~U\in \XXXX.
\end{equation}

(iii) For $k\ge\tau+1$, we set ${}_\tau\eta_k$ to be a random variable on $\widehat\Omega$, depending on $\oomega_k,$ i.e.
${}_\tau\eta_k(\oomega)={}_\tau\eta_k(\oomega_k)={}_\tau\eta_k(\oomega_{k-1},\omega_k),$ defined as
$$
{}_\tau\eta_k(\oomega_{k-1},\omega_k)=\zeta_k^{{}_\tau\eeta_{k-1}(\oomega_{k-1})}(\omega_{k}),
$$
where $\zeta_k^{\eeta}$ is the random field defined in (ii). Then
$\mathcal{D}({}_\tau\eta_k|\widehat\FF_{k-1})=Q_{k-1}({}_\tau\eeta_{k-1};\cdot)$.
By induction in~$k\ge\tau$, using the definition of the fields $\zeta_k^{\eeta}$ jointly with the Fubini
theorem (and using (i) for the base of induction),
we get that for any finite sequence  of integers $\{t_1<t_2<\cdots<t_{N}\leq k\}$  it holds  that
$$\mathcal{D}({}_\tau\eta_{t_{1}},\cdots,{}_\tau\eta_{t_{N}})=\mathcal{D}(\eta_{t_{1}},\cdots,\eta_{t_{N}}).$$
Thus, the process $\{{}_\tau\eta_k\}$ on the space
 $\widehat\Omega$ is distributed as $\{\eta_k\}$, and for each $k\geq\tau$, ${}_\tau\eta_k$ is $\widehat\FF_k-$measurable. So the filtration $\{ \pi_k(_\tau\eta)\}$ is a sub-filtration of $\{\FF_k\}$.
\medskip

Now we introduce a system in~$\XXXX$ with the probability space $\widehat\Omega$ by the following relation:
\begin{equation}\label{RDSL}
	U_k=\SSS(U_{k-1},  \xi_k(\oomega_k)), \quad
	 \xi_k(\oomega_k)= \zeta_k^{  \xi_{k-1}(\oomega_{k-1})}(\om_k),\quad k\ge\tau+1,
\end{equation}
where $U_k=(v_k, \xxi_k)$, $k\ge \tau+1$, depends on $\oomega_k$.
System~\eqref{RDSL} is supplemented with an initial condition
\begin{equation}\label{ICL}
	U_\tau=V:=(v,\eeta),
\end{equation}
where $V \in \XXXX$ is a point, or more generally an $\widehat\FF_\tau-$measurable random variable, $V=V(\oomega_\tau)$.

\begin{remark}\label{imp}
We draw the reader's attention to the fact that a solution of  problem \eqref{RDSL}, \eqref{ICL} is defined in terms of the initial value $V$ and the  random fields $\zeta_k$
with $k\ge \tau+1.$ So only in terms of $V$ and the transitional probabilities $\{Q_l\}$. \qed
\end{remark}

We call a process $\{U_{k}\}_{k\geq\tau}$, defined by \eqref{RDSL}, \eqref{ICL} \textit{a solution}, and denote it as $\{U_{k}(V)\}_{k\geq\tau}$,
or as $\{{}_\tau U_{k}(V)\}_{k\geq\tau}$ to indicate that the initial data is prescribed at $k=\tau$.
We do not distinguish solutions of \eqref{RDSL}, \eqref{ICL}, constructed using different suitable probability
spaces $\widehat\Omega$ and different random fields  $\zeta_k$ on them since by given below
Proposition \ref{p-markov-chain}, they have the same distribution.

For any $k$ and any $U=(v,\xxi)\in \XXXX$, let us set
\begin{equation}\label{kp}
\PPP_k(U;\cdot):=Q_{k}(\xxi;\cdot)\in \PP(\KK).
\end{equation}
Then by \eqref{zeta_law} $\PPP_k(U;\cdot)=\mathcal{D}(\zeta_{k+1}^{\xxi})$.

\begin{lemma}\label{l_lift} If in  \eqref{ICL} $\eeta$ equals ${}_\tau\eeta_\tau = \eeta_\tau$, \
$v$ is $\pi_{\tau}-$measurable,\footnote{Note that
$\pi_{\tau}(\{\eta_k\})=\pi_{\tau}(\{\prescript{}{\tau}\eta_k\})$.}
 and
 $\{u_k\}_{ k\geq\tau}$ is a weak solution of \eqref{RDSN}, \eqref{ICN}, where $u=v$ and $\eta_k$ is the random
 process $\{\prescript{}{\tau}\eta_k\}$, 
  then process $\{U_k=(u_k,{}_\tau\eeta_{k})\}_{k\geq\tau}$ is a solution of \eqref{RDSL}, \eqref{ICL}.
\end{lemma}
\begin{proof} The assertion follows from the form of map $\SSS$.
\end{proof}

So system  \eqref{RDSL} is a required Markov extention of  \eqref{Uk}. It is a
lifting of system \eqref{RDSN} to the product space $\XXXX$.

\begin{proposition}\label{p-markov-chain}
For any $\tau\in \Z$, the family of  trajectories  $\{ U_k(V)\}_{k\geq\tau}$ of~\eqref{RDSL} with all possible deterministic initial conditions~$V\in\XXXX$, form an inhomogeneous Markov process in~$\XXXX$, corresponding to the filtration~$\{\widehat\FF_k\}_{k\ge\tau}$, with the transition probability from time $k$ to $k+1$
	\begin{equation}\label{TF}
P_k(U;\cdot)=\SSS_*\bigl(U,\PPP_k(U;\cdot)\bigr), \quad U\in\XXXX, \quad k\ge\tau.
	\end{equation}
\end{proposition}
Relation \eqref{TF} means that $P_k(U;\cdot)\in \PP(\XXXX)$ is an image of the measure $\PPP_k(U;\cdot)$ under the mapping $\KK\rightarrow\XXXX,$
sending $\eta$ to $\SSS(U, \eta)$. So for any bounded measurable function $f$ on $\XXXX$,
\begin{equation}\label{9}
\int f(V)P_k(U;dV)=\int f(\SSS(U, \eta))\PPP_k(U;d\eta).
\end{equation}

\begin{corollary}\label{c_uniq}
The law of a solution
$\{U_k(V)\}$ of \eqref{RDSL}, \eqref{ICL} does not depend on specific choices of the suitable probability space
 $\widehat\Omega$ and  random fields $\zeta_k$.
\end{corollary}

\begin{proof}[Proof of Proposition~\ref{p-markov-chain}]
First, by construction, the random variable~$U_k(V)$, $k\ge\tau$, depends only on~$\oomega_k$, so it is $\widehat\FF_k$-measurable.
Then, for any integers $k\ge\tau$, any bounded measurable function $f:\XXXX\to\R$,
and any $\widehat\FF_k$-measurable function $\varphi(\oomega_k)$, in view of~\eqref{RDSL} with $k:=k+1$, we have
\begin{align*}
\E\bigl(f(U_{k+1}(V))\varphi(\oomega_k)\bigr)
&=\E^{\oomega_k}\Bigl[\E^{\omega_{k+1}} f\bigl(\SSS(U_{k}^{\oomega_k},\zeta_{k+1}^{{}_\tau\eeta_k(\oomega_k)}(\omega_{k+1})\bigl)\bigr)\varphi(\oomega_k)\Bigr]\\
&=\E^{\oomega_k}\Bigl[\E^{\omega_{k+1}} f\bigl(\SSS(U_{k}^{\oomega_k},\zeta_{k+1}^{U_k^{\oomega_k}}(\omega_{k+1})\bigl)\bigr)\varphi(\oomega_k)\Bigr]\\
&=\E^{\oomega_k}\Bigl[\varphi(\oomega_k)\,\E f\bigl(\SSS(U_{k}^{\oomega_k},\zeta_{k+1}^{U_k^{\oomega_k}}\bigl)\bigr)\Bigr]
\end{align*}
(we recall \eqref{lift}).  Thus,
\begin{equation}\label{M1}
\E\{f(U_{k+1}(V))\,|\,\widehat\FF_k\}=\E f\bigl(\SSS(x,\zeta_{k+1}^{ x})\bigr)\bigr|_{x=U_{k}}.
\end{equation}
It remains to note that the map $\XXXX \ni x\mapsto \E f(\SSS(x,\zeta_{k+1}^{ x}))$  defines a bounded measurable function~$\hat f_{\tau,k}(x)$, depending only on~$f,k$ and $\tau$. Thus, it holds that
\begin{equation*}
	\E\{f(U_{k+1}(V))\,|\,\widehat\FF_k\}=\hat f_{\tau,k}(U_{k}(V)) \quad \mbox{$\IP$-almost surely},
\end{equation*}
so $\{ U_k(V)\}$ is a Markov process. Relation~\eqref{TF} with $k:=k+1$ follows obviously from \eqref{M1} and \eqref{9}.
\end{proof}

The Markov process in Proposition \ref{p-markov-chain} defines Markov operators ${}_\tau\PPPP_k^*:\PP(\XXXX)\to \PP(\XXXX), k\geq \tau,$ by relations
\be\label{flow}
\prescript{}{\tau}\PPPP_{\tau}^*=id, \quad \prescript{}{t}\PPPP_{u}^*\circ\prescript{}{s}\PPPP_{t}^*=\prescript{}{s}\PPPP_{u}^*~~\forall s\leq t\leq u,
\ee
\[
\prescript{}{\tau}\PPPP_{\tau+1}^*\mu=\int P_\tau (U;\cdot)\mu(dU),~\forall \tau\in \Z.
\]
Alternatively, for any $k\ge \tau,$ $\prescript{}{\tau}\PPPP_{k}^*\mu=\DD ({}_\tau U_k(V))$, where $\DD(V)=\mu.$

\begin{lemma}\label{l_Lip}
1) The mappings $\XXXX \rightarrow (\PP(\XXXX),\|\cdot\|_L^*),~U \mapsto P_k(U)$ are uniformly  Lipschitz-continuous.
\par
2) Same about operators $\prescript{}{k}\PPPP_{k+1}^*: \PP(\XXXX)\rightarrow \PP(\XXXX)$.
\par
3) For any $k\in \Z$ and $ s\ge 0,$ the operator $\PP(\XXXX)\rightarrow \PP(\XXXX\times\cdots\times\XXXX),~\PP(\XXXX)\ni \DD(V) \mapsto \DD({}_kU_{k}(V),\ldots,{}_kU_{k+s}(V))\in \PP(\XXXX\times\cdots\times\XXXX)$ is Lipschitz-continuous, with respect to the dual-Lipschitz norms.
 Its Lipschitz constant is bounded by a $k-$independent number.
\end{lemma}
\begin{proof}[Proof of Lemma~\ref{l_Lip}]
1) By construction, for $V\in \XXXX$,
\begin{align*}
P_k(V;\cdot)=\DD ({}_k U_{k+1}(V))=\DD (\SSS(V,\zeta_{k+1}^V)).
\end{align*}
 So for any $f$ as in \eqref{LipD},
\begin{align*}
\langle f,P_k(V_1)\rangle-\langle f,P_k(V_2)\rangle&=\E \big(f(\SSS(V_1,\zeta_{k+1}^{V_{1}}))-f(\SSS(V_2,\zeta_{k+1}^{V_{2}}))\big)\\
&\leq Cd_{\XXXX}(V_1,V_2)+\E \big(f(\SSS(V_2,\zeta_{k+1}^{V_{1}}))-f(\SSS(V_2,\zeta_{k+1}^{V_{2}}))\big)\\
&\leq C_1d_{\XXXX}(V_1,V_2),
\end{align*}
where we used Hypothesis~{\rm \hyperlink{SF}{(SF)}}.

2) Let $f$ be as above, $\mu_1,\mu_2\in \PP(\XXXX)$ and $V_1,V_2$ be two random variables in $\XXXX$, depending on $\omega_k$ and distributed as $\mu_1$ and $\mu_2$, accordingly.
Then, by 1) with $V_j:=V_j^{\omega_k}$,
\begin{align*}
&\langle f,\prescript{}{k}\PPPP_{k+1}^*\mu_1\rangle-\langle f,\prescript{}{k}\PPPP_{k+1}^*\mu_2\rangle\\
=&\int\Big[\int\big(f(\SSS(V_1^{\omega_k},\zeta_{k+1}^{V_1^{\omega_k}}(\omega_{k+1})))-f(\SSS(V_2^{\omega_k},\zeta_{k+1}^{V_2^{\omega_k}}(\omega_{k+1})))\big)\IP(d\omega_{k+1})\Big]\IP(d\omega_k)
\\\leq& C\int d_{\XXXX}(V_1^{\omega_k},V_2^{\omega_k})\IP(d\omega_k).
\end{align*}
By \eqref{Lip_Kant} and \eqref{K_R}, the random variables $V_1$ and $V_2$
  may be chosen in such a way that the integral in the r.h.s.
   is no bigger than $C^{'}\|\mu_1-\mu_2\|_{L}^{*}$. This implies assertion 2).

3) We argue by induction in $s$.
For $s=0$, the assertion  is proved in 2). For $s\ge 1$, we write $({}_kU_{k}(V),\ldots,{}_kU_{k+s}(V))=(\U,U_{k+s})$,
 where $\U=\U(\hat\oomega)$,
$\hat\oomega=\oomega_{k+s-1}, U_{k+s}=U_{k+s}(\hat\oomega,\omega_{k+s})$, and similarly write $({}_kU_{k}(V'),\ldots,{}_kU_{k+s}(V'))$
as $(\U',U_{k+s}')$. Then for any function $F(\U,U_{k+s})$ on $\XXXX^s\times \XXXX$ as in \eqref{LipD}, we have
\begin{equation}
\begin{split}\label{xy}
&\E F(\U,U_{k+s})-\E F(\U'_k,U'_{k+s})\\
=&\int\Big[\int\big( F(\U,U_{k+s})-F(\U',U_{k+s})\big)+\big(F(\U',U_{k+s})-F(\U',U'_{k+s})\big)\IP(d\omega_{k+s})\Big]\IP(d\hat\oomega).
\end{split}
\end{equation}
The inner integral is a sum of two. For any $\hat\oomega$ fixed, the first one obviously is bounded by
 $C\E d_{\XXXX\times\ldots\times\XXXX}(\U(\hat\oomega),\U'(\hat\oomega))=:Z$. The second one may be estimated by
 1) with $k:=k+s$ and $U=U_{k+s-1}$, and also  is bounded by $Z$ (with another constant $C$).
So the r.h.s. of \eqref{xy} is bounded by $\E Z$. The l.h.s. of \eqref{xy} depends only on the laws of
$(\U,U_{k+s})$ and $(\U',U'_{k+s})$. By the Markovness, these two measures depend only on $\DD(\U)$ and $\DD(\U')$, so we may replace $\U$ and $\U'$ by any random
vectors with the same distributions, without changing \eqref{xy}. Using again \eqref{Kant},  \eqref{Lip_Kant},
we find modifications of $\U$ and $\U'$ such that
$$\E Z\le C\|\DD(\U)-\DD(\U')\|_{L,\XXXX\times\ldots\times\XXXX}^{*}.$$
So 3) follows by induction.
\end{proof}

\section{Main result: finite-dimensional case}\label{s_3}
In this section and in the next one
we state and prove our main results for the case of finite-dimensional phase space $H$.
 So, till the end of  Section~\ref{s_4},
$$\textmd{dim}~H<\infty.$$

\subsection{Main result for system \eqref{RDSN}, and mixing in the lifted system \eqref{RDSL}}

For any finite-dimensional subspace $F\subset E$, $E=F\oplus F^\perp,$ where $F^\perp$ is the orthogonal complement of $F$. Accordingly we write
$E\ni y=(y_F,y_F^\perp)$
and consider the associated projection operators $\mathsf P_F:E\to F, y=(y_F,y_F^\perp)\mapsto y_F$ and $\mathsf \mathsf P_{F}^{\perp}:E\to F^\perp, (y_F,y_F^\perp)\mapsto y_F^\perp$. Then
for any $\xxi\in\bKK$ and $l\in \Z$,
the measure $Q_l(\xxi;\cdot)$ admits a disintegration with respect to $\mathsf P_F^\perp$:
\begin{equation}\label{24}
	Q_l(\xxi;\dd y)=Q_{l F}^\perp(\xxi;\dd y_F^\perp)Q_{l F}(\xxi,y_F^\perp;\dd y_F),
\end{equation}
where $Q_{l F}^{\bot}(\xxi;\cdot)\in \PP(F^\bot)$ is the projection of $Q_l(\xxi;\cdot)$ to $F^\bot$,
and $Q_{l F}(\xxi,y_F^\perp;\cdot)\in \PP(F)$. See~\cite[Theorem~10.2.1]{dudley2002}.

Let us introduce the following assumption on the operator $S$ and process $\{\eta_k\}$.
\begin{description}\sl
	\item [\hypertarget{LCR}{(LCR) Linearised controllability and regularity}.]
~\par
(\textit{a}) There exist a finite-dimensional subspace $F\subset E$ and an open neighbourhood
 $O$ of  $X\times\KK$ in $H\times E$ such that for any $(u,\eta)\in O$,
the operator $D_\eta S(u,\eta): F\to H$ is surjective.

(\textit{b}) For any $l\in \Z$ and each $\xxi\in\bKK$, there is a Lipschitz continuous function $\rho_{l F}:\bKK\times E\to\R_+$ supported by~$\bKK\times\KK$ such that in decomposition \eqref{24}, it holds that
\begin{equation}\label{lcr}
Q_{l F}(\xxi,y_F^\perp;\dd y_F)=\rho_{l F}(\xxi,y_F^\perp,y_F)\ell_F(\dd y_F),
\end{equation}
where $\ell_F$ stands for the Lebesgue measure on~$F$, and  functions $\rho_{l F}$, $l\in \Z$,  are  uniformly Lipschitz.
\end{description}

\begin{remark}
1) Since the set $X\times\KK$ is compact and the map $S$ is $C^2-$smooth, the linearised controllability holds if for each $(u,\eta)\in X\times\KK$, the mapping $D_\eta S(u,\eta): F\to H$ is surjective.

\noindent
2) Moreover, if $\{e_1,e_2,\cdots\}$ is a Hilbert basis of $E$, then the space $F$ above
may be chosen of the form $span\{e_1,\cdots, e_N\}$,
for some $N<\infty$. \qed
\end{remark}

\begin{remark} \label{r_3.2}
 If $E$ is a finite-dimensional space, let us consider the following hypothesis:
\begin{description}\sl
\item [\hypertarget{F}{{}}]
$\,\; \;\;\;$  For any $(u,\eta)\in O$,
the operator $D_\eta S(u,\eta):E\to H$ is surjective. Moreover, for any $l\in \Z$ and each $\xxi\in\bKK$, there is a Lipschitz continuous function $\rho_{l}:\bKK\times E\to\R_+$ supported by~$\bKK\times\KK$ such that
\begin{equation}\label{finiteQ}
Q_{l}(\xxi;\dd y)=\rho_{l}(\xxi,y)\ell_E,
\end{equation}
where functions $ \rho_l$ are uniformly Lipschitz.
\end{description}
It provides a simple sufficient condition for the validity of Hypotheses~{\rm \hyperlink{SF}{(SF)}} and \hyperlink{LCR}{(LCR)}(b) with $F=E$. \qed
\end{remark}

Finally, we assume that the free system \eqref{RDSN}$|_{\eta_k\equiv0}$ is dissipative. To state this assumption, for
 an integer $k\ge1$, we denote $\mathbf{0}_k=(0,\dots,0)$, where $0\in E$ is repeated~$k$ times, and for a vector $\vec{\xi}_n=(\xi_1,\dots,\xi_n)\in\KK^n$, we set
\begin{equation}\label{notat}
\mbox{$S_n(v;\vec{\xi}_n)=u_n$, where $\{u_1,\dots,u_n\}$ is a trajectory of~\eqref{RDSN} with $\eta_k=\xi_k$}.
\end{equation}

\begin{description}\sl
	\item [\hypertarget{GD}{(GD) Global dissipation}.] There is an integer $k\ge1$ and a number $a\in(0,1)$ such that
\begin{equation*}
	\|S_k(u;\mathbf{0}_k)\|_H\le a\,\|u\|_H
	\quad\mbox{for any $u\in H$}.
\end{equation*}
\end{description}

The main result of this paper is the following assertion:
\begin{theorem}\label{mixing}
Suppose that Hypotheses~{\rm \hyperlink{SF}{(SF)}, \hyperlink{RZ}{(RZ)}, \hyperlink{LCR}{(LCR)}} and~\hyperlink{GD}{\rm(GD)} are fulfilled.
Then system  \eqref{RDSN} is exponentially mixing in the total variation distance. Moreover, there exists a process $\{\hat{u}_k\}_{k\in \Z}$ which satisfies \eqref{RDSN} in the sense of distributions:
\begin{equation}\label{16}
\DD(\hat{u}_k)=S_{*}(\DD(\hat{u}_{k-1},\hat\eta_k)) \quad \forall k\in \Z,
\end{equation}
where $\{\hat\eta_k\}_{k\in \Z}$ is distributed as $\{\eta_k\}_{k\in \Z}$. Process  $\hat u$ is such that
for any $\tau\in\Z$, any $\pi_\tau$-measurable
 initial state $u\in X$ and any $s\ge 0$, the corresponding trajectory $\{u_k={}_\tau u_k(u)\}_{k\geq \tau}$ of~\eqref{RDSN}, \eqref{ICN} satisfies inequalities
\begin{equation}\label{22}
\|\DD(\hat{u}_k,\cdots,\hat{u}_{k+s})-\DD({}_\tau u_k(u),\cdots,{}_\tau u_{k+s}(u))\|_{L}^{*}\leq C_s e^{-\gamma(k-\tau)},\quad k\ge \tau+1,
\end{equation}
and
\begin{equation}\label{uev}
\|\DD({}_\tau u_k(u))-\DD(\hat{u}_k)\|_{\mathrm{var}}\le C\,e^{-\gamma (k-\tau)}, \quad k\ge \tau+1.
\end{equation}
The positive constants~$C, C_s$ and~$\gamma$ do not depend on $\tau\in \Z$. The property \eqref{22} defines the distribution of process $\{\hat u_k\}$
in a unique way.
\end{theorem}

In Section \ref{s_6.1.1} we show that the restrictions {\rm \hyperlink{SF}{(SF)}, \hyperlink{RZ}{(RZ)} and \hyperlink{LCR}{(LCR)}}(b), imposed
on the process $\eta_k$, are mild, and are satisfied by plenty of random processes.

\begin{remark} 1) Let  $u, u' \in X$ be non-random vectors. Applying the argument that proves
  Theorem~\ref{mixingL} not to solutions $\{{}_\tau U_k(V)\}$ and $\{\hat U_k\}$, but to $\{{}_\tau U_k(V)\}$  and
  $\{{}_\tau U_k(V')\}$, where $V=(u, {}_\tau\eeta_\tau)$ and  $V'=(u', {}_\tau\eeta_\tau)$, we derive from
  \eqref{evL} that
  \begin{equation}\label{uev1}
\|\DD({}_\tau u_k(u))- \DD({}_\tau{u'}_k) (u')\|_{\mathrm{var}}\le C' \| u-u'\|_H \,e^{-\gamma (k-\tau)}, \quad k\ge \tau.
\qed
\end{equation}

 \noindent
2) If the forcing  $\{\eta_k\}_{k \in \mathbb{Z}}$ in system \eqref{RDSN} is a stationary process, then the process
$\{\hat{u}_k\}_{k \in \mathbb{Z}}$ is stationary as well. In this case, the assertions of the theorem correspond to
the mixing property for system \eqref{RDSN}, as discussed in \cite{KS-2025}, \cite{KS-2026}.
\end{remark}

To prove the theorem, we consider the lifting of  \eqref{RDSN} to the Markovian system \eqref{RDSL} (see Lemma~\ref{l_lift}),
show in the theorem below that the latter system is mixing and then derive Theorem~\ref{mixing} from that result.

\begin{theorem}\label{mixingL}
Suppose that Hypotheses~{\rm\hyperlink{SF}{(SF)}, \hyperlink{RZ}{(RZ)}, \hyperlink{LCR}{(LCR)}} and~\hyperlink{GD}{\rm(GD)} are fulfilled. Then the Markov process
 on $\widehat\Omega$ defined by RDS~\eqref{RDSL} is exponentially mixing in the dual-Lipschitz distance. That is, there are positive numbers~$C$ and~$\gamma$, independent from $\tau\in \Z$, such that for any $\tau\in \Z$ and any initial states $V,V^\prime\in \XXXX$, which are $\hat\FF_\tau-$measurable random variables, the corresponding trajectories $\{U_k\}_{k\geq \tau},\{U_k^\prime\}_{k\geq \tau}$ of~\eqref{RDSL}, \eqref{ICL} satisfy the inequalities
	\begin{equation}\label{evL}
		\|\DD(U_k)-\DD(U_k^\prime)\|_L^*\le C \| \DD(V) -\DD(V')\|_L^*\,e^{-\gamma (k-\tau)}, \quad k\ge \tau.
	\end{equation}
In the space $\XXXX$, there exists an unique in distribution inhomogeneous Markov process $\{\hat{U}_k\}_{k\in \Z}$ (defined on some other probability space)
with the transition probabilities $P_k$ as in \eqref{TF}, and such that for any $\tau\in \Z$ and any $\hat\FF_\tau-$measurable $V\in \XXXX$, the solution $\{ {}_\tau U_k(V)\}_{k\ge \tau}$
of~\eqref{RDSL}, \eqref{ICL} satisfies the inequalities
	\begin{equation}\label{Uev}
		\|\DD(\hat{U}_k)-\DD({}_\tau U_k(V))\|_L^*\le C\| \DD(V) -\DD(V')\|_L^*
		 \,e^{-\gamma (k-\tau)}, \quad k\ge \tau.
	\end{equation}
The process $\{\hat{U}_k\}_{k\in \Z}$ satisfies relations~\eqref{RDSL} with $\xi_k=\eta_k$
in the sense of distributions. It means that there exists a process $\{\hat\eta_k\}_{k\in \Z}$ distributed as $\{\eta_k\}_{k\in \Z}$, such that
\begin{equation}\label{15}
\DD(\hat{U}_k)=\SSS_{*}(\DD(\hat{U}_{k-1},\hat\eta_k)), \quad \forall k\in \Z.
\end{equation}
\end{theorem}

If the process $\{\eta_k\}$ is stationary, then system \eqref{RDSL} defines a homogeneous Markov processes, and
$\{\hat{U}_k\}_{k\in \Z}$ is a stationary Markov process  (unique in distribution).

\begin{remark}\label{Re1}
1) We note that system \eqref{RDSL} is constructed not in terms of the mapping $S$ and  process $\{\eta_{k}\}$, but rather in terms of
the map $S$ and the transition probabilities $\{Q_l\}_{l\in \Z}$. Similar, the conditions, imposed on the noise in the theorem, also are made not in terms of
the process, but in terms of probabilities $\{Q_l\}_{l\in \Z}$. That is, Theorem~\ref{mixingL} is an assertion about the mapping
$S$ and the transition probabilities $\{Q_l\}_{l\in \Z}$ which satisfy the assumptions of the theorem.

\noindent
2) Relation \eqref{evL} and assertion 3) of Lemma~\ref{l_Lip}, applied to measures $\DD(\hat{U}_k)$ and $\DD(U_k(V))$, imply that for any $k\ge \tau+1$ and any $s\ge 0,$
\begin{equation}\label{evLs}
		\|\DD(\hat{U}_k,\ldots,\hat{U}_{k+s})-\DD(U_k(V),\ldots,U_{k+s}(V))\|_L^*\le C_s e^{-\gamma (k-\tau)}. \qed
\end{equation}
\end{remark}

As we show below in Section 3.3, Theorem~\ref{mixing} is an easy consequence of Theorem~\ref{mixingL} and some results from \cite{KS-2026}. But the  Theorem~\ref{mixingL} is interesting by itself, and in the next section we show how it may be applied to get a Dobrushin-like result for  random processes.

\subsection{Reconstructing a random process from its regular conditional distributions}\label{s_3.2}
Let us analyse how Theorem~\ref{mixingL} applies to a system \eqref{RDSL}, where
 the map $S$ in \eqref{S-stationary} is very degenerate.
Namely, let us consider the case when $H=\{0\}$ (i.e. $H$ is the zero-dimensional space), $E$ is a Hilbert space and $S$ is the mapping
\begin{equation*}
S: H\times E\rightarrow H,~(0,\eta)\mapsto 0.
\end{equation*}
It trivially satisfies Hypotheses~{\rm \hyperlink{GD}{(GD)}, \hyperlink{LCR}{(LCR)}}, and for it Hypothesis~{\rm \hyperlink{LCR}{(LCR)}} with $F=\{0\}$ becomes a tautology{\footnote{If the reader is not happy with the trivial map $S$ we talk about, he/she may easily see that for it the proof of Theorem~\ref{mixingL}, given below in Section~\ref{s_4},
 remains true with some significant simplifications.}}.
Let $\KK\subset E$ and $\boldsymbol{\KK}=\KK^{\Z_{-}}$ be as above, and let $\{Q_l\}_{l\in \Z}$ be some transition probabilities, satisfying Hypotheses~{\rm \hyperlink{SF}{(SF)}} and~\hyperlink{RZ}{\rm(RZ)}. Dropping the trivial component of $U_k$ in $H=\{0\}$, we write the corresponding system \eqref{RDSL} with $k\geq \tau+1$ and some
 fixed $\tau\in \Z$ (defined on a suitable probability space $(\widehat\Omega,\widehat\FF,\widehat\IP)$ as in (i)-(iii) above) as
\begin{equation}\label{8}
{}_\tau\xxi_k=\bigl({}_\tau\xxi_{k-1},\zeta_k^{{}_\tau\xxi_{k-1}}\bigr), 
\end{equation}
$$
{}_\tau\xxi_\tau=\hat\eeta_\tau\in\bKK.
$$
Here the random fields $\{\zeta_k^\eeta\}$ are as in item (ii), and $\hat\eeta_\tau\in\bKK$ is an $\widehat\FF_\tau-$measurable random variable. System \eqref{8} defines in
$\bKK$ an inhomogeneous Markov process with transitional probabilities
\begin{equation}\label{2}
P_k(\xxi;\cdot)= \big(\delta_{\xxi},Q_k(\xxi;\cdot)\big)\in \PP(\bKK), \quad k\in\Z.
\end{equation}
Naturally, a solution $\{{}_\tau\xxi_k\}_{k\ge \tau}$ of \eqref{8} has the form
\begin{equation}\label{1}
{}_\tau\xxi_k=(\hat\eeta_\tau,\xi_{\tau+1},\ldots,\xi_{k}),~k\ge \tau.
\end{equation}
Applying Theorem~\ref{mixingL}, we get that there is a unique in distribution process $\{\hat\xi_k\}_{k\in\Z}$
such that the vectors $\hat\xxi_k=(\cdots,\hat\xi_{k-1},\hat\xi_{k}), k\ge \tau+1,$ make a trajectory of \eqref{8},
for any $\tau$. Since measures \eqref{2} are transition probabilities for the Markov process $\{\hat\xxi_k\}$,
then regular conditional distributions $\IP(\hat\xi_{k+1}|\hat\xxi_k)$ equal $P_k(\hat\xxi_k; \cdot) $.
For any $\tau\in \Z$ and any vector $\hat\eeta_{\tau}\in \bKK$, the law of ${}_\tau \xxi_{\tau+1}$
(see \eqref{8}) is the conditional distribution of $\hat\xxi_{\tau+1}$, given that $\hat\xxi_{\tau}=\hat\eeta_\tau$. And, more generally, the law of
${}_\tau \xxi_{k},k>\tau,$ is a conditional distribution of $\hat\xxi_{k}$, given that $\hat\xxi_{\tau}=\hat\eeta_\tau$. Thus, for any $\tau\in\Z,~k\ge 1$ and $s\ge 0$, taking in \eqref{Uev} $k=\tau+k+s$, $\hat U_k=\hat \xxi_k$ and $U_k(V)={}_\tau\xxi_k$, we get that
\begin{equation*}
\|\DD(\hat\xxi_{\tau+k+s}~|~\hat\xxi_{\tau}=\hat\eeta_\tau)-\DD(\hat\xxi_{\tau+k+s})\|_L^*\le C e^{-\gamma (k+s)},~\forall\hat\eeta_\tau\in \bKK.
\end{equation*}
This is a relation between measures in $\bKK=\{\xxi=(\ldots,\xi_{-1},\xi_{0})\}$. Applying to it the projection $\bKK\ni \xxi\mapsto(\xi_{-s+1},\ldots,\xi_{0})\in \KK^{s}$ (and remembering that the space $\bKK$ is given distance \eqref{distance-EE}), we arrive at the estimate
\begin{equation}\label{5}
\|\DD((\hat\xi_{\tau+k+1},\ldots,\hat\xi_{\tau+k+s})~|~\hat\xxi_{\tau}=\hat\eeta_\tau)-\DD(\hat\xi_{\tau+k+1},\ldots,\hat\xi_{\tau+k+s})\|_L^*
\le C_s e^{-\gamma k},
\end{equation}
valid for any $\tau\in \Z, k, s\ge 0$ and any vector $\hat\eeta_\tau\in \bKK.$  In the usual way this implies that
for any $\tau\in\Z$ and $N,s\ge 0,$
\begin{equation}\label{10}
\begin{split}
|\E f(\hat\xi_{\tau-N+1},\ldots,\hat\xi_{\tau})&g(\hat\xi_{\tau+k+1},\ldots,\hat\xi_{\tau+k+s})\\
-&\E f(\hat\xi_{\tau-N+1},\ldots,\hat\xi_{\tau})\E g(\hat\xi_{\tau+k+1},\ldots,\hat\xi_{\tau+k+s})| \le C_s e^{-\gamma k},
\end{split}
\end{equation}
where $f$ is any measurable function on $\R^{N}, |f|\le 1,$ and $g$ is any Lipschitz-continuous function on $\R^{s}$ such that
 $|g|\le 1$ and $\Lip g\le 1.$ This is an exponential mixing with a set of observables smaller than in the usual definitions
 (e.g. as in \cite{IL}). Following \cite[Section~4.2]{KS-2025}, we call it {\it feeble mixing}. We refer to \cite{KS-2025} for
 discussion, and note that various definitions of mixing with different sets of observables have been intensively discussed during the
 last 20-30 years; see \cite{DDL}.

We summarise these results as a corollary from Theorem~\ref{mixingL}:

\begin{corollary}\label{cor1}
Let $\{Q_l\}_{l\in\Z}$ be a system of transition probabilities from $\bKK$ to $\KK,$ satisfying Hypotheses~{\rm \hyperlink{SF}{(SF)}} and~\hyperlink{RZ}{\rm(RZ)}.
Then there exists a unique in distribution random process $\{\hat\xi_s\in\KK\}_{s\in\Z}$, such that for each $k\in\Z$ its regular conditional
 distribution $\IP(\hat\xi_{k+1}|\hat\xxi_k)$ equals $Q_k(\hat\xxi_k;\cdot)$. This process is feebly exponentially
  mixing in the sense that it satisfies relations \eqref{5} and \eqref{10}.
\end{corollary}

Clearly, it is easy to construct examples of transition probabilities $\{Q_l\}$, satisfying Hypotheses~{\rm \hyperlink{SF}{(SF)}} and~\hyperlink{RZ}{\rm(RZ)}, and  the corollary gives a tool to build  a unique in distribution
  random processes whose  regular conditional distributions equal  $Q_l$'s.
This result is related to Dobrushin's theorems on reconstructing random processes and random fields from their conditional distributions, see \cite{dobrushin-1970} and \cite{follmer1988}.

\subsection{Derivation of Theorem~\ref{mixing} from Theorem~\ref{mixingL}}
\begin{proof}[Proof of Theorem~\ref{mixing}]
Since relation \eqref{uev} implies \eqref{ev} with the doubled constant $C$, then it suffices to prove the assertions
 of the theorem, related to the process $\{\hat u_k\}$. To do that, we define the latter as $\hat u_k=\Pi_{X}\hat U_k, k \in \Z$, where $\{\hat U_k\}$ is the process from Theorem~ \ref{mixingL}. Then, relation
\eqref{16} directly follows from \eqref{15}.

To establish \eqref{22}, we note that by Lemma~\ref{l_lift}, the process $\{U_k=({}_\tau u_k(u),\eeta_k\}_{k\ge\tau}$ is
distributed as process $\{{}_\tau U_k(V)\}$, $V=(u,{}_\tau\eeta_\tau)$. Thus, \eqref{22} follows by applying the projection $\Pi_{X}\times\cdots\times\Pi_{X}$ to
relation \eqref{evLs}.

To establish \eqref{uev}, we denote $\DD(\hat U_k)=\hat m_k, \DD({}_\tau U_k(V))=m_k$. Then, by \eqref{Uev},
\begin{equation}\label{23}
\|\hat{m}_k-m_k\|_L^*\le C\,e^{-\gamma (k-\tau)}, \quad k\ge \tau+1.
\end{equation}
This implies  \eqref{uev}   with the total variation distance replaced by the weaker dual-Lipschitz one. To get \eqref{uev} in the
required form we exploit the approach, used in \cite{KS-2026}. For that we
awake Hypothesis~{\rm \hyperlink{LCR}{(LCR)}}, the decomposition $E=F\oplus F^\perp$
and projections $\mathsf P_F, \mathsf P_F^{\perp}$, used there.
Recall that both processed $\hat U_k$ and ${}_\tau U_k(V)$ are inhomogeneous Markov processes in $\XXXX$ with transitional probabilities $P_k(U;\cdot)$.

Denote
 $\nu_{k+1}=(\Pi_X)_* m_{k+1} =\DD(u_{k+1}) $.  Then by \eqref{TF},
$
\nu_{k+1}=\int S_*(u, Q_k(\xxi;\cdot)) m_k(d U)$, $ U=(u,\xxi).
$
Decomposing $Q_k(\xxi;\cdot)$ as in \eqref{24}  we get that
$$
S_* (u, Q_k(\xxi; \cdot)) = \int_{F^\perp}S_*\big(u, (\xi_F^\perp, Q_{kF}(\xxi, \xi_F^\perp; \cdot)\big) Q_{kF}^\perp(\xxi; d\xi_F^\perp).
$$
So
\begin{equation}\label{3}
\nu_{k+1}=\int_\XXXX\int_{F^\perp}
S_*\big(u,(\xi_{F}^{\perp}, Q_{kF}(\xxi,\xi_{F}^{\perp};\cdot))\big) Q_{k F}^{\perp}(\xxi; d\xi_{F}^{\perp})m_k(d U)
\end{equation}
By Hypothesis~{\rm \hyperlink{LCR}{(LCR)}}(b),
$$
\pi^k := Q_{kF} (\xxi, y_F^\perp; dy_F)= \rho_{kF}(\xxi, y_F^\perp, y_F) \ell_F(dy_F).
$$
This implies that measure  
 $
 S_*(u,(\xi_{F}^{\perp}, Q_{kF}(\xxi,\xi_{F}^{\perp};\cdot))) =  S_*(u,(\xi_{F}^{\perp}, \pi^k)) $
 has a density against the Lebesgue measure. That is,  it equals
$ g_{k+1}(\xi_F^\perp, U,x)\ell_H (dx)$, where functions $g_r$, $r\in \Z$, are uniformly Lipschitz. Indeed,
using a partition of unity on  a neighbourhood of $   \mathsf P_{F} \KK$ in $F$,  we write the density
$
y_F\mapsto \rho_{kF} (\xxi, y_F^\perp, y_F)
$
as a finite sum of functions, supported by small balls $B_j := B_F(y_{Fj}, r_j)$, where $j=1, \dots, N$ and $y_{Fj} \in  \mathsf P_{F} \KK$.
Thus we decomposed  $\pi^k$ to a sum of $N$ measures $\pi^k_j$, supported by the balls $B_j\subset F$. Next, using
{\rm \hyperlink{LCR}{(LCR)}}(a), in any ball $B_j$ we introduce local coordinates such that in them the map
$
y_F \mapsto S(u, y_F^\perp, y_F)
$
becomes a projection to the first $\dim H$ components of $y_F$. Then an easy calculation shows that the measures
$S_*(u,(\xi_{F}^{\perp},  \pi^k_j ))$ have Lipschitz densities against $\ell_H (dx)$, and the assertion for measure
$ S_*(u,(\xi_{F}^{\perp}, \pi^k)) $   follows.
See \cite[Theorem~4.1]{KS-2026} for a detailed proof.

Thus we derive from \eqref{3} that
$$
(\Pi_X)_*{m}_{k+1}= \nu_{k+1} =
 \Big[\int_\XXXX \int_{F^\perp} {g_{k+1}}({ \xi_F^\perp, U}, x)Q_{F}^{\bot}(\xxi;d\xi_{F}^{\perp}) m_k(dU)\Big]\ell_H(dx),
$$
and similar
$
(\Pi_X)_*\hat{m}_{k+1}=
 \Big[\int\int {g_{k+1}}({ \xi_F^\perp, U}, x)Q_{F}^{\bot}(\xxi;d\xi_{F}^{\perp})\hat m_k(dU)\Big]\ell_H(dx).
$
From here and \eqref{20} it follows that
\begin{equation}\label{d10}
\|(\Pi_X)_*\hat{m}_{k+1}-(\Pi_X)_*m_{k+1}\|_{var}\leq
 C\|Q_{F}^{\bot}(\xxi;d\xi_{F}^{\perp})\hat m_{k}(dU)-Q_{F}^{\bot}(\xxi;d\xi_{F}^{\perp})
 {m}_{k}(dU)\|_{L,\KK_{F}^{\perp}\times \XXXX}^{*}
\end{equation}
(see  Corollary 4.2 in \cite{KS-2026}).
From other hand, for any function $f(\xxi, \xi_F^\perp)$ such that $\| f\|_L \le1$ (see \eqref{LipD}) we have
$
\int Q_F^\perp (\xxi, d\xi_F^\perp) f(\xxi, \xi_F^\perp) = \int Q (\xxi, d\xi) f(\xxi,  \mathsf P_F^{\perp}\xi) := g(\xxi),
$
where by Hypothesis~{\rm \hyperlink{SF}{(SF)}}, $\| g\|_L \le C$. So the integral of $f$ against the measure under the norm sign in the r.h.s. of
\eqref{d10} is bounded by $C_1\|\hat{m}_{k}-m_{k}\|_{L}^{*}$, for any $f$ as above.
  Hence,  the r.h.s. itself is bounded by the same constant. This relation,
jointly with estimates \eqref{d10} and \eqref{23}, establish \eqref{uev} for each $k:=k+1$.

Finally, since convergence \eqref{22} defines finite dimensional distributions of process $\{\hat u_k\}$ uniquely, then the law of $\{\hat u_k\}$
also is uniquely defined.
\end{proof}

\section{Proof of Theorem~\ref{mixingL}}\label{s_4}
Firstly we  establish \eqref{evL}. For that we apply the method of Kantorovich functional
as in \cite{kuksin2006} and \cite[Section~3.1.1]{KS-book}. We do that in three steps:
 in Section \ref{3-1} we construct certain coupling operators.
Next in Section \ref{NewP} we use them to build a new pair of processes,
distributed as the solutions $\{U_k\}$ and $\{U^\prime_k\}$.
Then, to get \eqref{evL}, we iteratively calculate the values of a suitable Kantorovich functional
 on the pairs $(U_k,U_k')$, $k=\tau, \tau+1, \dots$,  and show that they exponentially fast decay to zero.
 This implies \eqref{evL}.
Finally, we use \eqref{evL} to build a solution $\{\hat{U}_k\}_{k\in \Z}$ and prove \eqref{Uev}.

\subsection{Construction of coupling operators~$(\RR_{k},\RR_{k}')_{k\in \Z}$}\label{3-1}

For any $\delta>0$, let us set
\begin{align}\label{D}
D_\delta&:=\{(U,U')\in\XXXX\times\XXXX:\dd_\XXXX(U,U') \le \delta\}.
\end{align}
Let $(\Omega,\FF,\IP)$ be a complete probability space as above.

\begin{proposition}\label{p-constructionRR}
Suppose that Hypothesis~{\rm \hyperlink{LCR}{(LCR)}} is fulfilled.
Then for each $k\in \Z$, there exist measurable maps $\RR_{k},\RR_{k}':\XXXX\times\XXXX\times\Omega\to\XXXX$ and positive numbers~$\theta$, $N$, and $q<1$ satisfying the inequality
\begin{equation}\label{const}
N\theta<1-q,	
\end{equation}
such that $\RR_{k}=\RR_{k}(U,U';\omega)$ and $\RR_{k}'=\RR_{k}'(U,U';\omega)$ form a coupling for the pair of measures $(P_{k}(U;\cdot), P_{k}(U';\cdot))$:
for any $U,U'\in\XXXX$,
\begin{equation}\label{couple}
\DD(\RR_{k})=P_{k}(U;\cdot), \quad
\DD(\RR_{k}')=P_{k}(U';\cdot).
\end{equation}
Moreover, the following properties are fulfilled:
\begin{description}	
	\item [\hypertarget{independence}{Independence}.] For $(U,U')\in  D_\theta^c$, the random variables $\RR_{k}(U,U')$ and $\RR_{k}'(U,U')$ are independent.
	\item [\hypertarget{squeezing}{Squeezing}.] 	
For $(U,U') \in  D_\theta$, we have
\begin{equation}\label{P-RR}
	\IP\bigl\{\dd_\XXXX\bigl((\RR_{k}(U,U'),\RR_{k}'(U,U')\bigr)\le q\,\dd_\XXXX(U,U')\bigr\}
	\ge 1-N\dd_\XXXX(U,U').
\end{equation}
\end{description}
\end{proposition}

\begin{proof}[Proof of Proposition~\ref{p-constructionRR}]
The proof is divided into several steps.
In what follows, we abbreviate $$d:= \dd_\XXXX(U,U').$$

\par
\smallskip
\par
\smallskip
{\it Step~1: Local stabilisation\/}.

We start with a local stabilisation result.

\begin{lemma}\label{LS}
Suppose that Hypothesis~{\rm \hyperlink{LCR}{(LCR)}} is fulfilled. Then there are positive numbers $C_*$, $\delta$, and a continuous map
$$
\varPhi:D_\delta\times \KK\to F
$$
such that, for any $(U,U')\in D_\delta$,
\begin{align}
&\sup_{\xi\in \KK}\|\varPhi(U,U',\xi)\|_E+\Lip_\xi\bigl(\varPhi(U,U',\cdot)\bigr)\le C_*d, \label{Phi-bound}\\
&\sup_{\xi\in \KK}\dd_\XXXX\bigl(\SSS(U,\xi),\SSS(U',\xi+\varPhi(U,U',\xi)\bigr) \le \alpha d, \label{Phi-squeezing}
\end{align}
where $\alpha \in (0,1)$ is as in \eqref{distance-EE}.
\end{lemma}
\begin{proof}
Let us denote
$
\widetilde D_\delta:=\{(v,v')\in X\times X:\|v-v'\|_H\le\delta\}
$
and
$$\Delta(v,v',\xi):=S(v,\xi)-S\bigl(v',\xi+\tilde\varPhi(v,v',\xi)\bigr),$$
where $\tilde\varPhi: \widetilde D_\delta\times \KK\to E$ is a map to be defined.
Using the Taylor formula and the $C^2$-smoothness of~$S$, we write
	\begin{equation}\label{taylor}
	\Delta(v,v',\xi)=(D_vS)(v,\xi)(v-v')-(D_\xi S)(v,\xi)\tilde\varPhi(v,v',\xi)+r(v,v',\xi),
\end{equation}
where the remainder term~$r$ satisfies the inequality
\begin{equation}\label{estimate-r}
	\|r(v,v',\xi)\|_H\le C_1\bigl(\|v'-v\|_H^2+\|\tilde\varPhi(v,v',\xi)\|_V^2\bigr)
\end{equation}
with some constant~$C_1$ not depending on~$v$, $v'$, and~$\xi$.

Since by Assumption~{\rm \hyperlink{LCR}{(LCR)}} the $C^1-$smooth map $D_\xi S(v,\xi) =: A(v,\xi)$
is surjective for all $(v,\xi)$ from a neighbourhood $O$ of $X\times\KK$ in $H\times E$ and  space $H$
is finite-dimensional, then the self-adjoint operator
$
(A A^*)(v,\xi) :H\to H
$
is positive for $(v,\xi)\in O$. Thus for $(x,\xi)$ from a smaller neighbourhood $O_-$ of $X\times\KK$, the operator
$(AA^*)^{-1}(v,\xi)$ exists and is $C^1$-smooth in $(x,\xi)$. So the operator $B(v,\xi) = \big( A^* (AA^*)^{-1}\big)(v,\xi)$
also is a $C^1$-function of $(v,\xi)$, and
\begin{equation}\label{right-inverse}
D_\xi S(v,\xi)B(v,\xi)f=f,
\end{equation}
for any $f\in H$, $(v,\xi)\in O_{-}$.  The operator $B$, constructed above, is the {\it Moore--Penrose right inverse }to
$D_\xi S$. We now set
\begin{equation}\label{Phi-definition}
\tilde\varPhi(v,v',\xi):=B(v,\xi)\,(D_vS)(v,\xi)(v-v').	
\end{equation}
The mapping $\tilde\varPhi:\widetilde D_{\delta}\times \KK\to F, (v,v',\xi)\mapsto\tilde\varPhi$ is $C^1-$smooth (as usual, it means that it extends to a $C^1-$smooth mapping, defined is the vicinity of $\widetilde D_\delta\times \KK$).

We claim that, for small enough $\delta>0$, the map $\tilde\varPhi:\widetilde D_{\delta}\times \KK\to F\subset E$ satisfies the inequalities
 \begin{align}
\sup_{\xi\in \KK}\Bigl(\|\tilde\varPhi(v,v',\xi)\|_E+\Lip_\xi\bigl(\tilde\varPhi(v,v',\xi)\bigr)\Bigr)&\le C'\|v-v'\|_H, \label{tPhi-bound}\\
\sup_{\xi\in \KK}\,\bigl\|S(v,\xi)-S\bigl(v',\xi+\tilde\varPhi(v,v',\xi)\bigr)\bigr\|_H &\le\tfrac12\alpha\|v-v'\|_H, \label{tPhi-squeezing}
\end{align}
where $C'$ is some positive number. Indeed, the validity of~\eqref{tPhi-bound} follows from the compactness of $X\times \KK$ and the $C^1-$smoothness of $\tilde\varPhi$ since $\tilde\varPhi(v,v,\xi)\equiv0$. Furthermore, substituting~$\tilde\varPhi$ into~\eqref{taylor} and using \eqref{right-inverse}, \eqref{estimate-r} and~\eqref{tPhi-bound}, we derive that
$$
\|\Delta(v,v',\xi)\|_H\le C_2\|v-v'\|_H^2\le C_2\delta\|v-v'\|_H.
$$
This leads to~\eqref{tPhi-squeezing} if $\delta$ is small enough.

Now we prove \eqref{Phi-bound} and \eqref{Phi-squeezing}, defining $\varPhi:D_\delta\times \KK\to E$ by the relation
$$
\varPhi(U,U',\xi)=\tilde\varPhi(v,v',\xi), \quad U=(v,\xxi), \quad U'=(v',\xxi').
$$
Recalling the constant $L\ge1$ in the distance~\eqref{metric-X},
we see that~\eqref{Phi-bound} is trivially satisfied with $C_*=C'L^{-1}$. To prove~\eqref{Phi-squeezing}, we use~\eqref{tPhi-bound} and~\eqref{tPhi-squeezing} to write
\begin{align*}
&\dd_\XXXX\bigl(\SSS(U,\xi),\SSS(U',\xi+\varPhi(U,U',\xi)\bigr)\\
=&L\, \|S(v,\xi)-S(v',\xi+\tilde\varPhi(v,v',\xi))\|_H+\alpha\,\dd(\xxi,\xxi')+\|\varPhi(U,U',\xi)\|_E\\
\le &\alpha\,\dd(\xxi,\xxi')+L\,(L^{-1}C'+\tfrac12\alpha)\|v-v'\|_H.
\end{align*}
Up to this point, the number~$L\ge1$ was arbitrary. We now choose it so large that $L^{-1}C'\le \frac{1}{2}\alpha$. Then the above estimate implies~\eqref{Phi-squeezing}.
\end{proof}

\par
\smallskip
{\it Step~2: Estimate for the images of transition probabilities\/}.

\begin{lemma}\label{p-image-measure}
Let Hypothesis~\hyperlink{LCR}{\rm(LCR)} holds, $\delta>0$ and~$\varPhi$ be the constant and map in
Lemma~\ref{LS}, and $$\varPsi(U,U',\eta):=\eta+\varPhi(U,U',\eta).$$ Then there is $\theta\in(0,\delta)$
such that, for any $(U,U') = \big( (v, \xxi), (v', \xxi')\big)
\in D_\theta$,
\begin{equation}\label{TV-estimate}
\bigl\|\PPP_k(U;\cdot)-\varPsi_*(U,U',\PPP_k(U;\cdot))\bigr\|_{\mathrm{var}}
	\le Cd,
\end{equation}
where $C$ stands for a number not depending on~$U$ and~$U'$, and we recall that $d=\dd_\XXXX(U,U')$.
\end{lemma}

\begin{proof}
Recalling \eqref{kp}, we derive from Hypothesis~\hyperlink{LCR}{(LCR)} that for $U =(v, \xxi)\in \XXXX$
 measure $\PPP_k(U;\cdot)$ admits the  disintegration
\begin{equation}\label{4}
\PPP_k (U;\dd y)=\PPP_k (U;\dd y_F,\dd y_{F}^\bot)=Q_{kF}^\bot( \xxi ;\dd y_{F}^\bot) Q_{kF}( \xxi ,y_{F}^\bot;\dd y_F),
\end{equation}
where for any $\xxi \in \bKK$ and $y_{F}^\bot\in F^\bot$ relation \eqref{lcr} holds for measure  $Q_{kF}( \xxi ,y_{F}^\bot;\cdot)$
with some Lipschitz function $\rho_{k F}:\bKK\times E\to\R_+$.

We claim that Proposition 6.1 from the appendix in \cite{KS-2025} applies to the measures $\PPP_k (U;\cdot)$ and
 $\Psi_*(U,U';\PPP_k (U;\cdot))$. Indeed, as in \cite{KS-2025}, the map $\eta\mapsto \Phi(U,U',\eta)$ is ranged in
 a finite-dimensional subspace $F\subset E$, relations (6.1) in \cite{KS-2025} follow from Lemma~\ref{LS},
and Hypothesis~(DLP) of \cite{KS-2025} follows from Hypothesis~\hyperlink{LCR}{(LCR)}. Thus Proposition~6.1
 applies to the two measures and implies \eqref{TV-estimate}.
\end{proof}

\par
\smallskip
{\it Step~3: Construction of operators~$\RR_k$ and~$\RR_k'$}.

Let us fix any $(U,U') = \big( (v, \xxi), (v', \xxi')\big) \in D_\theta$
 and consider two pairs of random variables in $\KK$,
$\bigl(\zeta_{k+1}^{\xxi},~\varPsi(U,U',\zeta_{k+1}^{\xxi})\bigr)$ and $\bigl(\varPsi(U,U',\zeta_{k+1}^{\xxi}),~\zeta_{k+1}^{\xxi'}\bigr)$,
where we recall that $\DD(\zeta_{k+1}^{\xxi})=\PPP_k(U, \cdot)$, $\DD(\zeta_{k+1}^{\xxi'})=\PPP_k(U', \cdot)$.
We denote by~$\lambda(U,U')$ the law of the first pair and define~$\lambda'(U,U')$ to be the law of a maximal coupling for the second pair, i.e., for the pair
of measures
$$
\bigl(\,\DD(\varPsi(U,U',\zeta_{k+1}^{\xxi}))=\varPsi_*(U,U', \PPP_k(U;\cdot)), \ \DD(\zeta_{k+1}^{\xxi'})= \PPP_k(U';\cdot)\,\bigr).
$$
In view of Theorem~1.2.28 in~\cite{KS-book}, we can assume that~$\lambda$ and~$\lambda'$ are transition probabilities from~$D_\theta$ to $E\times E$. Then $\Pi_{2*}\lambda(U,U')=\Pi_{1*}\lambda'(U,U')$ for any $(U,U')\in D_\theta$, where $\Pi_i:E\times E\to E$, $i=1,2$ denotes the projection to the $i^\text{\rm{th}}$ component. In view of a measurable version of the gluing lemma (see~\cite[Chapter~1]{villani2009} and~\cite[Corollary~3.7]{BM-2020}), there exists a triplet of $E$-valued random variables
$\tilde{\zeta}(U,U'),\xi(U,U'),\tilde\zeta'(U,U')$, defined on an independent copy of
 probability space $\Omega$, such that the maps $\tilde\zeta,\xi,\tilde\zeta':D_\theta\times\Omega\to E$ are measurable, and for any $(U,U')\in D_\theta$, we have
\begin{equation*}\label{noise-coupling}
	\DD\bigl(\tilde\zeta(U,U'),\xi(U,U')\bigr)=\lambda(U,U'), \quad
	\DD\bigl(\xi(U,U'),\tilde\zeta'(U,U')\bigr)=\lambda'(U,U').
\end{equation*}
Next, we extend $\tilde\zeta(U,U')$ and $\tilde\zeta'(U,U')$ from $(U,U')\in D_\theta$  to random fields on $\XXXX\times \XXXX$ in such a way that
\begin{equation}\label{exten}
\text{$\tilde\zeta(U,U')$ and $\tilde\zeta'(U,U') $ are independent if $(U,U')\in D_\theta^c$},	
\end{equation}
and their laws are $\PPP_{k}(U; \cdot)$ and $\PPP_{k}(U';\cdot)$, respectively. Finally, for any $U,U'\in\XXXX$, we set
\begin{equation}\label{RR-definition}
\RR_{k}(U,U',\omega):=\SSS\bigl(U,\tilde\zeta(U,U',\omega)\bigr), \quad
\RR_{k}'(U,U',\omega):=\SSS\bigl(U',\tilde\zeta'(U,U',\omega)\bigr).
\end{equation}
Then $(\RR_{k}, \RR_{k}')$ is a coupling for the measures~$P_{k}(U;\cdot)$ and~$P_{k}(U';\cdot)$.

\smallskip
{\it Step~4: Properties of~$\RR_{k}$ and~$\RR_{k}'$\/}.

We claim that the maps defined by~\eqref{RR-definition} satisfy the properties of \hyperlink{independence}{Independence} and \hyperlink{squeezing}{Squeezing}. Indeed, the independence of $\RR_{k}$ and~$\RR_{k}'$  for $(U,U')\in D_\theta^c$ follows from~\eqref{exten}, so we only need to establish~\eqref{P-RR} for $(U,U')\in D_\theta$.

Since~$\DD(\xi,\tilde\zeta')=\lambda'(U,U')$ is the law of a maximal coupling for $\DD(\varPsi(U,U',\zeta^{\xxi}))$
and~$\DD(\zeta^{\xxi'})$, it follows from~\eqref{TV-estimate} and Hypothesis~{\rm \hyperlink{SF}{(SF)}} that
\begin{equation}\label{maximal-coupling}
\begin{split}
&\IP\{\xi(U,U')\ne\tilde\zeta'(U,U')\}
=\|\DD(\varPsi(U,U',\zeta^{\xxi}))-\DD(\zeta^{\xxi'})\|_{var}\\
\leq&\|\DD(\varPsi(U,U',\zeta^{\xxi})-\DD(\zeta^{\xxi})\|_{var}
+\|\DD(\zeta^{\xxi})-\DD(\zeta^{\xxi'})\|_{var}
\le C\,{d}.
\end{split}
\end{equation}

On the other hand, since the law of $(\tilde\zeta,\xi)$ equals to $\lambda$ and coincides with that of $(\zeta^{\xxi},\varPsi(U,U',\zeta^{\xxi}))$, the definition of~$\varPsi$ and inequality~\eqref{Phi-squeezing} imply that, with probability~$1$,
$$
\dd_\XXXX\bigl(\SSS(U,\tilde\zeta(U,U')),\SSS(U',\xi(U,U'))\bigr) \le q\,{d}.
$$
Therefore, if $\omega\in\Omega$ is such that  $\xi(U,U')=\tilde\zeta'(U,U')$, then
$$
\dd_\XXXX\bigl(\SSS(U,\tilde\zeta(U,U')),\SSS(U',\tilde\zeta'(U,U'))\bigr) \le q\,{d}.
$$
Combining this with~\eqref{maximal-coupling} and~\eqref{RR-definition}, we arrive at relation \eqref{P-RR} with $N=C$.
\end{proof}

\subsection{Construction of a coupling for the laws of a pair of solutions of \eqref{RDSL}}\label{NewP}
In Proposition \ref{p-constructionRR}, on probability space $\Omega$, for any $k\in \Z$
 we constructed measurable maps $\RR_{k},\RR_{k}':\XXXX\times\XXXX\times\Omega\to\XXXX$ such that $\RR_{k}=\RR_{k}(V,V';\omega)$ and $\RR_{k}'=\RR_{k}^\prime(V,V';\omega)$ form a coupling for measures
$(P_{k}(V;\cdot), P_{k}(V';\cdot))$, for any $V,V'\in\XXXX$. I.e.,
\begin{equation}\label{couple1}
\DD(\RR_{k})=P_{k}(V;\cdot), \quad
\DD(\RR_{k}')=P_{k}(V';\cdot).
\end{equation}
Next, we denote by~$(\widehat\Omega,\widehat\FF,\widehat\IP,\widehat\FF_k)$ the suitable filtered probability space as in Section~2, constructed from the space $\Omega$  and formed by points $\oomega=(\omega_k,k\in \Z)$ with $\omega_k\in\Omega$, and denote, as before, $\oomega_k=(\omega_l,l\le k)$.
Now, let us take any two $\hat\FF_\tau-$measurable random variables $U,U'\in \XXXX$ and consider corresponding solutions
$\{U_k(U)\}_{k\ge \tau}$ and $\{U_k(U')\}_{k\ge \tau}$ of \eqref{RDSL}. Our goal in this section is to use iterations of operators
 $\RR_{k}$ and $\RR_{k}'$ to construct two special solutions $\{{}_\tau\varPhi_k\}_{k\ge\tau}$ and
 $\{{}_\tau\varPhi'_k\}_{k\ge\tau}$ of \eqref{RDSL} with the initial data $U$ and $U'$, defined on $\widehat\Omega$,
 and examine their properties. We define processes $\{({}_\tau\varPhi_k,{}_\tau\varPhi'_k)\}_{k\ge\tau}$ by the formulas, where $k\ge \tau+1$
\begin{equation}\label{Phi-Phi'}
\left.
\begin{aligned}
{}_\tau\varPhi_\tau(U,U';\oomega_\tau)&=U, \quad {}_\tau\varPhi_\tau'(U,U';\oomega_\tau)=U', \\
{}_\tau\varPhi_k(U,U';\oomega_k)&=\RR_{k}\bigl({}_\tau\varPhi_{k-1}(U,U';\oomega_{k-1}),{}_\tau\varPhi_{k-1}'(U,U';\oomega_{k-1}),\omega_k\bigr),\\
{}_\tau\varPhi_k'(U,U';\oomega_k)&=\RR_{k}'\bigl({}_\tau\varPhi_{k-1}(U,U';\oomega_{k-1}),{}_\tau\varPhi_{k-1}'(U,U';\oomega_{k-1}),\omega_k\bigr).
\end{aligned}
\right\}	
\end{equation}
The construction of $\RR_{k}, \RR_{k}'$ (see \eqref{RR-definition}) and Proposition \ref{p-constructionRR} imply that $\{{}_\tau\varPhi_k(U,U')\}_{k\ge\tau}$ and $\{{}_\tau\varPhi_k'(U,U')\}_{k\ge\tau}$ are solutions of~\eqref{RDSL} with some suitable choices of the random fields $_\tau \zeta_k$ and $_\tau \zeta_k'$, $k\ge \tau$. So
\begin{equation}\label{laws}
\begin{aligned}
\DD\bigl(\{{}_\tau\varPhi_k(U,U')\}_{k\ge\tau}\bigr) &= \DD\bigl(\{ U_k(U)\}_{k\ge\tau}\bigr),\\
\DD\bigl(\{{}_\tau\varPhi_k'(U,U')\}_{k\ge\tau}\bigr) &= \DD\bigl(\{ U_k(U')\}_{k\ge\tau}\bigr).
\end{aligned}
\end{equation}
Besides, it is easy to see that
\begin{equation}\label{Mark}
\begin{gathered}
\text{\sl the family $\{({}_\tau\varPhi_k(U,U'), {}_\tau\varPhi'_k(U,U'))\}_{k\ge\tau}$ is a trajectory of a non-homogenous}\\
\text{\sl Markov  process in~$\XXXX\times \XXXX$ with respect to the filtration~$\{\widehat{\FF}_k\}$}.
\end{gathered}
\end{equation}

The following results establish two key properties of this process.

\begin{proposition}\label{pc1}
Decreasing, if needed, the number $\theta>0$ in Proposition~\ref{p-constructionRR} we achieve
 that, for any $(U,U')\in D_\theta$, the following inequality holds for the constructed trajectory of Markov process~\eqref{Mark}{\rm:}
\begin{equation}\label{infinite-coupling}
	\widehat\IP\bigl\{\dd_\XXXX\bigl({}_\tau\varPhi_{\tau+k}(U,U'),{}_\tau\varPhi_{\tau+k}'(U,U')\bigr)
	\le q^{k}d\mbox{ for all  $k\ge0$}\bigr)\bigr\}
	\ge 1-N_1d,~~d=\dd_\XXXX(U,U').
\end{equation}
Here $N_1=N(1-q)^{-1}$ and $N>0$, $q<1$  are the constants from Proposition \ref{p-constructionRR}.
\end{proposition}
\begin{proof}
Given an integer $k\ge 1$, we define the event
$$
G_k=\Bigl\{\dd_\XXXX\bigl({}_\tau\varPhi_{\tau+k}(U,U'),{}_\tau\varPhi_{\tau+k}'(U,U')\bigr)\le q\,\dd_\XXXX\bigl({}_\tau\varPhi_{\tau+k-1}(U,U'),{}_\tau\varPhi_{\tau+k-1}'(U,U')\bigr)\Bigr\}.
$$
We first show that, for any $(U,U')\in D_\theta$,
\begin{equation}\label{PGk}
	\widehat\IP\bigl(\,\widetilde  G_k\bigr)
	\ge 1-N{d}\sum_{l=0}^{k-1}q^l=:\delta_k, \quad \text{where}\quad
	 \widetilde  G_k = \bigcap_{l=1}^k G_l.
\end{equation}
 Since ${}_\tau\varPhi_{\tau+k}$ and ${}_\tau\varPhi'_{\tau+k}$ are $\widehat{\FF}_{\tau+k}-$measurable, then $G_k\in \widehat{\FF}_{\tau+k}.$

To prove~\eqref{PGk}, we  argue by induction. For $k=1$ the inequality coincides with~\eqref{P-RR}.

Assuming that~\eqref{PGk} holds for $k$, we now establish it for $k+1$. By \eqref{Phi-Phi'},
\eqref{Mark} and  the Markov property, we have
$$
\widehat\IP\bigl\{G_{k+1}\,|\,\widehat\FF_k\bigr\}
=\widehat\IP\Bigl\{\dd_\XXXX\bigl(\RR_{\tau+k}(V,V'),\RR_{\tau+k}'(V,V')\bigr
)\le q\,\dd_\XXXX(V,V')\Bigr\},
$$
where $V(\oomega_{\tau+k})={}_\tau\varPhi_{\tau+k}(U,U')$ and $V'(\oomega_{\tau+k})={}_\tau\varPhi_{\tau+k}'(U,U')$. Now note that if $\oomega_{\tau+k}\in \widetilde G_k$, then
$$\dd_\XXXX({}_\tau\varPhi_{\tau+k}(U,U',\oomega_{\tau+k}),{}_\tau\varPhi_{\tau+k}'(U,U',\oomega_{\tau+k}))\le q^k{d}\le \theta.$$
So according to \eqref{P-RR}, we have
$$\IP^{\omega_{\tau+k+1}}\Bigl\{\dd_\XXXX\bigl(\RR_{\tau+k+1}(V,V',\omega_{\tau+k+1}),\RR_{\tau+k+1}'(V,V',\omega_{\tau+k+1})\bigr
)\le q\,\dd_\XXXX(V,V')\Bigr\}\geq 1-Nq^k{d}.$$
Using the induction hypothesis, we thus obtain
\begin{align*}
\widehat\IP\bigl(\widetilde G_{k+1}\bigr)
&=\IP^{\omega_{\tau+k+1}}\IP^{\oomega_{\tau+k}}\bigl(\widetilde G_{k+1}\bigr)
\ge \bigl(1-Nq^k{d}\bigr)\,\widehat\IP\bigl(\widetilde G_k\bigr)\\
&\ge \bigl(1-Nq^k{d}\bigr)\,\delta_k>\delta_{k+1}.
\end{align*}
This completes the induction step and proves~\eqref{PGk}.

With the help of \eqref{PGk}, the probability of the intersection~$\widetilde G$ of all $G_k$ can be minorised by
$$
1-Nd\sum_{l=0}^\infty q^l=1-N_1 d.
$$
This implies~\eqref{infinite-coupling} since~$\widetilde G$ is a subset of  the event in the left-hand side of~\eqref{infinite-coupling}.
\end{proof}

\begin{proposition}\label{pc2}
For any $\delta>0$, there exist a number $p>0$ and an integer $l\ge1$, both independent from $\tau$, such that for any $(U,U')\in\XXXX\times\XXXX$,
\begin{equation}\label{to-coupling}
	\widehat\IP\bigl\{\dd_\XXXX\bigl({}_\tau\varPhi_{\tau+l}(U,U'),{}_\tau\varPhi_{\tau+l}'(U,U')\bigr)
	\le\delta\bigr\}\ge p.
\end{equation}
\end{proposition}

\begin{proof}
The proof is divided into two steps.

Step~1: We start with an
 assertion concerning  single trajectory $\{U_k(U), k\ge\tau\}$, $U=(u,\xxi)\in X\times \bKK$.
Namely, we will show that for any $\delta>0$, there exists $\rho_{\delta}>0$ and an integer $m_{\delta}\geq1$, both independent from $\tau$,
such that
\begin{equation}\label{transition-U-Udelta}
	\widehat\IP\bigl\{d_\XXXX(U_{\tau+m_\delta}(U),(0,\mathbf0))<\delta\bigr\}\ge p_\delta\quad\mbox{for any initial point $U\in\XXXX$},
\end{equation}
where $\mathbf0$ is the zero-sequence in $E^{\Z_-}$.

Indeed, in view of the definition of distance $d$ in $\bKK$, for $\delta_1=\frac{\delta}{L+1}$ there exists $l_1 \geq 1$ and $r_1>0$ such that
\begin{equation}\label{13}
d(\xxi,\mathbf0)<\delta_1~~{\rm{if}}~~|\xi_{j}|<r_1,~\forall j\geq -l_1.
\end{equation}
By~\hyperlink{GD}{(GD)} and the continuity of $S$, there exist $l_2\geq 1$ and $r_2>0$ such that
\begin{equation}\label{14}
\|S_{l_2}(u;\eta_1,\cdots,\eta_{l_2})\|_{H}<\delta_1~~{\rm{if}}~~|\eta_{j}|<r_2,~\forall 1\leq j\leq l_2,
\end{equation}
for each $u\in X.$ Without loss of generality, we may assume that $l_2\geq l_1$ and $r_2\le r_1$.
Now, we apply~\hyperlink{RZ}{(RZ)} with $l=\tau, n=l_2$ and $\delta=r_2$.
For $s$ as in \eqref{RZ}, consider
the piece of trajectory $\{U_{\tau+j}(U)=(u_{\tau+j},\xxi_{\tau+j}),0\leq j\leq s+n\}$, and let $Q$ be the event
$$Q=\{\hat{\omega} : |\xi_{j}|<r_2~\forall s+1\leq j\leq s+n\}.$$
By~\hyperlink{RZ}{(RZ)}, $\widehat\IP(Q)\geq p_r>0,$ and by \eqref{13}, $d(\xxi_{\tau+s+n},\mathbf0)<\delta_{1}$.
Using \eqref{14}, we derive that $\|u_{\tau+s+n}\|_H<\delta_{1}$
for each $\hat{\omega}\in Q$. Thus,
$$d_\XXXX(U_{\tau+s+n}(U),(0,\mathbf0))<\delta_{1}+L\delta_{1}=\delta,~\forall\hat{\omega}\in Q.$$
This proves \eqref{transition-U-Udelta} with $m_\delta=s+n$ and $p_\delta=p_r$.

\smallskip
{\it Step~2: Proof of \eqref{to-coupling}\/}.
Let us note that all our constructions and conclusions above remain true if we replace the number~$\theta$, defining the set~$D_\theta^c$ in Proposition \ref{p-constructionRR},
 with any smaller positive constant, depending only on the system~\eqref{RDSL}. Thus, without loss of generality, we shall assume in what follows that
\begin{equation}\label{modif}
	\theta\le\frac{1}{2N_1},
\end{equation}
where $N_1>0$ is the number in~\eqref{infinite-coupling}. In this case, the right-hand side of~\eqref{infinite-coupling} is minorised by~$1/2$ if $d=\dd_\XXXX(U,U')\le \theta$.

In what follows, we  assume without loss of generality that~$\delta<\theta$.
Let $k_*>0$ be the smallest integer such that $\theta q^{k_*} \le \delta$. We claim that~\eqref{to-coupling} holds with $l=m_{\delta/2}+k_*$, where $m_{\delta/2}$ is the integer entering relation~\eqref{transition-U-Udelta}, where the radius~$\delta$ of the ball is replaced with~$\delta/2$.

Abbreviating ${}_\tau\Phi_k(U, U') = {}_\tau\Phi_k$, ${}_\tau\Phi'_k(U, U') = {}_\tau\Phi'_k$ and
 $m=m_{\delta/2}$, for $l\ge \tau,$
we set $\zeta_{l}=({}_\tau\varPhi_{l},{}_\tau\varPhi_{l}')\in\XXXX\times\XXXX$, $d(\zeta_{l})=\dd_\XXXX({}_\tau\varPhi_{l},{}_\tau\varPhi_{l}')$ and define the stopping times
$$
\sigma=\sigma(U,U')=\min\bigl\{k\ge\tau:d(\zeta_{k})\le\theta\bigr\}, \quad \sigma^m=\sigma\wedge(\tau+m+1).
$$
Without loss of generality, we may assume that  construction~\eqref{Phi-Phi'} is such that
when $(U,U')\in D_\theta^{c}$, we take random variables $\RR_{k}(U,U',\omega_k)$, $\RR_{k}'(U,U',\omega_k)$, $k\ge\tau$, from some fixed collection of pairs of independent random fields
$$
(\RR_{k}(\cdot,\omega_k), \RR_{k}'(\cdot,\omega_k))_{k\ge \tau}.
$$
Let us now define another processes $\{_\tau\hat\varPhi_{k}, k\ge \tau\}$ and $\{_\tau\hat\varPhi_{k}', k\ge \tau\}$
 by relation~\eqref{Phi-Phi'}, where always the random fields~$\RR_{k}$ and~$\RR_{k}'$ are independent and are taken
 from the collection above.
We set $\hat\zeta_{l}=(_\tau\hat\varPhi_{l},{}_\tau\hat\varPhi_{l}')$ and define $d(\hat\zeta_{l})$, $\hat\sigma$, $\hat\sigma^m$ similarly to $d(\zeta_{l})$, $\sigma$, $\sigma^m$. Note that
$$
\sigma^m=\hat\sigma^m\quad\mbox{and}\quad \zeta_{l}=\hat\zeta_{l}\quad\mbox{for $l\le \sigma^m$}.
$$
In view of~\eqref{transition-U-Udelta} and the independence of processes $(_\tau\hat\varPhi_{l},{}_\tau\hat\varPhi_{l}')$,
\begin{equation}\label{Z1}
	\widehat\IP\{\hat\sigma^m=\tau+m+1\}\le \widehat\IP\{d(\hat\zeta_{\tau+m})>\theta\}
	\le 1-(p_{\theta/2})^2.
\end{equation}
Let us consider event $A_m=\{\dd_\XXXX({}_{\sigma^m}\varPhi_{{\sigma^m}+k}(\zeta_{\tau+\sigma^m}),{}_{\sigma^m}\varPhi_{{\sigma^m}+k}'(\zeta_{\tau+\sigma^m}))\le q^k\theta\mbox{ for all }k\ge1\}$. By the strong Markov property, applied to process~$\{\zeta_l\}$, it follows that
\begin{align}
\widehat\IP\{d(\zeta_{\sigma^m+k})\le q^k\theta\mbox{ for }k\ge 1\}
&=\widehat\IP(A_m)\ge \widehat\IP\bigl(\{\sigma^m\le \tau+m\}\cap A_m\bigr)\notag\\
&=\E\bigl({\mathbf1}_{\{\sigma^m\le \tau+m\}}\,\IP\{A_m\,|\,_\tau\widehat\FF_{\tau+\sigma^m}\}\bigr).
\label{lower-bound-dz}
\end{align}
But for every~$\omega\in \{\sigma^m\le \tau+m\}$, we have $\sigma^m=\sigma$ and $d(\zeta_{\tau+\sigma^m})\le\theta$. Thus for every such~$\omega$ it holds that $\IP\{A_m\,|\,_\tau\widehat\FF_{\tau+\sigma^m}\}\ge\frac12$ in view of~\eqref{infinite-coupling} and~\eqref{modif}. Since $\sigma^m=\hat\sigma^m$, then by~\eqref{Z1} the right-most term in~\eqref{lower-bound-dz} is no less than~$\frac12(p_{\theta/2})^2$. As the event in~\eqref{to-coupling} contains the event~$\{d(\zeta_{\sigma^m+k})\le q^k\theta\mbox{ for }k\ge1\}$ if $l=m+1$, then~\eqref{to-coupling} is proved with $l=m+1$ and $p=\frac12(p_{\theta/2})^2$.
\end{proof}

\subsection{Proof of Theorem~\ref{mixingL}}\label{PML}
It is well known (e.g. see \cite[Section 6.4]{kuksin2006}) that in view of the
compactness of the space $\XXXX$,
the exponential mixing for system \eqref{RDSL} follows from estimates \eqref{evL}. To prove \eqref{evL} we
use the method of Kantorovich functional (see \cite{kuksin2006} and \cite[Section~3.1.1]{KS-book}).
We define
\begin{equation*}
F(U,U')=
\begin{cases}
\dd_\XXXX(U,U'),~~\dd_\XXXX(U,U')\leq\vt,\\
~~R,~~~~~~~~~~\dd_\XXXX(U,U')>\vt,
\end{cases}
\end{equation*}
where $R=2\vt$ and $\vt$ is a small constant specified later.
We see that $F:\XXXX\times\XXXX\to\R_+$ is a measurable symmetric function minorised by~$\dd_\XXXX \wedge \vt$. We define the {\it Kantorovich functional associated with~$F$\/} by the formula
$$
K_F:\PP(\XXXX)\times \PP(\XXXX)\to\R_+, \quad
K_F(\mu,\mu')=\inf_\MMM\int_{\XXXX\times\XXXX}F(V,V')\,\MMM(\dd V,\dd V'),
$$
where the infimum is taken over all  couplings $\MMM $  of $\mu$ and $\mu'$, i.e., over all measures $\MMM \in\PP(\XXXX\times\XXXX)$  whose marginals coincide with~$\mu$ and~$\mu'$. Then for any measures $\mu,\mu'\in\PP(\XXXX)$,
\begin{equation}\label{KKant}
\|\mu-\mu'\|_L^*\le CK_F(\mu,\mu'),
\end{equation}
where $C>0$ depends only on $\vt$ (see \cite[Lemma~5.10]{kuksin2006} and \cite[Lemma~1.2.31]{KS-book}).

Let us take any $\hat\FF_\tau-$measurable random variables $U,U'\in \XXXX,$ distributed as
the initial data $V$ and $V'$,  consider solutions $U_k={}_\tau U_k(U), U_k'={}_\tau U_k(U')$ and the corresponding trajectory of \eqref{Mark}.
By \eqref{Lip_Kant} and \eqref{K_R} one can choose $U$ and $U'$ in such a way that
\begin{equation}\label{KB}
\E F(U,U')\leq C\|\DD(U)-\DD(U')\|_L^*,
\end{equation}
where $C>0$ is a constant. It follows from Proposition \ref{pc2} with $\delta=\vt$ that there exist $p>0$ and $l\in \N$ such that
\begin{equation*}
\IP\bigl\{\dd_\XXXX\bigl({}_\tau\varPhi_{\tau+l}(U,U'),{}_\tau\varPhi_{\tau+l}'(U,U')\bigr)
\le\vt\bigr\}\ge p.
\end{equation*}
For $l$ as above and any $k\ge 0,$ we set
\begin{equation*}
\begin{split}
&\tau_k:=\tau+kl,\\
&\varPhi_{\tau_k}:={}_\tau \varPhi_{\tau_k}(U,U'),\\
&\varPhi_{\tau_k}':={}_\tau \varPhi_{\tau_k}'(U,U').
\end{split}
\end{equation*}
By \eqref{laws}
$\DD(\varPhi_{\tau_k})=\DD(U_k),\DD(\varPhi_{\tau_k}')=\DD(U_k')$. We first prove that there exists some $\kappa\in (0,1)$ such that
\begin{equation}\label{12}
\E F(\varPhi_{\tau_{k+1}},\varPhi_{\tau_{k+1}}')\leq \kappa \E F(\varPhi_{\tau_k},\varPhi_{\tau_k}'),~~\forall k\ge 0.
\end{equation}
Indeed, we recall that $\varPhi_{\tau_{k+1}}$ and $\varPhi_{\tau_{k+1}}'$ depend on the random parameter $\oomega_{\tau_{k+1}}$, which we will write as $\oomega_{\tau_{k+1}}=(\oomega_{\tau_k},\omega_{\tau_k+1},\ldots,\omega_{\tau_{k+1}})=:(\oomega_{\tau_k},\oomega^{l})$. We will denote by $\E^{\oomega_{\tau_k}}$ and $\E^{\oomega^{l}}$ the expectations, respectively, in $\oomega_{\tau_k}$ and in $\oomega^{l}$. To prove \eqref{12}, obviously it suffices to verify \eqref{12} for every fixed $\oomega_{\tau_k}$ and with $\E$ replaced
by $\E^{\oomega^{l}}$. Let $d=\dd_\XXXX(\varPhi_{\tau_{k}},\varPhi_{\tau_{k}}')$ (this is a random variable, depending on $\oomega_{\tau_{k}}$).

If $\oomega_{\tau_k}\in\{d>\vt\}$, we have $F(\varPhi_{\tau_{k}},\varPhi_{\tau_{k}}')=R$. Denoting event $A:=\{\oomega^{l}~|~\dd_\XXXX(\varPhi_{\tau_{k+1}},\varPhi_{\tau_{k+1}}')\leq \vt\}$, we have from Proposition \ref{pc2} that in this case
\begin{equation*}
\begin{split}
\E^{\oomega^{l}} F\bigl(\varPhi_{\tau_{k+1}},\varPhi_{\tau_{k+1}}'\bigr)\leq (1-\IP^{\oomega^{l}}(A))R+\IP^{\oomega^{l}}(A)\vt\leq (1-\frac{p}{2})R.
\end{split}
\end{equation*}
If $\oomega_{\tau_k}\in\{d\leq\vt\}$, we have $F(\varPhi_{\tau_{k}},\varPhi_{\tau_{k}}')=d$.
Denoting event $B:=\{\oomega^{l}~|~\dd_\XXXX(\varPhi_{\tau_{k+1}},\varPhi_{\tau_{k+1}}')\leq qd\}$, we have by Proposition \ref{pc1} that $\IP^{\oomega^{l}}(B)\ge 1-N_1d$. In this case,
\begin{equation*}
\begin{split}
\E^{\oomega^{l}} F\bigl(\varPhi_{\tau_{k+1}},\varPhi_{\tau_{k+1}}'\bigr)
\leq qd\IP^{\oomega^{l}}(B)+(1-\IP^{\oomega^{l}}(B))R\leq (q+N_1R)d.
\end{split}
\end{equation*}
Choosing $\vt>0$ such that $q+N_1R=q+2N_1\vt<1$ and taking $\kappa:=\max\{1-\frac{p}{2},q+N_1R\}$, we achieve \eqref{12}.

It follows from \eqref{12} that $\E F(\varPhi_{\tau_{k}},\varPhi_{\tau_{k}}')\leq \kappa^k \E F(U,U')$.
This and \eqref{KB} imply that
\begin{equation*}
\E F(\varPhi_{\tau_{k}},\varPhi_{\tau_{k}}')\leq C\kappa^k\|\DD(U)-\DD(U')\|_L^*,~~\forall k\ge 0.
\end{equation*}
Thus,
\begin{equation*}
K_F(\DD(U_{\tau_{k}}),\DD(U_{\tau_{k}}')\leq C\kappa^k\|\DD(U)-\DD(U')\|_L^*,~~\forall k\ge 0.
\end{equation*}
So by \eqref{KKant}
\begin{equation}\label{26}
\|\DD(U_{\tau_{k}})-\DD(U_{\tau_{k}}')\|_L^*\leq C\kappa^k\|\DD(U)-\DD(U')\|_L^*,~~\forall k\ge 0.
\end{equation}
This establishes \eqref{evL} if $k=\tau_l$ for some $l$.
For any $k\ge \tau$, we write it as $k=\tau_m+r$, where $m\in \N, 0\le r< l$. Then 2) in Lemma~\ref{l_Lip}, \eqref{26} and \eqref{flow} imply \eqref{evL} for this $k$.

To prove  existence of a Markov process $\{\hat U_k\}_{k\in \Z}$,  satisfying \eqref{Uev} for  each $\tau$ and
 each solution $\{{}_\tau U_k(V)\}$, we first define
 space $\widehat{\XXXX}=\XXXX^\Z=\{\textbf{V}=(V_k\in\XXXX), k\in \Z\}$, provide it with Tikhonov's topology, a corresponding distance, and the Borel $\sigma$-algebra. Then for any $\tau\in \Z$, we consider solution $\{{}_\tau U_k(0,\mathbf{0})=:{}_\tau U_k\}$. It is defined for $k\geq \tau$, and we extend it as ${}_\tau U_l=(0,\mathbf{0})$ for $l< \tau$. Then each process ${}_\tau U_\cdot$ defines a random variable in $\hat{\XXXX}$. For $j\in \Z, s\geq 0$, consider the integer interval $J_j^s=[j,j+1,\cdots,j+s]$, and for any $\tau\in \Z$ denote $\pi_{j}^{s}(\tau)=\DD({}_\tau U_j,\cdots,{}_\tau U_{j+s})$. In view of \eqref{evLs}, for any fixed $j\in \Z$ and $s\geq 0$, the sequence of measures $\{\pi_{j}^{s}(\tau)\}_{\tau\in\Z}$ is a Cauchy sequence in the space $(\PP(\XXXX^s),\|\cdot\|_{L}^*)$, as $\tau\rightarrow -\infty$. Thus,
\begin{equation}\label{17}
\pi_j^s(\tau)\rightharpoonup\pi_j^s~{\rm{as}}~\tau\rightarrow -\infty,
\end{equation}
for some $\pi_j^s\in \PP(\XXXX^{s+1})$. Measures $\pi_{j_1}^{s_1}(\tau)$ and $\pi_{j_2}^{s_2}(\tau)$ with $j_1\leq j_2\leq j_2+s_2\leq j_1+s_1$ agree with respect to the projection $\XXXX^{s_1+1}\rightarrow\XXXX^{s_2+1}$ which
naturally corresponds to the embedding of intervals $J_{j_2}^{s_2}\subset J_{j_1}^{s_1}$. Then by \eqref{17} the measures $\pi_{j_1}^{s_1}$ and $\pi_{j_2}^{s_2}$ also agree. Thus, by Kolmogorov's theorem, they define a Borel measure $\pi\in \PP(\widehat\XXXX)$.
Let $\{\hat U_k\}_{k\in \Z}$ be any random process such that its law equals $\pi$. Convergence \eqref{17} means that
\begin{equation}\label{18}
\DD({}_\tau U_j,\cdots,{}_\tau U_{j+s})\rightharpoonup\DD(\hat{U}_j,\cdots,\hat{U}_{j+s})~{\rm{as}}~\tau\rightarrow -\infty, ~\forall j\in \Z,~\forall s\ge 0.
\end{equation}
From here it easily follows that $\{\hat{U}_{j}\}_{j\in \Z}$ is a Markov process with the same transition probabilities \eqref{TF}. Relation \eqref{Uev} follows from \eqref{evL} and convergences \eqref{18}.

If $\{\bar U_k\}_{k\in \Z}$ is another Markov process in $\XXXX$ with the transition probabilities \eqref{TF},  then
for any $\tau\in \Z$ taking an $\hat\FF_\tau$-measurable random variable $V\in \XXXX$, distributed as
$\DD(\bar U_\tau)$, we find from \eqref{Uev}, where $\DD({}_\tau U_k(V)) = \DD(\bar U_k)$ for $k\ge\tau$, that
$
\| \DD(\hat U_k) - \DD(\bar U_k) \|_L^* \le C e^{-\gamma(k-\tau)}.
$
Sending $\tau$ to $-\infty$ we get that $\DD(\hat U_k) = \DD(\bar U_k) $ for each $k$. So the Markov
process $\{\bar U_k\}$ is  distributed as $\{\hat U_k\}$.
\medskip

Finally, writing ${}_\tau U_j=({}_\tau u_j, {}_\tau \eeta_j), j\in \Z,$ where ${}_\tau \eeta_j=(\cdots,{}_\tau \eta_{j-1},{}_\tau \eta_j)$, and writing $\hat U_j$ as $\hat U_j=(\hat{u}_j, \hat\eeta_j)$, we derive from \eqref{18} that the process $\{\hat\eta_j\}$ is distributed as $\{\eta_j\}$, and obtain from item 3) of Lemma~\ref{l_Lip} with $s=1$  that for each $j\in \Z$,
\begin{equation}\label{19}
\DD({}_\tau U_j,{}_\tau \eta_{j+1})\rightharpoonup\DD( \hat{U}_j, \hat{\eta}_{j+1})~{\rm{as}}~\tau\rightarrow -\infty.
\end{equation}
From \eqref{RDSL} it follows that
\begin{equation*}
\DD({}_\tau U_k)=\SSS_*(\DD({}_\tau U_{k-1},{}_\tau\eta_k)),~\forall \tau,~\forall k.
\end{equation*}
Passing in this relation to a limit as $\tau\rightarrow -\infty$ using \eqref{19}, we get \eqref{15}.

\section{Main result: infinite-dimensional case}\label{s_5}
In this section we discuss systems \eqref{RDSN} with infinite-dimensional  phase-spaces $H$.
We recall that the results  of Sections 1-2 were proved in the general setting,  $\dim H\le\infty$.
So here we only discuss how
the results in  Sections~\ref{s_3},\,\ref{s_4} should be modifies to treat the infinite-dimensional case.
The corresponding changes address only those parts of the proof, where we used that the dimension of  space $H$
is finite. Essentially this is Lemma~\ref{LS} and its proof, who require a modification of Hypothesis \hypertarget{LCR}{(LCR)}.
Our  presentation below is sketchy, but missing details may be extracted from the work \cite{KS-2025}.
As in  Sections~\ref{s_3},\,\ref{s_4}, we assume the validity of Hypotheses~{\rm \hyperlink{SF}{(SF)}, \hyperlink{RZ}{(RZ)}} and~\hyperlink{GD}{\rm(GD)}. For the reader's convenience, we repeat them now:

\begin{itemize}
	\item [\hypertarget{SF}{\bf(SF)}] \sl For any $l\in \Z$, the regular conditional distribution $Q_l(\xxi;\cdot)$ can be chosen to be a Lipschitz-continuous mapping from~$\bKK$ to the space~$(\PP(\KK),\|\cdot\|_{var})$. Moreover, there is $C>0$, such that
\begin{equation*}
	\|Q_l(\xxi;\cdot)-Q_l(\xxi';\cdot)\|_{var}\le C\dd(\xxi,\xxi')\quad
	\mbox{for any $\xxi,\xxi'\in\bKK$~and $l\in \Z$}.
\end{equation*}
\end{itemize}

\begin{description}\sl
	\item [\hypertarget{RZ}{(RZ) Recurrence to zero}.] \sl For any $n\in\N$   and $\delta>0$, there is an integer $s\ge0$ with
	the property that
\begin{equation*}
\inf_{l\in\Z}	\inf_{\xxi\in\bKK}\mathbb{P}\{| {\eta}(k)|<\delta,~k=l+s+1,\ldots,l+s+n \mid \eeta_l=\xxi
\}>0.
	\end{equation*}
\end{description}

\begin{description}\sl
	\item [\hypertarget{GD}{(GD) Global dissipation}.] There is an integer $k\ge1$ and a number $a\in(0,1)$ such that
\begin{equation*}
	\|S_k(u;\mathbf{0}_k)\|_H\le a\,\|u\|_H
	\quad\mbox{for any $u\in H$}.
\end{equation*}
\end{description}

In addition to them, we impose two more hypotheses, which replace the finite-dimensional Hypothesis \hypertarget{LCR}{(LCR)}.

\begin{description}\sl
	\item [\hypertarget{DLP}{(DLP) Decomposability and Lipschitz property}.]
	There is an increasing sequence of finite-dimensional subspaces $F_n\varsubsetneq F_{n+1}\subset E$, with
	complementary subspaces~$F_n^\bot\subset E$ and the corresponding orthogonal projections
	 $\mathsf P_n:E\to F_n$ and $\mathsf P_n^\perp:E\to F_n^\perp$. It is assumed that the following
	  property holds for any space  $F_n$, $n\ge1$: for each $\xxi\in\bKK$ and $l\in \Z$, the measure $Q_l(\xxi;\cdot)$
	  admits a disintegration with respect to the projection $\mathsf P_n^\bot$,
\begin{equation*}
Q_l(\xxi;\dd y)=Q_{l\, n}^\bot(\xxi;\dd y_n^\bot) Q_{l\, n}(\xxi,y_n^\bot;\dd y_n),
\end{equation*}
where $Q_{l\, n}^\bot(\xxi;\cdot)\in\PP(F_n^\bot)$ is the image of~$Q_l(\xxi;\cdot)$ under the projection~$\mathsf P_n^\bot$, and
$Q_{l\, n}(\xxi,y_n^\bot;\cdot)$ is a measure on $F_n$.
Moreover, there are uniformly  Lipschitz  functions $\rho_{l\, n}:\bKK\times E\to\R_+$, $l\in\Z$,
 supported by~$\bKK\times\KK$, such that
\begin{equation*}
Q_{l\, n}(\xxi,y_n^\bot;\dd y_n)=\rho_{l\, n}(\xxi,y_n^\bot,y_n)\ell_n(\dd y_n).
\end{equation*}
\end{description}

Note that in difference with \cite{KS-2025} we do not assume that the space $\cup_n F_n$ is dense in $E$.

To state the last hypothesis we have to assume that system  \eqref{RDSN} has a finite-dimensional determining space
 (see Definition~\ref{d_determining}):
\begin{description}\sl
	\item [\hypertarget{ALC}{(ALC) Approximate linearised controllability}.] There exists an open set $O\supset X\times\KK$ and a finite-dimensional determining subspace $G\subset H$ such that the closure in $H$ of the set $D_\eta S(u,\eta)(F_{\infty})$ contains~$G$ for any $(u,\eta)\in O$, where $F_{\infty}:=\cup_{n=1}^{\infty}F_n$ and spaces $F_n$ are as in Hypothesis~{\rm \hyperlink{DLP}{(DLP)}}
\end{description}

\begin{remark}\label{r_determ}
System \eqref{RDSN} has a finite-dimensional determining subspace if the map~$S$ is smoothing in the sense that for some
Hilbert space~$V$, compactly and densely embedded in~$H$, the map~$S$ is continuously differentiable from $H\times E$ to~$V$. Indeed, in this case~$H$ admits a Hilbert basis $\{e_j, j\ge1\}$ which also is an orthogonal basis of~$V$ such that the norms $\|e_j\|_V$ go to infinity with~$j$ (e.g., see~\cite[Section~2.1]{LM1972}). Then~\eqref{determining} is fulfilled  if we take for~$G$ the vector span of $\{e_j, 1\le j\le N\}$ with a sufficiently large~$N$. \qed
\end{remark}

Now we state the main result of this section.
\begin{theorem}\label{Inf-mixing}
Suppose that Hypotheses~{\rm \hyperlink{SF}{(SF)}, \hyperlink{RZ}{(RZ)}, \hyperlink{GD}{\rm(GD)}, \hyperlink{DLP}{(DLP)}}
and~\hyperlink{ALC}{\rm(ALC)} are fulfilled. Then system \eqref{RDSN} is exponentially mixing in the dual-Lipschitz distance in space $\PP(H)$. Moreover, there exists a process $\{\hat{u}_k\}_{k\in \Z}$ which satisfies \eqref{RDSN} in the sense of distributions:
\begin{equation}\label{Inf-16}
\DD(\hat{u}_k)=S_{*}(\DD(\hat{u}_{k-1},\hat\eta_k)) \quad \forall k\in \Z,
\end{equation}
where $\{\hat\eta_k\}_{k\in \Z}$ is a process in $\KK$, distributed as $\{\eta_k\}_{k\in \Z}$. Process  $\hat u$ is such that
for any initial state $v\in X$ and any $\tau\in \Z$, the corresponding trajectory
$\{u_k={}_\tau u_k(v)\}_{k\geq \tau}$ of~\eqref{RDSN}, \eqref{ICN} satisfies estimates \eqref{22} for all $k\ge\tau+1$
 and all $s$.
 This property defines the distribution of process $\hat u$
in a unique way.
\end{theorem}
As we see, in the infinite-dimensional case, all assertions of Theorem~\ref{mixing} remain true, apart from \eqref{uev}.
The proof of the latter crucially uses that $dim~H<\infty$, and  most likely that result is wrong in general infinite-dimensional case.

\begin{proof}[Sketch of the proof of Theorem~\ref{Inf-mixing}]
As in the finite-dimensional case, the theorem  is proved via lifting system \eqref{RDSN} to the Markov system
\eqref{RDSL} (see Lemma~\ref{l_lift}).
For the latter system, we have

\begin{theorem}\label{Inf_LThm}
Under the hypotheses of Theorem~\ref{Inf-mixing}, all assertions of Theorem~\ref{mixingL} stay true.
\end{theorem}
Derivation of Theorem~\ref{Inf-mixing} from Theorem~\ref{Inf_LThm} is literally the same as derivation of Theorem~\ref{mixing} from Theorem~\ref{mixingL}.

Below we outline a proof of Theorem~\ref{Inf_LThm}, adopting the approach, used to prove Theorem~\ref{mixingL}.

\smallskip
{\it Step~1: Local stabilisation\/}.

\begin{lemma}\label{Inf-LS}
There are positive numbers $C_*$, $\delta$, and $q<1$, a finite-dimensional subspace $F:=F_n \subset E$ as in (DLP),
 and a continuous map \
$
\varPhi:D_\delta\times \KK\to F
$
such that, for any $(U,U')\in D_\delta$,
\begin{align}
	\sup_{\xi\in \KK}\|\varPhi(U,U',\xi)\|_E+\Lip_\xi\bigl(\varPhi(U,U',\cdot)\bigr)&\le C_*\dd_\XXXX(U,U'), \label{Inf-Phi-bound}\\
	\sup_{\xi\in \KK}\dd_\XXXX\bigl(\SSS(U,\xi),\SSS(U',\xi+\varPhi(U,U',\xi)\bigr) &\le q\,\dd_\XXXX(U,U'). \label{Inf-Phi-squeezing}
\end{align}
\end{lemma}

This assertion is very similar to that of Lemma~\ref{LS}, but its proof is more complicated since now  dimension of $H$
is infinite. We start with a result  from linear algebra, crucial for the proof of Lemma~\ref{Inf-LS}.
 It improves a bit a corresponding statement in \cite{KS-2025},
 and its demonstration  follows the argument in  that work.

\begin{lemma}\label{Inf-RI}
Let  $G\subset H$ be a finite-dimensional vector subspace, $O$ be an open subset of a Hilbert space~$\HH$.
Assume that $A:O\to\LL(E,H)$ is a $C^1$-smooth map such that, the closure of $A(y)(F_{\infty})$ contains~$G$
for any $y\in O$, where $F_{\infty}=\cup_{n=1}^{\infty}F_n$ as in Hypothesis~\hyperlink{ALC}{(ALC)}.
 Then, for any $\e>0$ and any compact set $Y\subset O$, there is an integer $n\ge1$ and a $C^1$-smooth map $B_\e:O\to\LL(G,F_n)$ such that
\begin{equation}\label{Inf-ARI}
	\|A(y)B_\e(y)f-f\|_H\le \e\,\|f\|_H
	\quad\mbox{for any $f\in G$, $y\in Y$}.
\end{equation}
\end{lemma}
\begin{proof}
The proof goes in three steps.

1)
Let $Y\subset O$ be a compact set and  $\e$ be a positive number.
 We denote by $\dd_H(f,H')$ the distance  in $H$
  of a vector $f$ from a subspace~$H'$. Since $A(y)(F_{\infty})$ is dense in~$H$, then the continuous functions $(y,f)\mapsto \dd_H\bigl(f,A(y)F_n\bigr)$ pointwise monotonically converge to zero as~$n\to\infty$. This convergence is uniform on compact sets according to the Dini theorem.  Thus there exists $n=n(\eps)$  such that
 \be\label{d1}
 \sup_{y\in Y} \sup_{f \in \AB} d_H\big(f, A(y) F_n) \le \eps.
 \ee

2) For $n\ge1$, $y\in Y$ and $\delta\in(0,1]$
 denote by $A_n(y): F_n \to H$ the restriction of $A(y)$ to $F_n$, and consider the operator
$$
B(y,\delta) = A_n(y)^* \big( A_n(y) A_n(y)^* +\delta I \big)^{-1}: H \to F_n.
$$
Denote $D(y, \delta) = A(y) B(y,\delta)$ and suppose we have proved  that for each $y \in Y$ and $f\in \AB$,
\be\label{d2}
\limsup_{\delta\to0} \| D(y, \delta)f - f\| \le 3\eps.
\ee
Now we begin to prove \eqref{Inf-ARI}.
Let $f_1, \dots, f_N \in G$ be an $\eps$-net for $\AB$. Then for any $(y, \delta) \in Y\times (0,1]$ and  $ f \in \AB$,
$$
\| D(y, \delta) f -f\|_H \le 2\eps +\| D(y,\delta) f_j-f_j\|,
$$
where $\| f-f_j\| \le \eps$, and we used that $\| D(y, \delta) \|_{\LL(H)} \le 1$. By this relation and \eqref{d2} we have  that for
 each $y\in Y$,
$$
\sup_{f \in \AB} \| D(y,\delta) f-f\|_H \le 6\eps \qquad \forall\, 0< \delta\le\delta(y,\eps),
$$
for a suitable $\delta(y,\eps)>0$. From here, by compactness, we can find $\delta_\eps>0$ such that
$$
\sup_{y\in Y} \sup_{f \in \AB} \|D(y, \delta_\eps) f - f\|_H \le 7\eps.
$$
This proves \eqref{Inf-ARI} with $\eps:= 7\eps$ and $B_\eps(y) = B(y, \delta_\eps)$.
So it remains to establish \eqref{d2} for fixed  $y \in Y$ and $f\in \AB$.

3)  By \eqref{d1} for every $f\in \AB$  there is $f_\eps \in A_n(y)F_n$ such that $\| f-f_\eps\| \le \eps$.
Denote by $Q$ the operator $A_n(y) A^*_n(y): H\to H$. Then $Q=Q^*\ge0$, and
\be\label{d9}
D(y, \delta) f_\eps- f_\eps = Q(Q+\delta I)^{-1} f_\eps- f_\eps.
\ee
 Using again that
 $\| D(y, \delta) \|_{\LL(H)} \le 1$ we obtain the inequality
 \be\label{d8}
 \| D(y, \delta) f -f \|_H \le 2\eps + \| D(y, \delta) f_\eps - f_\eps\|_H.
 \ee
 Since $A_n(y)$ is an injection, then  $A^*_n(y)$ is a surjection, and so
 $
 f_\eps \in Q(H) = A_n(y)(F_n).
 $
 From here it follows that the vector in the r.h.s. of \eqref{d9} goes to zero with $\delta$. Thus the r.h.s. of \eqref{d8}
 is bounded by $3\eps$ is $\delta$ is small, which proves \eqref{d2}.
\end{proof}

With this lemma in hands, a proof of Lemma~\ref{Inf-LS} follows that of  Lemma~\ref{LS}.
We define the set
$$
\widetilde D_\delta=\{(v,v')\in X\times X:\|v-v'\|_H\le\delta\}
$$
and suppose that, for a small $\delta>0$, we have constructed a continuous map
$$
\tilde\varPhi:\widetilde D_{\delta}\times \KK\to E, \quad (v,v',\xi)\mapsto\xi',
$$
such that, for any $(v,v')\in \widetilde D_{\delta}$, the map $\tilde\varPhi(v,v',\cdot):\KK\to E$ is continuously differentiable and satisfies the following inequalities (cf.~\eqref{Inf-Phi-bound} and~\eqref{Inf-Phi-squeezing})
 \begin{align}
\sup_{\xi\in \KK}\Bigl(\|\tilde\varPhi(v,v',\xi)\|_E+\Lip_\xi\bigl(\tilde\varPhi(v,v',\xi)\bigr)\Bigr)&\le C'\|v-v'\|_H, \label{Inf-tPhi-bound}\\
\sup_{\xi\in \KK}\,\bigl\|S(v,\xi)-S\bigl(v',\xi+\tilde\varPhi(v,v',\xi)\bigr)\bigr\|_H &\le q'\|v-v'\|_H, \label{Inf-tPhi-squeezing}
\end{align}
where $C'$ and $q'<1$ are some positive numbers. In this case, recalling that $L\ge1$ is the constant in the distance~\eqref{metric-X} and defining $\varPhi:D_\delta\times \KK\to E$ by the relation
$$
\varPhi(U,U',\xi)=\tilde\varPhi(v,v',\xi), \quad U=(v,\xxi), \quad U'=(v',\xxi'),
$$
we see that~\eqref{Inf-Phi-bound} is trivially satisfied with $C_*=C'L^{-1}$. To prove~\eqref{Inf-Phi-squeezing}, we use~\eqref{Inf-tPhi-bound} and~\eqref{Inf-tPhi-squeezing} to write
\begin{align*}
\dd_\XXXX\bigl(\SSS(U,\xi),\SSS(U',\xi+\varPhi(U,U',\xi)\bigr)\le \alpha\,\dd(\xxi,\xxi')+L\,(L^{-1}C'+q')\|v-v'\|_H.
\end{align*}
We now choose $L$ to be so large that $L^{-1}C'+q'<1$. Then the above estimate implies inequality~\eqref{Inf-Phi-squeezing} with $q=\max\{\alpha$, $L^{-1}C'+q'\}<1$. Thus, it remains to construct~$\tilde\varPhi$ satisfying~\eqref{Inf-tPhi-bound} and~\eqref{Inf-tPhi-squeezing}.

To this end, let us denote by $\Delta(v,v',\xi)$ the expression under the norm sign in the left-hand side of~\eqref{Inf-tPhi-squeezing}, where $\tilde\varPhi$ is a map to be defined. Using the Taylor formula and the $C^2$-smoothness of~$S$, we write
	\begin{equation}\label{Inf-taylor}
	\Delta(v,v',\xi)=(D_vS)(v,\xi)(v-v')-(D_\xi S)(v,\xi)\tilde\varPhi(v,v',\xi)+r(v,v',\xi),
\end{equation}
where the remainder term~$r$ satisfies the inequality
\begin{equation}\label{Inf-estimate-r}
	\|r(v,v',\xi)\|_H\le C_1\bigl(\|v-v'\|_H^2+\|\tilde\varPhi(v,v',\xi)\|_E^2\bigr)
\end{equation}
with some constant~$C_1$ not depending on~$v$, $v'$, and~$\xi$.

By Hypothesis~\hyperlink{ALC}{(ALC)}, for any $(v,\xi)\in O\subset H\times E$, the closure of the image of the
 operator $D_\xi S(v,\xi):F_{\infty}\to H$ contains the finite-dimensional subspace~$G$. Applying Lemma~\ref{Inf-RI} to $A(v,\xi)=D_\xi S(v,\xi)$, for any $\e>0$, we find a finite-dimensional space $F_n\subset E$ and a $C^1$-smooth mapping $B_\e:O\to\LL(G,F_n)$ such that
\begin{equation}\label{Inf-right-inverse}
\|D_\xi S(v,\xi)B_\e(v,\xi)f-f\|_H\le\e\,\|f\|_H\quad\mbox{for $(v,\xi)\in X\times \KK$, $f\in G$}.
\end{equation}
We now set $F=F_{n}$ and
\begin{equation}\label{Inf-Phi-definition}
\tilde\varPhi(v,v',\xi)=B_\e(v,\xi)\,{\mathsf P}_G(D_vS)(v,\xi)(v-v'). 	
\end{equation}
The validity of~\eqref{Inf-tPhi-bound} with some constant $C'=C'(\e)$ follows from the compactness of $X\times \KK$ and the continuity of the functions entering~\eqref{Inf-Phi-definition}. Furthermore, substituting~$\tilde\varPhi$ into~\eqref{Inf-taylor} and using~\eqref{determining}, \eqref{Inf-right-inverse}, and~\eqref{Inf-estimate-r}, we derive
$$
\|\Delta(v,v',\xi)\|_H\le (\varkappa+C_2\e)\|v-v'\|_H+C_3\|v-v'\|_H^2,
$$
where we set
$$
C_2=\sup\bigl\{\bigl\|{\mathsf P}_G(D_vS)(v,\xi)\bigr\|_{\LL(H,G)}:(v,\xi)\in X\times \KK\bigr\},\quad C_3=C_1(1+C').
$$
We now choose $\e>0$ and $\delta>0$ so small that $q':=\varkappa+C_2\e+C_3\delta<1$. Then
 inequality~\eqref{Inf-tPhi-squeezing} holds.
This completes the proof of Lemma~\ref{Inf-LS}.
\par
\smallskip
{\it Step~2: Estimate for the image of transition probabilities under $\varPsi$\/}.

\begin{lemma}\label{Inf-p-image-measure}
Let Hypothesis~\hyperlink{DLP}{\rm(DLP)} hold,~$\varPhi$ be the map in Lemma~\ref{Inf-LS} and $\varPsi(U,U',\eta):=\eta+\varPhi(U,U',\eta)$. Then there is $\theta\in(0,\delta)$ such that, for any $(U,U')\in D_\theta$,
\begin{equation*}
\bigl\|\PPP_k(U;\cdot)-\varPsi_*(U,U',\PPP_k(U;\cdot))\bigr\|_{\mathrm{var}}
	\le C_1d,
\end{equation*}
where $C_1$ stands for some positive number not depending on~$U$ and~$U'$.

\end{lemma}

We recall that $F=F_n\subset E$ is the finite-dimensional space as in Lemma~\ref{Inf-LS}, and that
\begin{equation*}
\PPP_k(U;\cdot)=Q_{k}(\xxi;\cdot) =\DD(\zeta_{k+1}^{\xxi}).
\end{equation*}
This relation and Hypothesis~\hyperlink{DLP}{\rm(DLP)} imply that $\PPP_k(U;\cdot)$, $U\in\XXXX,$
 admits a disintegration of the form
\begin{equation*}
\PPP_k (U;\dd y_F,\dd y_{F}^\bot)=Q_{kF}^\bot(\pi_{\bKK}U;\dd y_{F}^\bot)Q_{kF}(\pi_{\bKK}U,y_{F}^\bot;\dd y_F),
\end{equation*}
where
\begin{equation*}
Q_{kF}(\pi_{\bKK}U,y_{F}^\bot;\dd y_F)=\rho_{k F}(\pi_{\bKK}U,y_{F}^\bot,y_F)\,\ell_F(\dd y_F),
\end{equation*}
and functions  $\rho_{r F}:\bKK\times E\to\R_+$, $r\in \Z$, are uniformly  Lipschitz.

The remaining part of the lemma's proof is carried out by repeating that of Lemma~\ref{p-image-measure}.
This completes the proof of Lemma~\ref{Inf-p-image-measure}.

\smallskip
{\it Step~3: Proof of Theorem~\ref{Inf-mixing}}.

The rest of the proof repeats Steps 3,4 in Section \ref{3-1} and the arguments in Section \ref{NewP}, \ref{PML}
since they do not use that
the dimension of $H$ is finite. This allows to show  that there are positive numbers~$C$ and~$\gamma$ such that, for any initial states $U,U^\prime\in \XXXX$, the corresponding trajectories $\{U_k\}_{k\geq \tau},\{U_k^\prime\}_{k\geq \tau}$ of~\eqref{RDSL}, \eqref{ICL} satisfies the inequality
	\begin{equation*}
		\|\DD(U_k)-\DD(U_k^\prime)\|_L^*\le C\,e^{-\gamma (k-\tau)}, \quad k\ge \tau+1.
	\end{equation*}
As before, we use this estimate to construct a Markov process $\{\hat U_k\}$, satisfying \eqref{Uev},
and next use this to prove other assertions of Theorem~\ref{Inf_LThm}.
\end{proof}

\section{Applications} \label{s_6}

Our main results, Theorems~\ref{mixing} and \ref{Inf-mixing},
 naturally apply  to discrete-time random systems \eqref{RDSN}. Below in  Section~\ref{A1} we  discuss the restrictions, imposed in the theorems
 on the
 process $\{\eta_k\}$, and prove that they are rather mild and are met by a large class of processes. Then we show that many continuous time
 equations \eqref{LZ3}  via reduction \eqref{C2}-\eqref{hatS} lead to systems with random processes $\{\eta_k\}$ from the class above.
 For that end in
 Sections~\ref{s_6.1}-\ref{s_6.2}  we construct  two classes of
processes $\eta(t)$, such that the corresponding discrete time processes $\{\eta_k\}$ belong to that class.
 Next in Section~\ref{s_6.3} we use Theorem~\ref{mixing}  to establish
 the mixing for differential equations in finite-dimensional spaces,
 driven by random processes $\eta(t)$ as above. Finally in Section~\ref{s_6.4} we  discuss applications of
 Theorem~\ref{Inf-mixing} to  a non-linear PDE, stirred by random processes as in  Section~\ref{s_6.2}.

 \subsection{Random processes $\{\eta_k\}$, satisfying  hypotheses in Sections 1-4.} \label{A1}

 In Section 1-4 we imposed on  process $\{\eta_k\in E\}$ restrictions~\hyperlink{SF}{\rm(SF)},~\hyperlink{RZ}{\rm(RZ)}, and~\hyperlink{LCR}{\rm(LCR)(b)}.
 Now we will demonstrate   that they
 are satisfied by plenty of random  processes.

 If space $E$ is finite-dimensional, then, as we show in Remark \ref{r_3.2}, assumptions (SF) and
 (LCR)(b) are met if the transition probabilities
 $Q_l$ have the  form \eqref{finiteQ}, where the densities $\rho_l$ are uniformly Lipschitz
 continuous functions on $\bKK\times\KK$.
 If $\rho_l(\xxi, 0) \ge\delta>0$ for all $\xxi$ and $l$, then assumption (RZ) also obviously holds with $s=0$. By Corollary~\ref{cor1} such transition
 probabilities define a unique in distribution random process $\eta$. So in this case the restrictions,
 imposed on the process, hold for bounded  processes $\{\eta_k\}$
 in $E$ which are regular (in the sense that relations \eqref{finiteQ} hold with uniformly Lipschitz densities  $\rho_l$), and
 are non-degenerate (in the sense that   $\rho_l(\xxi, 0) \ge\delta>0$ for all $\xxi$ and $l$).
 \medskip

 Now let $E$ be an infinite-dimensional   Hilbert space with a Hilbert basis $\{e_1, e_2, \dots\}$. For any $n\ge1$ and any $n$ vectors
 $e_{j_1}, \dots, e_{j_n}$ let $F_n$ be their linear span, and $F_n^\perp$ be the  orthogonal complement to $F_n$ in $E$, so
 $E= F_n \oplus F_n^\perp$.  We will write vectors $y\in E$ as $(y_F, y_F^\perp)$, and use the natural projections
 $ P_F: E\to F_n$, $\mathsf P_{F}^{\perp}: E \to F_n^\perp$.

 \subsubsection{ $\dim H<\infty$.} \label{s_6.1.1}

For assumption (LCR) to hold we must have that the map  $D_u S(u,\eta): E\to H$ is surjective for all
 $(u, \eta) \in X\times \KK$. If this property is valid, then for each $(u, \eta) \in X\times \KK$ exists a space
  $F_n$, $n=n(u, \eta)$ as  above,
 such that $D_u S(u,\eta): F_n\to H$ is a surjection. Using the $C^2$-smoothness of $S$ and compactness of $X\times \KK$ we derive from
 here that
 \be\label{must_have}
 D_u S(u, \eta) : F_N \to H \; \; \text{ is a surjection for all $(u, \eta)$ from a neighbourhood $O$ of $X\times \KK$},
 \ee
 for some $N$ which depends only on the system. Below we abbreviate this $F_N$ to $F$.

  We disintegrate measures $Q_l(\xxi; \cdot)$ with respect to the projection $\mathsf P_{F}^{\perp}$:
 \be\label{b1}
 Q_l(\xxi; dy) = Q_{lF}^\perp(\xxi; dy_F^\perp) Q_{lF}(\xxi, y_F^\perp; dy_F),
 \ee
 where $Q_{lF}^\perp(\xxi; \cdot) = (P_F)_* Q_l(\xxi; \cdot)$, and assume that
\smallskip

 a) the measures $Q_{lF}^\perp(\xxi; \cdot) $ satisfy
 $
 \| Q_{lF}^\perp(\xxi_1; \cdot)  - Q_{lF}^\perp(\xxi_2; \cdot) \|_{var} \le C d(\xxi_1, \xxi_2),
 $
 and for each $\delta>0$ we have that
 $
 Q_{lF}^\perp(\xxi; B_{F_N^\perp}(\delta) )\ge \eps_\delta >0$ for all $\xxi\in\bKK$ and all $ l.
 $

 b) For all $\xxi\in \bKK$, $y_F^\perp \in \mathsf P_{F}^{\perp} (\KK)$ and  $l \in \Z$ we have that
 $
 Q_{lF}(\xxi, y_F^\perp; dy_F) =  \rho_l(\xxi, y_F^\perp, y_F) l_F(dy_F),
 $
 where the function $\rho_l$ are uniformly Lipschitz. Moreover,
 $
 \rho_l(\xxi, y_F^\perp, 0) \ge\gamma>0
 $
 for all $\xxi$ and $y_F^\perp$.

\begin{proposition}\label{p_6.1}
If for some space $F_N$ as above the probabilities $Q_l$ decompose as in \eqref{b1} and components of this decomposition satisfy
 a) and b), then the transition probabilities $Q_l$ satisfy Hypotheses~\hyperlink{SF}{\rm(SF)},~\hyperlink{RZ}{\rm(RZ)}, and~\hyperlink{LCR}{\rm(LCR)(b)} holds with $F=F_N$.
\end{proposition}
\begin{proof}
To establish (SF) we take any Borel set $B\subset \KK$. Let $\xxi_1, \xxi_2 \in \bKK$ and $d(\xxi_1, \xxi_2) =:d$. Then, abbreviating
$Q_l$ to $Q$, $Q_{lF}$ to $Q_F$ etc, have that
\[ \begin{split}
Q(\xxi_1; B) - Q(\xxi_2; B)  = \int Q_F^\perp(\xxi_1; dy_F^\perp) \big( Q_F(\xxi_1, y_F^\perp; B) - Q_F(\xxi_2, y_F^\perp; B) \big)\\
- \int\big( Q_F^\perp (\xxi_1; dy_F^\perp) - Q_F^\perp (\xxi_2; dy_F^\perp) \big) Q_F(\xxi_2, y_F^\perp; B)
\end{split}
\]
(we regard a measure $Q_F(\xxi, y_F^\perp; \cdot)$ both as a measure on $F_N$ and a measure on the layer
$\{ z \in \KK \mid \mathsf P_{F}^{\perp} z = y_F^\perp\}$.)  Due to b), in the first integral the integrand is bounded by $ Cd$. Since
$
y_F^\perp \mapsto Q_F(\xxi_2, y_F^\perp; B)
$
is a Borel function, bounded by one, then by a) the second integral also is bounded by $Cd$. This establishes (SF).

Property (RZ) with $s=0$ is a consequence of the second parts of a) and b), while (LCR)(b) is a part of b).
\end{proof}

 By Corollary \ref{cor1} the system of transition probabilities $Q_l$ in the proposition above defines a process $\eta_k$, uniquely in
 distribution. If map $S$ in \eqref{RDSN} satisfies \eqref{must_have}, then (LCR)(b) also holds. So if $S$ meets  \eqref{must_have} with
 some $F_N$, then Hypotheses~\hyperlink{SF}{\rm(SF)},~\hyperlink{RZ}{\rm(RZ)}, and~\hyperlink{LCR}{\rm(LCR)(b)} make a mild restriction on the process $\eta_k$, and there are plenty of processes
 which satisfy them.

 \subsubsection{ $\dim H=\infty$.} \label{s_6.1.2}

 This case is treated in Section~\ref{s_4}, and now we will discuss processes, satisfying Hypotheses~\hyperlink{SF}{\rm(SF)},~\hyperlink{RZ}{\rm(RZ)}, and~\hyperlink{DLP}{\rm(DLP)}, imposed
 there.

Let space $E$ be as above.  For any $l\in \Z$, let $\{\zeta_{l m}\}_{m\ge1}$ be a sequence of
 independent scalar random variables such that $|\zeta_{l m}| \le a_m$ for all $l, m$ and $\om$,
  where $\{a_m\}_{m\ge 1}$ is an $l^2-$sequence, i.e. $\sum a_m^2<\infty$.
 Assume that the laws of these random variables  possess Lipschitz continuous densities~$d_{l m}:\R\to\R_+$
 with respect to Lebesgue measure on~$\R$, such that $d_{lm}(0)>0$ for each $l$ and $m$.
 The random series
$
\zeta_{l}:=\sum_{m=1}^\infty \zeta_{l m} e_m
$
converges  in~$E$ for every $\omega$, and we  denote by $\nu_l\in\PP(E)$ the law of $\zeta_l$. Then for each $l$
supp$\, \nu_l \subset \KK$, where
$
\KK =\{\sum x_m e_m \mid |x_m| \le a_m\}
$
is a compact subset of $E$.

\begin{proposition}\label{P1}
Let the above hypotheses be fulfilled and the set $\bKK=\KK^{\Z_-}$ be given the distance~$\dd$
as in~\eqref{distance-EE}. Let $g_l:\boldsymbol{\KK}\times E\to\R_+$ be a Lipschitz continuous
function such that $g_l(\xxi,y)\ge c>0$ for all $\xxi\in\boldsymbol{\KK}$ and $y\in E$, and let
$m_l(\xxi)=\int_{E} g_l(\xxi,y)\nu_l(\dd y)$. Then the transition probabilities
\begin{equation}\label{mq}
Q_l(\xxi;\dd y):=\rho_l(\xxi,y)\nu_l(\dd y), \quad \rho_l(\xxi,y):=m_l(\xxi)^{-1}g_l(\xxi,y),
\end{equation}
satisfy Hypotheses~\hyperlink{SF}{\rm(SF)}  and \hyperlink{LCR}{\rm(LCR)(b)}
with $F=F_n$, where $n\ge1$ is any, and  $F_n$ is the vector span of some $n$ vectors  $e_{j_1}, \dots, e_{j_n}$.
The density of the
 conditional measure
 $Q_{l\, n}(\xxi,y_n^\bot;\cdot) = Q_{l F_n}(\xxi,y_n^\bot;\cdot)$,
 entering~\eqref{24}, is given by the formula
\begin{equation}\label{cd}
\rho_{ l n}(\xxi,y_n^\bot,y_n)=\frac{g_l(\xxi,y_n,y_n^\bot)D_{l\, n}(y_n)}{\int_{F_n}g_l(\xxi,z,y_n^\bot)\nu_{l\, n}(\dd z)},
\end{equation}
where $\nu_{l\, n}$ is the projection of~$\nu_l$ to~$F_n$, and $D_{l\, n}$ stands for the product of
 functions $d_{l j_1},\dots,d_{l\, j_n}$.
Moreover, Hypothesis~\hyperlink{RZ}{\rm(RZ)} with $s=0$ holds for the family $\{Q_l(\xxi;\cdot),\xxi\in\boldsymbol{\KK}\}$.
\end{proposition}

\begin{proof}
The validity of Hypothesis~\hyperlink{SF}{\rm(SF)} follows immediately from \eqref{mq} and \eqref{20}.
In view of~\eqref{mq}, for any integer $n\ge1$ and arbitrary bounded continuous functions $f:F_n\to\R$ and $g:F_n^\bot \to\R$, we can write
\begin{multline} \label{fg}
	\int_{E}f(y_n)g(y_n^\bot)Q_l(\xxi;\dd y)
	=\int_{F_n^\bot}\biggl\{\int_{F_n}f(y_n)g(y_n^\bot)\rho_l(\xxi,y_n,y_n^\bot)D_{l\, n}(y_n)\ell_n(\dd y_n)\biggr\}\nu_{l\, n}^\bot(\dd y_n^\bot),	
\end{multline}
where $\nu_{l\, n}^\bot$ is the projection of~$\nu_l$ to~$F_n^\bot$, and $\ell_n$ stands for  the Lebesgue measure on~$F_n$.
Relation \eqref{fg} remains true for linear combinations of various functions $f(y_n)g(y_n^\bot),$ so by the Stone-Weierstrass theorem, it holds true for all continuous functions on $\KK.$
On the other hand, the projection of~$Q_l(\xxi;\cdot)$ to~$F_n^\bot$ is given by
$$
Q_{l\, n}^\bot(\xxi;\dd y_n^\bot)=\biggl\{\int_{F_n}\rho_l(\xxi,z,y_n^\bot)\nu_{l\, n}(\dd z)\biggr\}\nu_{l\, n}^\bot(\dd y_n^\bot).
$$
Combining this with~\eqref{fg}, we arrive at relation~\eqref{cd} for the density of the measure~$Q_l(\xxi;y_n^\bot,\cdot)$. The Lipschitz continuity of~$ \rho_{l\, n}$ follows
 from similar property of the function~$g_l$ and the inequality $g_l\ge c$.

 Hypothesis~\hyperlink{RZ}{\rm(RZ)} with $s=0$  follows from the fact that $d_{lm}(0)>0$ for each $l$ and $m$, since the sum
 $\sum a_m^2$ is finite.
\end{proof}

Again, by Corollary \ref{cor1} the system of transition probabilities $Q_l$ in Proposition \ref{P1}  defines a process $\eta_k$, uniquely in
 distribution. This provides us with a large class of processes $\eta_k$
 which satisfy ~\hyperlink{SF}{\rm(SF)},~\hyperlink{RZ}{\rm(RZ)}, as well as ~\hyperlink{LCR}{\rm(LCR)(b)} for  spaces $F_N$ with any $N$. So they
 satisfy \hyperlink{DLP}{\rm(DLP)}.

\subsection{Random processes with $L^{2}_{loc}-$trajectories.}\label{s_6.1}

For a given process $\eta(t)$ with $L^2_{loc}$-trajectories we define a process $\{\eta_k \in\hat E, k\in \Z\}$, $\hat E = L^2(I, E)$, as in Introduction:
$
\eta_k = \eta\!\mid_{I_k}$, $I_k =[k-1, k).
$
Let $\vp_1, \vp_2, \dots$ be some Hilbert basis of $L^2(I)$, and as before, $\{e_1, e_2, \dots\}$ be a Hilbert basis of $E$. Then
$\{ \vp_i \otimes e_j,\, i, j \ge1\}$, is a Hilbert basis of $\hat E$. Replacing in the constructions in Section~\ref{A1}
 basis $\{e_j\}$ with the basis $\{ \vp_i \otimes e_j\}$
we obtain transition probabilities $Q_l$ from $\bKK$ to $\KK$ as in \eqref{mq}, where now $\KK\subset \hat E$ and $\bKK \subset {\boldsymbol {\hat E}}$. Next by applying Corollary~\ref{cor1} we obtain
a process $\{\eta_k \in \hat E, k\in\Z\}$, corresponding to these transition probabilities, and next  define a process $\eta(t) \in E$ by relation
\begin{equation}\label{28}
	\eta(t):=\eta_k(t-k+1)\quad\mbox{for $t\in I_{k}$, $k\in\Z$}.
\end{equation}
Thus we arrive  at the following result:

\begin{proposition}\label{P3}
Any set of transition probabilities $Q_l$ as above defines a process $\eta(t) \in E$ with $L^2_{loc}$-trajectories such that for the corresponding process $\{\eta_k\}$ in
$\hat E$ its conditional distributions  are $Q_l$ (see \eqref{cond}). This
process satisfies Hypotheses~\hyperlink{SF}{\rm(SF)},~\hyperlink{RZ}{\rm(RZ)}, and satisfy \hyperlink{LCR}{\rm(LCR)(b)}
with $F=F_n$, equal to the linear span of any $n$ vectors $\varphi_l\otimes e_m$, for arbitrary $n\ge 1$.
\end{proposition}

So there are many   process $\eta(t) \in E$ with $L^2_{loc}$-trajectories such that the corresponding processes $\{\eta_k\in \hat E\}$
  satisfies Hypotheses~\hyperlink{SF}{\rm(SF)},~\hyperlink{RZ}{\rm(RZ)}, and~\hyperlink{LCR}{\rm(LCR)(b)} with
  $\dim F =n$,  where  $n\ge 1$ is any.

\subsection{Random processes with continuous trajectories.}\label{s_6.2}
Consider  Sobolev space $H^{1}(0,1)=:H^1$, provided with the norm $\|\cdot\|_{1}$,
where $\|u\|_{1}^{2}=\|u\|^{2}+\|u_x\|^{2}$ and $\|\cdot\|$ is the $L^{2}-$norm.
Let us denote $H^1_0(0,1)=:H^1_0$.  This is a subspace of $H^1$ of codimension two. Consider two real functions on $I$, 
$\phi_1(t)=1-t$ and $\phi_2(t)= t$. Then
$$
\phi_1(0)=\phi_2(1) = 1, \qquad \phi_1(1)=\phi_2(0) = 0.
$$
So if $\xi(t) \in H^1$, then the function $\xi - {\phi}_2 \xi(0) - {\phi}_1 \xi(1)$ is an $H^1$-function which vanishes at $t=0$ and $t=1$.
Thus
\be\label{b2}
\xi = {\phi}_2 \xi(0) + {\phi}_1 \xi(1)+ \xi^0, \qquad \xi^0 \in H^1_0.
\ee

Let $Y$ be a Hilbert space of finite or infinite dimension with a Hilbert basis $\{e_j\}_{j\ge 1}$. Consider the spaces
$$
\HH^1 = H^1(I,Y) = H^1 \otimes Y, \qquad \HH^1_0= H_0^1(I,Y) = H^1_0 \otimes Y.
$$
From \eqref{b2} it follows that for any $\eta(t) \in\HH^1$,
$
\eta  = {\phi}_2 \eta(0) + {\phi}_1 \eta(1)+ \eta^0$, where  $\eta^0 \in \HH^1_0.
$
This relation decomposes space $\HH^1$ to a (non-orthogonal) direct sum
\be\label{decomp}
\HH^1 = H^1\otimes Y = \big( {\phi}_2\otimes Y\big)\oplus \big({\phi}_1\otimes Y\big)\oplus \HH^1_0.
\ee
So if $\eta(t) \in Y$ is a random process with $H^1_{loc}$-trajectories and $\eta_k = \eta\!\mid_{I_k} \in \HH^1$, $I_k= [k-1, k)$, and
$\nu_k= \eta(k) \in Y$, then
\be\label{b4}
\eta_k = \nu_{k-1} \phi_1 + \nu_k \phi_2 + \eta_k^0, \qquad \eta_k^0 \in \HH_0^1.
\ee
Other way round, let $\{\nu_k\in Y\}_{k\in\Z}$ and $\{\eta_k^0\in \HH^1_0\}_{k\in\Z}$ be given process. If we define $\eta_k\in \HH^1$
by relation \eqref{b4}, then the process $\eta(t)$, defined as
\begin{equation}\label{b5}
\eta(t):=\eta_{k}(t-k+1),~t\in I_k,~\forall k\in \mathbb{Z},
\end{equation}
is continuous and piecewise $H^1$-smooth. So its trajectories are $H^1_{loc}$-curves. Finally, let us set
\be\label{b6}
\hat\eta_k =  \nu_k \phi_2 + \eta_k^0 
 \in \big( {\phi}_2\otimes Y\big) \oplus \HH_0^1 =\{\eta(x) \in \HH^1\mid \eta(0)=0 \}=:E_0,
\qquad k \in \Z
\ee
(random variables $\hat\eta_k$ are modified increments of process $\eta(t)$ on segment $I_k$).

Let $\{\vp_m\}_{ m\ge1}$ be a Hilbert basis of $H^1_0$, proportional to the sine-basis. Then
$$
\{{\phi}_2 \otimes e_j\}_{{ j\ge1}} \cup \{\vp_i \otimes e_j\}_{i, j\ge1} =: \{\xi_j\}_{j\ge1},
$$
is a  (non-orthogonal) basis of $E_0$.\footnote{This basis is not Hilbert, but it makes no difference for what follows since it is ``as good as
a Hilbert basis  of $E$", as it become one  if we change the  norm in $E$ to an equivalent one.}
 We impose the following restriction on an $H^1_{loc}$-process $\eta(t)\in Y$:

\begin{description}\sl
\item [\hypertarget{H}{(H)}]
Random processes  $\hat\eta_k\in E_0$ as in \eqref{b6}
 is such that $\hat\eta_k^\om \in \KK$, where $\KK$ is a compact  subset of $E_0$. Moreover, it
satisfies Hypotheses~\hyperlink{SF}{\rm(SF)},~\hyperlink{RZ}{\rm(RZ)},  \hyperlink{LCR}{\rm(LCR)(b)} with respect to the basis
$\{\xi_j\}$ holds for any $n$.
\end{description}

We can obtain plenty of processes $\eta(t)$ which satisfy this hypothesis, using relation \eqref{b6}, where we use processes
$\{\nu_k \in Y\}$ and $\{ \eta_k^0 \in \HH_0^1\}$ as in Proposition~\ref{p_6.1} if $\dim H<\infty$, or as in  Proposition~\ref{P1} if
$\dim H= \infty$.

\subsection{Applications to randomly forced ordinary differential equations}\label{s_6.3}

Consider an ODE in $H:=\R^n$, stirred by an $Y-$valued random process $\eta (t)$, where $Y$ is a finite dimensional space,
\begin{equation}\label{LZ3}
\dot x(t)=f(x(t),\eta (t)), \quad x(t) \in H,
\end{equation}
supplemented with the initial condition
\begin{equation}\label{IC-LZ3}
	x(0)=x_{0},
\end{equation}
where $f:H\times Y\rightarrow H$ is a $C^2-$smooth map. A solution $x(t)$ of system \eqref{LZ3}, \eqref{IC-LZ3} is unique, if
it exists, and we denote by $S^t(\cdot,\eta):H\rightarrow H$, the map which takes an initial condition $x_0$ to the value
at time $t$ of the solution  (if such a solution exists).

\subsubsection{Equations, driven by
  random forces  with $L_{loc}^2-$trajectories.}
Suppose first  that  $\eta (t)$ in \eqref{LZ3} is a random process
with $L_{loc}^2-$trajectories, constructed in Proposition \ref{P3}, and let $\eta _k \in L^2(I, Y), \ k\in \Z$, be
 a restriction of $\eta $ to interval $I_k:=[k-1,k).$ Then, assuming that
equation \eqref{LZ3} with  such a process $\eta $ is well posed,  and denoting
 $u_k=x(k), 0\le k\in\Z,$ we rewrite equation for $\{u_k\}_{k\ge 0}$ as a discrete time random system
in $X:$
\begin{equation}\label{LZ8}
u_k=S(u_{k-1},\eta_k ),~k\ge 1,~~u_0=x,
\end{equation}
where $S(u,\eta)=S^1(u,\eta)$. This is a system \eqref{RDSN}, where $E=L^2(I,Y)$.
It turns out that Theorem~\ref{mixing} is well applicable to it if function $f$ in \eqref{LZ3} is affine in $\eta $.
See \cite[Section 5]{KS-2025}, where such  application is given for the case when processes $\eta (t)$ is statistically
1-periodic.

In  next subsection we show that our results apply to prove the mixing for general equations \eqref{LZ3} if there
 $\eta (t)$ are continuous in time processes as in Section~\ref{s_6.2}.

\subsubsection{Equations,  driven by   random forces   with $H^1_{loc}$-trajectories.}\label{s_6.3.2}
 Let  random force $\eta $ in
equation  \eqref{LZ3} be a process with continuous trajectories, satisfying Hypothesis~\hyperlink{H}{\rm(H)} in Section~6.2.
We set
$$E:=E_0\otimes Y,$$
where space $E_0$ is defined in \eqref{b6}. We recall that for each $k$
\begin{equation}\label{Z4}
\eta_k (t)=\nu_{k-1} \phi_1(t)+\nu_{k} \phi_2(t)+\eta_k^{0}(t)\in H^1(I,Y)=\HH^1, \quad t\in I.
\end{equation}
 By Hypothesis~\hyperlink{H}{\rm(H)},  for all $k$ and $\omega$  we have that
\begin{equation}\label{Z6}
\nu_k\in \bar{B}_{Y}(r_0) =: B, \quad  \eta_k \in  \bar B_{\mathcal{H}^1}(r)=: \mathcal{B},
\quad \eta_k^{0}\in \KK^0\subset H_0^1(I,Y)=: \HH^1_0,
\end{equation}
for some $r_0, r>0$ and some compact set $\KK^0\subset \HH^1_0$.

 We  impose the following regularity assumption on the equation:
\begin{description}\sl
\item [\hypertarget{K}{(K)}]
If in \eqref{LZ3} $\eta \in H^1_{loc}(\R;Y)$ and
 $\|\eta _{|I_k} \|_{\HH^1}\le r$ for each $k$, then for every  $x_0\in H$,
a solution $x(t)$ for \eqref{LZ3}, \eqref{IC-LZ3} exists for all $t\ge 0$. Moreover, there exists an  $R>0$ such that
\par
a) if $x_0\in \bar{B}_{R}(H)=:X$, then $x(t)\in X$ for all $t\ge 0$;
\par
b) for each $x_0\in H$, there exists $T=T(\|x_0\|_H)\geq 0$ such that $x(t)\in X$ for $t\ge T$;
\par
c) for any ball $B$ in $H$, it holds that
\begin{equation}\label{converge0}
	\sup_{x_{0}\in B}\|S^t(x_{0},0)\bigr\|_H\to0\quad\mbox{as} \quad t\to+\infty.
\end{equation}
\end{description}

\begin{remark}\label{rem1}
Hypothesis~\hyperlink{K}{\rm(K)} holds if equation
 \eqref{LZ3} admits a suitable Lyapunov function. Namely, if there exists a $C^1-$function $V$ on $H$ such that
\par
i) $V(x)\rightarrow+\infty$ as $\|x\|_H \rightarrow+\infty$.
\par
ii) There exist constants $V_0, \delta>0$ such that
\begin{equation*}
(\nabla V(x),f(x,\eta))\le -\delta \quad \text{ if $\eta\in \mathcal{B}$ and $V(x)\geq V_0$},
\end{equation*}
where $(\cdot,\cdot)$ is the inner product in $H$. 
\par
iii) If $V(x)\leq V_0,$ we have $(x,f(x,0))\leq -\mu \|x\|_H^2,\;  \mu>0$.

Indeed, let us set $Q:=\{x\in H: V(x)\leq V_0\}$. By i) this set is bounded. Since for $x\in H\backslash Q$, we have
\begin{equation}\label{21}
\frac{d}{dt}V(x(t))=(\nabla V(x),f(x,\eta))\le -\delta,
\end{equation}
then starting from $x_0\in Q$, trajectory $x(t)$ will never exit $Q$, so a) holds. If $x_0\in H\backslash Q$,
then $V(x_0) >V_0$, and \eqref{21} holds while $x(t)$ stays outside $Q$. Thus $x(t)$ will enter $Q$
by the time $T={(\tilde{C}(\|x_0\|)-V_0)}/{\delta}$, and will stay there afterwards. So b) also holds.
If $\eta\equiv0$, then, firstly, by b) trajectory $x(t)$ will enter  $Q$. There by iii) its norm satisfies
$
(d/dt) \| x\|_H^2 \le -2\mu \|x\|_H^2.
$
So c) in Hypothesis~(K) aslo is valid.
\smallskip

There are many equations  \eqref{LZ3} with Lyapunov functions, satisfying i)-iii).
For example, if $f(x,\eta)=-\gamma x+f_1(x,\eta)$,
$\gamma>0$ and $(x,f_1(x,\eta))\leq \frac{1}{2}\gamma\|x\|_H^2+D(\|\eta\|_{E})$, where $D$ is a continuous function that
vanishes at zero, then i)-iii) hold with $V(x)=\frac{1}{2}\|x\|_H^2.$ \qed
\end{remark}
\medskip

Since $H^1(0,1)$ is a Banach algebra, then the mapping
\[
 {H}^1(I,H)  \times  \mathcal{H}^1\rightarrow  L^2(I,H) \times H, \quad
 (x(t),\eta(t))\rightarrow \big( \dot{x}(t)-f(x(t),\eta(t)), \,x(0) \big)
\]
is $C^2-$smooth.
 Its differential in the first argument, evaluated at  $(x(\cdot), \eta(\cdot))$, is the mapping
$$
y(t) \mapsto \big(\dot y(t) -D_x f(x(t), \eta(t)) y(t), \, y(0)\big) \qquad t \in I.
$$
The latter corresponds to the initial-value problem for a linear ODE with continuous coefficients  in space $H$. So it defines
an isomorphism
$
{H}^1(I,H)  \xrightarrow{\sim}  L^2(I,H) \times H,
$
and  by the implicit function theorem a solution $x \in H^1(I,H)$ of \eqref{LZ3},  \eqref{IC-LZ3} with  $x_0\in H$ and
$\eta \in \ \mathcal{H}^1$, if it exists, is a $C^2$-smooth function of $x_0$ and $\eta $.  Jointly with Hypothesis~\hyperlink{K}{\rm(K)}, this implies that the map
\be\label{IFT}
X\times \mathcal{B}\rightarrow H^1(I,H), \quad (x_0, \eta(\cdot)) \mapsto x(\cdot),
\ee
where $x(\cdot)$ is a solution of  \eqref{LZ3},  \eqref{IC-LZ3}, is $C^2$-smooth (i.e. it extends to a $C^2$-smooth
mapping, defined in the vicinity of $X\times \mathcal{B}$ in $H\times \mathcal{H}^1$). So the   mapping
$
S=S^1: X\times \mathcal{B} \rightarrow  X$, which sends $ (x_0,\eta (\cdot)) $ to $ x(1),$
also
 is well defined and $C^2-$smooth.  Its differential $D_\eta S(x_{0}, \eta)$ is the linear mapping
\begin{equation}\label{LZ7}
D_\eta S(x_{0}, \eta): \mathcal{H}^1\rightarrow H,~\xi\mapsto z(1),
\end{equation}
where $z(t)$ is a solution of the linear equation
\begin{equation}\label{LLZ3}
\dot z=D_xf(x(t),\eta(t))z+D_\eta f(x(t),\eta(t))\xi(t),~~z(0)=0.
\end{equation}
The mapping $S$ defines a discrete time random system for variables $u_k= x(k) \in X:$
\begin{equation}\label{CZ8}
u_k=S(u_{k-1},\eta_k ),~~\eta_k  \in\HH^1,\qquad
k\ge 1,~~u_0=x.
\end{equation}

Let us assume that equation \eqref{LZ3} satisfies the following linearised  controllability assumption:

\begin{description}\sl
	\item [\hypertarget{LC}{\bf(LC) Linearised controllability}.]
For $x_0\in X$, $\eta \in\KK$
the linear mapping $D_\eta S(x_{0}, \eta)$ in \eqref{LZ7} extends to a continuous surjective mapping \
$L^2(I, Y) \rightarrow H$, which continuously depends on $(x_0, \eta)$.
\end{description}

Hypothesis~\hyperlink{LC}{\rm(LC)} is equivalent to the fact that for any $x_0\in X$ and  $\eta \in\KK$
the associated linear system~\eqref{LLZ3} is controllable in~$H$ at time $t=1$ with controls from $L^2(I, Y)$.  It
  is not restrictive:

\begin{proposition}\label{e_control}
If Hypothesis~\hyperlink{K}{\rm(K)} holds and, for each $x \in H$ and $\zeta \in Y$, the mapping $D_\eta f(x,\zeta) \in \mathcal{L}(Y; H)$ is surjective, then Hypothesis~\hyperlink{LC}{\rm(LC)} is satisfied.
\end{proposition}

\begin{proof}
Since spaces $H$ and $Y$ are finite-dimensional and the operator $D_\eta f(x,\zeta) : Y \to H$ is surjective for every $x \in H$ and
 $\zeta \in Y$, it admits the Moore--Penrose  right inverse operator
\[
R(x,\zeta) := (D_\eta f(x,\zeta))^* \bigl[ D_\eta f(x,\zeta) (D_\eta f(x,\zeta))^* \bigr]^{-1}
\]
(see \eqref{right-inverse}).
 As the map $f: H \times Y \to H$ is $C^2$-smooth, then the  operator  $R(x,\zeta)$ depends continuously on $(x,\zeta)$.

Let $x_0 \in X$ and $\eta   \in \mathcal{K}$ in \eqref{LZ3},    \eqref{IC-LZ3}
be given and let $x(t)$ be a solution of \eqref{LZ3}, \eqref{IC-LZ3}.
 For any $z_1 \in H$, choose a smooth
curve $\tilde{z} \in C^\infty(I, H)$ such that $\tilde{z}(0)=0$ and $\tilde{z}(1)=z_1$. Define
\[
\tilde{\xi}(t) := R(x(t),\eta(t)) \bigl( \dot{\tilde{z}} - D_x f(x(t),\eta(t)) \tilde{z} \bigr).
\]
Thus defined curve $\tilde{\xi}(t)$ is continuous in $t\in I$, so it belongs to  $L^2(I,Y)$.
Moreover,  since $\tilde{z}$ satisfies \eqref{LLZ3} with $\xi = \tilde{\xi}$ and $\tilde{z}(1) = z_1$, then
$D_\eta S(x_0, \eta)(\tilde\xi) = z_1$. Thus,  the linear mapping $D_\eta S(x_0, \eta) : L^2(I,Y) \to H$ is surjective.

In equation \eqref{LLZ3} let us denote
$
D_\eta f(x(t),\eta(t))\xi(t) = : \zeta(t).
$
Since in that equation the coefficient $D_x f(x,\eta)$  is a continuous in time matrix-function that
continuously depends on $(x_0,\eta) \in X \times \KK$, then the mapping
$$
L^2(I,H) \to H^1(I,H), \quad \zeta(\cdot) \mapsto z(\cdot) ,
$$
is a continuous linear operator that continuously depends on $(x_0,\eta)$.

As $\zeta \in L^2(I,H)$ continuously depends on $(x_0,\eta)$, then the solution $z\in H^1(I,H)$ continuously depends
on $(x_0,\eta)$. So $z(1) \in H$ also does. This completes the verification of Hypothesis~\hyperlink{LC}{\rm(LC)}.
\end{proof}

By this result many equations as in Remark \ref{rem1}, apart from Hypothesis~\hyperlink{K}{\rm(K)} also meet  Hypothesis~\hyperlink{LC}{\rm(LC)}.
We now state a result on the mixing for equation \eqref{LZ3} with continuous random forces.
\begin{theorem}\label{App-C}
Assume that in \eqref{LZ3} $\eta (t)\in Y$ is a continuous random process satisfies
Hypotheses~\hyperlink{H}{\rm(H)}, and equation \eqref{LZ3} satisfies Hypotheses~\hyperlink{K}{\rm(K)} and \hyperlink{LC}{\rm(LC)}. Then system \eqref{CZ8} is exponentially mixing in the total variation distance.
\end{theorem}
\begin{proof}[Proof of Theorem \ref{App-C}]
According to \eqref{Z4} and Hypothesis~\hyperlink{K}{\rm(K)}, we can rewrite   system \eqref{CZ8} as
\begin{equation}\label{C1}
u_k=S(u_{k-1},\nu_{k-1}\phi_1+\nu_{k}\phi_2+\eta_k^{0}),~k\ge 1,~~u_0=x,
\end{equation}
where $u_k\in X=\bar{B}_R(H)$ for all $k$. Let us denote
\[ \begin{split}
&\hat X:= X \times B \subset H \times Y =: \hat H, \quad \hat H =\{w=(u,\nu)\},\\
&\hat\KK:= B\times \KK^0\subset Y \times \HH^1_0 =: \hat E, \quad \hat E =\{ \eta =(\nu, \eta^0)\}
\end{split}
\]
(see \eqref{Z6}), and set
$$
w_k =(u_k, \nu_{k}) \in \hat X, \qquad  \bar\eta_k = (\nu_k, \eta_k^0) \in \hat\KK.
$$
We recall that $\nu_k =\eta_k(1) = \eta(k)$ and that $\eta_k^0$ is the  projection of $\eta_k\in \HH^1$ to the space
$\HH_0^1$, corresponding to the decomposition \eqref{decomp}.
Note that the mapping
$
(\nu, \eta^0(t)) \mapsto \nu\phi_1(t) +\eta^0(t)
$
identifies space $\hat E$ with space $E_0$ (see \eqref{b6}), and identifies $\bar\eta_k$ with $\hat\eta_k$, which is a
modified increment of process $\eta$ on segment $I_k$.

Due to \eqref{C1}, system \eqref{CZ8} which we examine implies  the following random dynamics for variable $w_k$
in space $\hat X$:
\begin{equation}\label{C2}
w_k=\hat{S}(w_{k-1},\eta_{k}),~k\ge 1,~~w_0=(x,\nu_{0}),
\end{equation}
where
\begin{equation} \label{hatS}
\hat{S}:\hat{X}\times \hat\KK \rightarrow \hat{X},~~(w, \bar\eta)=
\big((u,\mu ),(\nu , {\eta^0})\big)\mapsto ({S}(u,\mu  \phi_1+ {\nu }\phi_2+ {\eta^0}), {\nu }).
\end{equation}
Clearly the map  $\hat{S}$ extends  to a $C^2-$smooth mapping
$\hat{S}: \hat H \times \hat E \to \hat H$, and we see immediately  that if
 $\{u_k\}_{k\geq0}$ is a solution of \eqref{C1}, then $\{w_k=(u_k,\nu_{k})\}_{k\geq0}$ solves \eqref{C2}.

We wish to  apply Theorem~\ref{mixing} to system \eqref{C2} to
 establish Theorem~\ref{App-C}. For that it suffices to verify for \eqref{C2}
  the hypotheses of Theorem~\ref{mixing}.  Indeed, by Hypothesis~{\rm \hyperlink{H}{(H)}}
 process $\{\eta_k\}$ satisfies Hypotheses~{\rm \hyperlink{SF}{(SF)}, \hyperlink{LCR}{(LCR)}}(b) and{\rm ~\hyperlink{RZ}{(RZ)}}.  Then, according
 to  item c) of Hypotheses~{\rm \hyperlink{K}{(K)}},
 Hypothesis~{\rm \hyperlink{GD}{(GD)} holds for $\hat{S}$.

 So it remains to verify that map  $\hat S$ \  satisfies
  Hypothesis~\hyperlink{LCR}{(LCR)}}(a). From \eqref{hatS} we see that if $w=(u, \mu ), \bar\eta=(\nu , \eta^0)$ and
  $\tilde\eta=(\tilde\nu, \tilde\eta^0)$, then
  \begin{equation} \label{dS}
D_\eta\hat{S}(w, \bar\eta) \tilde\eta = \big( D_\eta S(u, \mu  \phi_1+ {\nu }\phi_2+ {\eta^0})(\tilde\nu \phi_2 +\tilde\eta^0), \tilde\nu\big).
\end{equation}
Functions $\{\varphi_j, j\ge1\}$ form orthogonal bases  both of spaces $L_2(I)$ and  $H^1_0(I)$. For any $N$ we set
$\FF_N =\,$span$\{\varphi_j, 1\le j\le N\} \subset H^1_0(I) \subset H^1(I)$. Writing
$
\HH^1 =  H^1(I) \otimes Y
$
and
$
\HH^1_0 =  H^1_0(I) \otimes Y,
$
we identify  $ \FF_N\otimes Y$ with a finite-dimensional subspace $F_n\subset \HH^1_0\subset \HH^1$.

Also, we denote $\KK^+ = \{ \mu \phi_1 +\nu  \phi_2 +\eta^0\in\HH^1 \mid \mu , \nu  \in B, \eta^0\in \KK\}$.
This is a compact subset of $\HH^1$.

\begin{lemma}\label{lem1}
There exists an open neighbourhood $Q$ of $X\times \KK^+$ in $H\times \HH^1$ and $N\ge1$ such that for every
$(x,\eta)\in Q$
\be\label{surj}
\text {
the mapping
$\ D_\eta S(x, \eta) :  F_N \to H\ $ is surjective.
}
\ee
\end{lemma}
\begin{proof}
By {\rm \hyperlink{LC}{(LC)}}, for any $(x,\eta)\in X \times \KK^+$ the mapping
\be\label{L2}
D_\eta S(x, \eta): L^2(I, Y) \to H
\ee
is a continuous surjection. Since dim$\,H<\infty$ and the system $\{\varphi_j, j\ge1\}$ is a basis of $L^2(I)$ (as well as of $H_0^1$),
 then there
exists an $N(x,\eta)$ for which \eqref{surj} holds. Since  map \eqref{L2} is continuous in $(x,\eta)$, then the function
 $N(x,\eta)$ is lower semi-continuous, and \eqref{surj} holds in the vicinity of $(x, \eta)$ with a suitable
 $N=N'(x,\eta)$. As the set $X\times \KK^+$ is compact, then the assertion follows.
\end{proof}

For $N$ as in Lemma \ref{lem1}, let us set
$
F =Y\times F_N \subset Y\times \HH^1_0= \hat E.
$
Now the form \eqref{dS} of $D_\eta \hat S$ and the lemma imply the validity of Hypothesis~(LCR)(a).
This proves  Theorem \ref{App-C}.
\end{proof}

Evoking property b) in Hypotheses~\hyperlink{K}{\rm(K)}, we get
\begin{corollary}\label{Cor-C}
If $\{u_k\}_{k\ge 0}$ and $\{u'_k\}_{k\ge 0}$ are solution of equation \eqref{CZ8} such that $\|u_0\|_H, \|u'_0\|_H\le \bar{R}$ for some $\bar{R}>0,$ then convergence \eqref{ev} still holds with some $C=C(\bar{R})$ and same $\gamma$.
\end{corollary}

Now let us return from system \eqref{LZ8} to the original equation \eqref{LZ3}. Let $x(t)$ and $x'(t)$ be
its solutions  with $x(0)=x_0$ and $x'(0)=x'_0$, respectively. Let
$\|x_0\|_H, \|x'_0\|_H \le \bar R$ for some $\bar R>0$.

\begin{corollary}\label{t_cont_time}
Under the assumptions of Theorem \ref{App-C},
\begin{equation}\label{test}
\|\DD(x(t))-\DD(x^\prime(t))\|_L^* \le C'(\bar{R})\,e^{-\gamma t}, \quad t\ge 0,
\end{equation}
for  some constant $C'(\bar{R})>0$.
\end{corollary}

\begin{proof}
Using (b) in Hypotheses~\hyperlink{K}{\rm(K)} we find $T\ge0$ such that
$
x(T), x'(T) \in X = \bar B_R(H),
$
for all $\om$. Then, shifting time by $T$, we may assume that the initial data $x_0, x_0'$ are random variables in
$X$, measurable with respect to the sigma-algebra $\pi_0( \{\eta_k \})$. We set
$$
u_k=x(k),\;\; u'_k= x'(k), \;\; w_k =(u_k, \nu_k), \;\; w'_k =(u'_k,  \nu_k), \quad k \in\Z_+.
$$
Processes $\{w_k\}$ and  $\{w'_k\}$ are solutions of \eqref{C2} with $\eta_k=(\nu_k, \eta_k^0)$ and
$w_0= (x_0, \nu_0), \ w'_0= (x'_0, \nu_0)$.
Consider the lifting of system \eqref{C2} to a corresponding Markovian system \eqref{RDSL}, where
$$
U_k = \big(w_k,  (\nnu_k, \eeta^0_k)\big), \qquad U'_k = \big(w'_k,  (\nnu_k, \eeta^0_k)\big)
$$
(we use Lemma \ref{l_lift} with $\tau=0$, where we re-denoted\,  ${}_\tau \nu_k$ to $\nu_k$ and ${}_\tau \eta^0_k$
to $ \eta^0_k$). By  Theorem~\ref{mixingL},
\be\label{v1}
\| \DD(U_l) - \DD(U'_l)\|_L^* \le C \|x_0-x'_0\|_H e^{-\gamma l}, \qquad l\in \Z_+.
\ee

For any $k\ge1$ let us denote $\breve{x}_k= x(\cdot)\!\mid\!\!_{I_k}$, and recall that
$\eta _k = \eta (\cdot)\!\mid\!\!_{I_k}$. Then $\breve{x}_k\in H^1(I,Y)$ is a $C^2$-smooth function of
$u_{k-1}\in H$ and $\eta_k \in \HH^1$ (see the discussing of mapping \eqref{IFT}).
Identifying $\eta_k $ with $(\nu_{k-1}, \nu_k, \eta_k^0)$ (see \eqref{Z4}) we get that
\be\label{xk}
\breve{x}_k\; \text{ is a  $C^2$-function of } \; \big(w_{k-1}, (\nu_k, \eta_k^0)\big).
\ee

Below we provide spaces of measures with the dual-Lipschitz distances. As process $\eta $ meets
Hypotheses~{\rm \hyperlink{SF}{(SF)}} by assumption {\rm \hyperlink{H}{(H)}}, then
$\DD(\nu_k, \eta_k^0)$ is a Lipschitz function of $\eeta_{k-1}$. From here and \eqref{xk} we get that
\be\label{xk1}
\DD(\breve{x}_k ) \, \text{ is a Lipschitz function of  } \; \DD(U_{k-1}), \quad k\in\N,
\ee
with a $k$-independent Lipschitz  constant (cf. proof of item 2) of Lemma~\ref{l_Lip}). For $t$ in \eqref{test} let
$k$ be the smallest integer bigger than $t$. Then $t\in I_k$, and \eqref{test} follows from
\eqref{v1} with $l= k-1$  and  \eqref{xk1}, where the constant $C'$ depends on $\bar R$ due to the time-shift
 which we made at the beginning of the proof.
\end{proof}

\subsection{Application to a randomly forced non-linear PDE}\label{s_6.4}
In this section, we study the mixing problem for the randomly forced 3D primitive equations of atmospheric mechanics. These
equations can be written as a system of the form \eqref{a2} in a suitable infinite-dimensional space $H$, with $\eta(t)$ entering
the mapping $F$ additively. Exponential mixing for the equations was examined in \cite{B22, BGN23}, where the random force
$\eta$ was assumed to be a bounded discontinuous random Haar series whose restrictions to different intervals $I_j$ are
independent. Due to the latter property, the discrete-time system \eqref{a3} to which the equations may be reduced, defines
 a Markov process in $H$. In \cite{B22}, exponential mixing was established for nondegenerate random forces, for which all
 Fourier modes in $x$ of the force are excited, and in \cite{BGN23}, for degenerate forces, when
  only a few modes are excited.
 Our presentation in this section uses some constructions and lemmas from \cite{B22}.

Denote by  $\mathbb{O} $ the torus  $\mathbb{T}_L^2 \times \mathbb{T}_h$, where
$\mathbb{T}_L^2 := (\mathbb{R}/L\mathbb{Z})^2$ and $\mathbb{T}_h := \mathbb{R}/h\mathbb{Z}$.
 The randomly
 forced 3D primitive equations in  $\mathbb{O}$ is the following system:
\begin{equation}\label{PE1}
\partial_t u-\Delta u+(u\cdot \nabla_2)u-(\int_{-h}^z {\rm{div}}_2 u(\cdot,\cdot,\xi)d\xi)\partial_z u+\nabla_2 p=\eta.
\end{equation}
Here  $u = u(x,y,z,t) = (u_1, u_2)(x,y,z,t)$ for $(x,y,z) \in \mathbb{O}$, $t\ge0$,
and  $\operatorname{div}_2 u = \partial_x u_1 + \partial_y u_2$.  See Introduction in \cite{B22} and references therein.
We supplement the equations with periodic boundary conditions, odd in $z$.
 The former means that  $(x,y,z) \in \mathbb{O}$, while the latter --  that
\begin{equation}\label{PE2-2}
\begin{split}
u(x,y,-z,t)=u(x,y,z,t),
\end{split}
\end{equation}
and add an initial condition
\begin{equation}\label{PE-IC}
u(\cdot,0)=u_0.
\end{equation}

We assume that the initial data $u_0$ and the random force $\eta$ have zero meanvalues:
$$\int_{\mathbb{O}}u_0(\zeta)d\zeta=0,~~\int_{\mathbb{O}}\eta(\zeta,t)d\zeta=0,~~\forall t\ge0.$$
This implies that the solution to system \eqref{PE1},  \eqref{PE2-2}, \eqref{PE-IC} also satisfies
$$\int_{\mathbb{O}}u(\zeta,t)d\zeta=0,~~\forall t\ge0.$$
Defining
$ 
b(u,v):=(u\cdot \nabla_2)v-(\int_{-h}^z {\rm{div}}_2 u(\cdot,\cdot,\xi)d\xi)\partial_z v,
$ 
we rewrite \eqref{PE1} in a more compact form
\begin{equation}\label{PE5}
\partial_t u-\Delta u+b(u,u)+\nabla_2 p=\eta(t).
\end{equation}

For any $m\in \N$, we denote by $H^{m}(\mathbb{O}; \R^2)$ the Sobolev space of degree $m$ on $\mathbb{O}$,
 denote
$$H^m:=\{u\in H^{m}(\mathbb{O};\R^2) :\int_{\mathbb{O}}u(\zeta)d\zeta=0,~u~{\rm{satisfies\; \eqref{PE2-2}}}\},$$
and provide $H^{m}$  with the homogenous scalar product
$\langle u,v\rangle_m:=\langle (\nabla)^{m}u,(\nabla)^{m}v\rangle.$
Next we set
\begin{equation*}
\begin{split}
H^m\oplus \mathbb{R} {\bf1}:=\{v+c~|~v\in H^m,c\in \mathbb{R}\},\qquad
V^m:=\{u\in H^{m}: \int_{-h}^{0}{\rm{div}}_2 u(x,y,z)\,dz\equiv 0\},
\end{split}
\end{equation*}
and  denote by~$\Pi$ the orthogonal projection in space $H^m\oplus \mathbb{R}{\bf1}$ onto~$V^m$. Then
$\Pi(\nabla_2 p)=0$ and $\Pi \Delta u=\Delta u$.
Finally, by
applying $\Pi$ to  equation  \eqref{PE5},
we rewrite it as a  nonlocal PDE,
\begin{equation}\label{PE-reduced}
	\p_tu-\Delta u+B(u)=\eta^H(t), \qquad B(u)=\Pi b(u,u), \quad \eta^H(t) =\Pi  \eta (t),
\end{equation}
  supplemented with  initial condition~\eqref{PE-IC}.

Study of equations \eqref{PE1}=\eqref{PE-reduced},
 where $\eta$ is a  white in time random force,  presents serious difficulties: their well-posedness  is established
only in some restricted sense, and the mixing property for them is unknown; again see  Introduction to \cite{B22} for discussion
and references.  Our goal is to consider those  equations, driven by bounded continuous random forces as in
Section~\ref{s_6.2}. Then the well-posedness follows from the deterministic theory of equations \eqref{PE1},
and by applying Theorem~\ref{Inf-mixing} to
the equations \eqref{PE1}=\eqref{PE-reduced} in a manner similar to the application of Theorem~\ref{mixing} to equation \eqref{LZ3} in Section~\ref{s_6.3.2}, we will show that now  the equations are mixing.
For that end we fix any integer $m\ge2$ and
denote
$$H := V^m,\;\; \mathcal{H}^1:=H^1(I;H),\;\;  \mathcal{H}_0^1:=H_0^1(I;H).$$
 We suppose that  $\eta^H \in H^1(\mathbb{R}^1_{loc};H)$ in \eqref{PE-reduced} is a random process,
  constructed in Section~\ref{s_6.2} and meeting
Hypothesis~\hyperlink{H}{\rm(H)} with $Y=H$.  As before, for $k\in\Z$,
 $\eta^H_k $ stands for  a restriction of $\eta^H$ to the interval $I_k$. Thus
we obtain a random  sequence $\{\eta^H_k\}_{k\in \Z}$ in space  $\mathcal{H}^1$.

 It is  known that
 for any $u_0\in H$ and for  each $\omega$,  equation~\eqref{PE-reduced}, \eqref{PE-IC}  has
 a unique solution $u(t;u_0)$, belonging to the space $C(\R_+,H)\cap L_{\mathrm{loc}}^2(\R_+,V^{m+1})$
  (e.g., see  \cite[Theorem~1]{B22}).
  Let  $\{u_k(u_0)  :=u(k;u_0)\}_{k\in\mathbb{Z}_+}$ be
  restriction of this  solution to integer times, and let
$S: H \times \mathcal{H}^1 \to H$ denote the time-$1$ shift along trajectories of equation~\eqref{PE-reduced}. Then
	\begin{equation}\label{PED}
		u_k=S(u_{k-1},\eta^H_k), \quad k\ge1.
	\end{equation}
By \cite[Theorem 3]{B22} the map $S$ is $C^2$-smooth (in fact, it is analytic).
Let  $\KK\subset \mathcal{H}^1$ be a
compact set that contains the supports of distributions of random variables $\eta_k^H$. By
\cite[Theorem 4]{B22}  there is a compact set $X\subset H$ such that
$$
S: X \times \KK \rightarrow X.
$$
The arguments in the first three lines of Page 156 in \cite{B22} show that the invariant set  $X$  can be chosen
to be absorbing: for any $u_0 \in H$ there is $N=N(\|u_0\|_H)$ with the property
that for any
$\eta^H_1, \dots \eta^H_N \in \KK$, a solution of \eqref{PED} is such that $u_N\in X$.

Finally, let $B\subset H$ be a closed ball  that contains the set $\KK$.
The following theorem describes the asymptotic behaviour of~$u_k$ as $k\to\infty$.

\begin{theorem}\label{PE-C}
Assume that $\eta^H(t)\in H$ is a continuous random process, satisfying
Hypothesis~\hyperlink{H}{\rm(H)} with $Y=H$. Then system  \eqref{PED} is exponentially mixing in the dual-Lipschitz distance.
\end{theorem}
\begin{proof}
We apply the strategy, used to establish  Theorem \ref{App-C}.  For that, as in Section~\ref{s_6.3.2}, we
 rewrite system \eqref{PED} as
\begin{equation}\label{PEC1}
u_k=S(u_{k-1},\nu_{k-1}\phi_1+\nu_{k}\phi_2+\eta_k^{0}),~~k\ge 1,
\end{equation}
where $u_k\in X$.  As in the proof of  Theorem \ref{App-C},  let us denote
\[ \begin{split}
&\hat X:= X \times B \subset H \times H=: \hat H, \quad \hat H =\{w=(u,\nu)\},\\
&\hat\KK:= B\times \KK\subset H\times \HH^1_0 =: \hat E, \quad \hat E =\{ \eta =(\nu, \eta^0)\},
\end{split}
\]
and set
$
w_k:=(u_k, \nu_{k}) \in \hat X$, $\bar\eta_k:= (\nu_k, \eta_k^0) \in \hat\KK.
$
According to \eqref{PEC1},   random process $\{w_k\}$ satisfies
\begin{equation}\label{PEC2}
w_k=\hat{S}(w_{k-1}, \bar\eta_{k}),~k\ge 1,~~w_0=(u_0,\nu_{0}),
\end{equation}
where
\begin{equation} \label{PEhatS}
\hat{S}:\hat{X}\times \hat\KK \rightarrow \hat{X},~~ (w,\bar \eta)=
\big((u,\mu ),(\nu , {\eta^0})\big)\mapsto ({S}(u,\mu  \phi_1+ {\nu }\phi_2+ {\eta^0}), {\nu }).
\end{equation}
Clearly the map  $\hat{S}$ extends  to a $C^2-$smooth mapping
$\hat{S}: \hat H \times \hat E \to \hat H$, and we see immediately  that if
 $\{u_k\}_{k\geq0}$ is a solution of \eqref{PED}, then $\{w_k=(u_k,\nu_{k})\}_{k\geq0}$ solves \eqref{PEC2}.

Now we  apply Theorem~\ref{Inf-mixing} to  system \eqref{PEC2} to establish Theorem~\ref{PE-C}.
Firstly,  according to Hypothesis~{\rm \hyperlink{H}{(H)}}, the random process $\{\bar\eta_k\}$ satisfies Hypotheses~{\rm \hyperlink{SF}{(SF)}, \hyperlink{DLP}{(DLP)}} and{\rm ~\hyperlink{RZ}{(RZ)}}.  Then, it follows from \cite[Proposition 2]{B22}
that Hypothesis~{\rm \hyperlink{GD}{(GD)} holds for $\hat{S}$.
So it remains to verify that map  $\hat S$ \  satisfies Hypothesis~\hyperlink{ALC}{(ALC)}}. Indeed,  by \eqref{PEhatS} we have  that if $w=(u, \mu ), \bar\eta=(\nu , \eta^0)$ and
$\tilde\eta=(\tilde\nu, \tilde\eta^0)$, then
\begin{equation*}
D_\eta\hat{S}(w, \bar\eta) \tilde\eta = \big( D_\eta S(u, \mu  \phi_1+ {\nu }\phi_2+ {\eta^0})(\tilde\nu \phi_2 +\tilde\eta^0), \tilde\nu\big).
\end{equation*}
Let $\{\xi_j\}_{j\ge1}$ be the basis of $E_0$ as in Section \ref{s_6.2}, and $F_n = $span$\,(\xi_1, \dots, \xi_n)$. Then by
\cite[Proposition~3]{B22}   space
$D_\eta S(u_0,\eta)\big( \cup_{n=1}^{\infty}F_n\big) $ is dense in $H$. Thus
 Hypothesis \hyperlink{ALC}{\rm(ALC)} holds.
This completes the proof of the theorem.
\end{proof}

By applying the argument in the proof of Corollary \ref{t_cont_time},  we get:

\begin{corollary}\label{PEt_cont_time}
Let $u(t)$ and $u'(t)$ be solutions of equation \eqref{PE1} with initial data $u(0)=u_0$ and $u'(0)=u'_0$, respectively. Suppose that $\|u_0\|_H, \|u'_0\|_H \le \bar R$ for some $\bar R>0$. Then, under the assumptions of Theorem~\ref{PE-C}, we have that
$
\|\mathcal{D}(u(t))-\mathcal{D}(u'(t))\|^*_{L} \le C'(\bar R)\, e^{-\gamma t}$ for all $t \ge 0,
$
for some constant $C'(\bar R)>0$.
\end{corollary}

\par
\noindent \footnotesize
\addcontentsline{toc}{section}{Bibliography}
\newcommand{\etalchar}[1]{$^{#1}$}
\def\cprime{$'$} \def\cprime{$'$}
  \def\polhk#1{\setbox0=\hbox{#1}{\ooalign{\hidewidth
  \lower1.5ex\hbox{`}\hidewidth\crcr\unhbox0}}}
  \def\polhk#1{\setbox0=\hbox{#1}{\ooalign{\hidewidth
  \lower1.5ex\hbox{`}\hidewidth\crcr\unhbox0}}}
  \def\polhk#1{\setbox0=\hbox{#1}{\ooalign{\hidewidth
  \lower1.5ex\hbox{`}\hidewidth\crcr\unhbox0}}} \def\cprime{$'$}
  \def\polhk#1{\setbox0=\hbox{#1}{\ooalign{\hidewidth
  \lower1.5ex\hbox{`}\hidewidth\crcr\unhbox0}}} \def\cprime{$'$}
  \def\cprime{$'$} \def\cprime{$'$} \def\cprime{$'$}
\providecommand{\bysame}{\leavevmode\hbox to3em{\hrulefill}\thinspace}
\providecommand{\MR}{\relax\ifhmode\unskip\space\fi MR }
\providecommand{\MRhref}[2]{%
  \href{http://www.ams.org/mathscinet-getitem?mr=#1}{#2}
}
\providecommand{\href}[2]{#2}

\end{document}